\documentclass[11pt, reqno,tikz,border=8pt]{amsart}

\usepackage{tikz}

\usetikzlibrary{calc}

\usepackage{xcolor}
\usetikzlibrary{
    positioning,
    arrows.meta,
    shapes.geometric,
    calc
}

\usepackage{amsmath, amsthm, amsfonts, amssymb, color}
\usepackage{mathrsfs}
\usepackage{amstext,amsxtra}
\usepackage{hyperref}
\hypersetup{colorlinks=true,linkcolor=black, urlcolor=blue, citecolor=blue}

\allowdisplaybreaks

\newcommand{\EQ}[1]{\begin{equation}\begin{split} #1 \end{split}\end{equation}}
\newcommand{\EQN}[1]{\begin{equation*}\begin{split} #1 \end{split}\end{equation*}}

\newcommand{\CAS}[1]{\begin{cases} #1 \end{cases}}

\usepackage[margin=3cm, a4paper]{geometry}

\newcommand{\norm}[1]{\left\|#1\right\|}

\newcommand{\B}{\mathcal{B}}

\newcommand{\LL}{\mathcal{L}}
\newcommand{\M}{\mathcal{M}}
\newcommand{\cS}{\mathcal{S}}
\newcommand{\cH}{\mathcal{H}}

\newcommand{\T}{\mathcal{T}}

\newcommand{\C}{\mathbb{C}}
\newcommand{\N}{\mathbb{N}}
\newcommand{\R}{\mathbb{R}}
\newcommand{\Z}{\mathbb{Z}}

\newcommand{\A}{\mathcal{A}}

\newcommand{\al}{\alpha}
\newcommand{\be}{\beta}
\newcommand{\ga}{\gamma}
\newcommand{\de}{\delta}
\newcommand{\e}{\varepsilon}
\newcommand{\fy}{\varphi}
\newcommand{\om}{\omega}
\newcommand{\la}{\lambda}
\newcommand{\te}{\theta}
\newcommand{\s}{\sigma}
\newcommand{\ta}{\tau}
\newcommand{\ka}{\kappa}
\newcommand{\y}{\eta}
\newcommand{\z}{\zeta}
\newcommand{\ro}{\rho}
\newcommand{\up}{\upsilon}

\newcommand{\bb}{\vec b}
\newcommand{\uu}{\vec u}
\newcommand{\vv}{\vec v}

\newcommand{\tf}{\tfrac}

\newcommand{\De}{\Delta}
\newcommand{\Om}{\Omega}

\newcommand{\p}{\partial}
\newcommand{\na}{\nabla}
\newcommand{\re}{\mathop{\mathrm{Re}}}
\newcommand{\im}{\mathop{\mathrm{Im}}}

\newcommand{\sgn}{\operatorname{sgn}}
\newcommand{\Cu}{\bigcup}

\newcommand{\lec}{\lesssim}
\newcommand{\les}{\lesssim}
\newcommand{\gec}{\gtrsim}
\newcommand{\ges}{\gtrsim}
\newcommand{\etc}{,\ldots,}

\newcommand{\I}{\infty}

\newcommand{\ti}{\widetilde}
\newcommand{\ba}{\overline}
\newcommand{\U}{\underline}
\newcommand{\LR}[1]{{\langle #1 \rangle}}

\newcommand{\brk}[1]{{\left( #1 \right)}}

\newcommand{\fr}{\frac}

\newcommand{\pt}{&}
\newcommand{\pr}{\\ &}
\newcommand{\pq}{\quad}
\newcommand{\pn}{}
\newcommand{\prq}{\\ &\quad}
\newcommand{\prQ}{\\ &\qquad}
\newcommand{\prQQ}{\\ &\qquad\qquad}

\newcommand{\g}{\mathfrak{g}}
\newcommand{\HH}{\mathcal{H}}
\newcommand{\cQ}{\mathcal{Q}}
\newcommand{\sign}{\operatorname{sign}}
\newcommand{\rad}{{\operatorname{rad}}}

\newcommand{\mat}[1]{\begin{pmatrix} #1 \end{pmatrix}}
\newcommand{\smat}[1]{\left[\begin{smallmatrix} #1 \end{smallmatrix}\right]}

\newcommand{\lp}[1]{{#1}}
\newcommand{\pe}[1]{#1^{>J}}
\newcommand{\np}[1]{\mathfrak{#1}}

\newcommand{\fS}{\mathfrak{S}}

\newcommand{\diff}[1]{{\triangleleft #1}}

\renewcommand{\vec}[1]{\mathbf{#1}}

\newcommand{\nw}{\newcommand}
\nw{\aD}{\al D}
\nw{\cL}{\mathcal{L}}
\nw{\bS}{\mathbb{S}}
\nw{\pQ}{\qquad}
\nw{\fg}{\mathfrak{g}}
\nw{\cB}{\mathcal{B}}
\nw{\ub}{\tilde{\mathbf u}}
\nw{\vb}{\tilde{\mathbf v}}
\nw{\vga}{\boldsymbol{\gamma}}
\nw{\ep}{\boldsymbol{\epsilon}}
\nw{\eps}{\epsilon}
\nw{\Te}{\Theta}
\nw{\Si}{\Sigma}
\nw{\Sol}{\operatorname{Sol}}
\nw{\ck}{\check}
\nw{\prQQQ}{\prQQ\pQ}
\nw{\cZ}{\mathcal{Z}}
\nw{\La}{\Lambda}
\nw{\fX}{\mathcal{X}}
\nw{\fY}{\mathcal{Y}}
\nw{\fZ}{\mathfrak{Z}}
\nw{\ww}{{\vec w}}
\nw{\dist}{\operatorname{dist}}
\nw{\cN}{\mathcal{N}}
\nw{\sU}{\mathscr{U}}
\nw{\io}{\iota}

\numberwithin{equation}{section}

\newtheorem{thm}{Theorem}[section]

\newtheorem{lem}[thm]{Lemma}
\newtheorem{prop}[thm]{Proposition}
\theoremstyle{remark}
\newtheorem{rem}{Remark}

\begin{document}

\title[3D Zakharov system above the ground state]{Global dynamics above the ground state energy for the 3D Zakharov system}
\thanks{Z. Guo is supported by ARC FT230100588 and DP260100485. K.~Nakanishi is supported by JSPS KAKENHI Grant Numbers 22H01132 and 23K22403.}

\author[Z. Guo]{Zihua Guo}
% \address{School of Mathematics, Monash University, VIC 3800, Australia }
% \email{zihua.guo@monash.edu}
	
\author[K. Nakanishi]{Kenji Nakanishi}
% \address{Research Institute for Mathematical Sciences, Kyoto University, Kyoto 606-8502, JAPAN}
% \email{kenji@kurims.kyoto-u.ac.jp}

\subjclass{35Q55, 35L52, 37K40}

\keywords{}
	
\begin{abstract}
We study the global dynamics for the 3D Zakharov system with radial initial data of energy slightly above the ground state energy, proving that the initial data set splits into nine nonempty, pairwise disjoint regions in which the solutions have distinct behaviors: growup, scattering to $0$, or trapped by the ground state, as time goes in the forward or backward directions. It is a classification similar to those for the nonlinear Klein-Gordon equation (\cite{NS-KG}) and for the nonlinear Schr\"odinger equation (\cite{NS-NLS}), extending the previous results in \cite{Zak-KM} on the Zakharov system below the ground state energy. The proof relies on the normal form technique used in \cite{Zak-KM} to handle the quadratic frequency interactions and careful modulational analysis around the ground state. One novelty is a new family of localized virial estimates with monotonicity. These virial estimates are crucial for us to study the dynamics away from the ground state and may be of independent interest for other purposes. 
\end{abstract}

\maketitle

\tableofcontents

\section{Introduction}

In this paper we consider the global Cauchy problem for the 3D Zakharov system 
\EQ{\label{eq:Zak}
\CAS{
i\dot u-\De u = nu,\\
\ddot n/\al^2-\De n= -\De|u|^2
}
}
where $\alpha>0$ denotes the ion sound speed, and $u(t,x):\R\times \R^3\to \C$ and $n(t,x):\R\times \R^3\to \R$ are the unknown functions
with given initial data $u(0,x)$, $n(0,x)$, and $\dot n(0,x)$. 

This system \eqref{eq:Zak} was introduced by
Zakharov \cite{Zak} as a mathematical model for the Langmuir
turbulence in unmagnetized ionized plasma. Formally, in the limit $\alpha\to \infty$, \eqref{eq:Zak} is reduced to the cubic nonlinear Schr\"odinger equation
\EQ{\label{eq:NLS}
i\dot u-\De u = |u|^2u.
}

The Zakharov system is invariant under the spatial rotation, and hence the radial symmetry is preserved under the flow.  It also preserves the mass 
and the energy, so we naturally consider the initial data in the energy space: 
\EQ{
u(0,x)\in H^1(\R^3),\pq n(0,x) \in L^2(\R^3),\pq \dot n(0,x) \in \dot H^{-1}(\R^3).
}
The Zakharov system \eqref{eq:Zak} has been extensively
studied. For the well-posedness (existence, uniqueness and continuous dependence), there are many studies, e.g. see \cite{KPV,BoCo,GTV,BGHN, ChenWu,Candy,Sanwal}.  In particular, well-posedness in the energy space was proved in \cite{BoCo}, and the sharp results were obtained in \cite{Candy} for $d\geq 4$ and in \cite{Sanwal} for $d=2,3$. 
The weak solutions are unique \cite{MN-uniq} among those weakly continuous in the energy space for $d\le 3$. 
For the study on blow-up solutions, \cite{GM} constructed self-similar blowup solutions for $d=2$. 
In the 3D radial case, solutions with negative energy were shown by \cite{Merle} to blowup in finite or infinite time (which we will call growup for brevity). The latter result was extended to data with energy below the ground state in \cite{Zak-KM} for $d=3$ and \cite{GN2} for $d=4$.

The large-time behavior of global solutions, e.g. the scattering theory, was
studied in \cite{Shimo,GV,OT2} which considered only the final data
problem (namely the wave operators). The initial data problem (e.g. the asymptotics of global solutions) was studied in a series of works \cite{GN,Zak-KM,GLNW,Guo,BGHN,GN2} by the authors and their collaborators, and also by \cite{HPS} in weighted Sobolev spaces.  
In \cite{GN} the authors obtained small energy
scattering in the 3D radial case by using the normal form technique
and radial-improved Strichartz estimates. Then in \cite{Zak-KM} the smallness was extended to the threshold given by the ground state (precise details are given in \eqref{eq:groundstate}).  In \cite{GLNW,Guo}, the radial symmetry assumption was replaced by some assumptions on the angular variables.  In \cite{BGHN}, small energy scattering in the 4D non-radial case was proved and in \cite{GN2} the smallness was extended to the threshold given by the ground state in the radial case. It is worth mentioning that scattering in the energy space in the non-radial case in 3D, even for small data, remains a challenging open problem. 

The present paper is a continuation of \cite{Zak-KM} 
going slightly beyond the ground state energy as in the NLS case \cite{NS-NLS}.  
For notational convenience, we rewrite the system \eqref{eq:Zak} as follows 
\EQ{\label{eq:Zak2}
 \CAS{(i\p_t-\De)u_s=u_s\re u_w, \\
 (i\p_t+\al D)u_w=\al D|u_s|^2,} }
where $D:=\sqrt{-\De}$, $\vec u:=(u_s,u_w):I\times\R^3\to\C^2$, $I$ is a time interval, $u_s:=u$ and $u_w:=n-iD^{-1}\dot n/\alpha$. 
The difference from the NLS appears in terms of
\EQ{
u_d:=u_w-|u_s|^2.
}
This system has the following scaling invariance with a scaling parameter $\la>0$: 
\EQ{\label{rescaling}
(u_s,u_w,u_d,\al) \mapsto (\la u_s(\la^2t,\la x),\la^2 u_w(\la^2t,\la x),\la^2 u_d(\la^2t,\la x),\la\al).
}
Moreover,
the system \eqref{eq:Zak2} can be written as the Hamiltonian system:
\EQ{
 \pt \p_t\vec u = JE_Z'(\vec u), \pq J:=\mat{i & 0 \\ 0 & 2i\al D}
}
where
\EQ{
  \pt E_Z(\vec u):=\int_{\R^3}[\tf12|\na u_s|^2+\tf14|u_w|^2-\tf12u_w|u_s|^2]dx = E_S(u_s)+\tf14\int_{\R^3}|u_d|^2dx,
 \prq E_S(\fy):=\int_{\R^3}[\tf12|\na\fy|^2-\tf14|\fy|^4]dx.
}
Note that $E_S(u)$ is the Hamiltonian for NLS \eqref{eq:NLS}.
Under the flow of \eqref{eq:Zak2}, $E_Z(\vec u)$ is conserved and the mass 
\EQ{
 M(\vec u):=M(u_s):=\tf12\int_{\R^3}|u_s|^2dx} 
is also conserved. 
Hence it is natural to work on the radial energy space
\EQN{
 \HH:=\{\vec u=(u_s,u_w)\in H^1_\rad(\R^3;\C)\times L^2_\rad(\R^3;\C)\},}
where $\cdot_\rad$ indicates the radial subspace. 

In \cite{GN}, small data scattering in $\HH$ was proved.  More precisely, the following result was proved: if $\norm{\vec u(0)}_\HH \ll 1$, the global solution $\vec u$ scatters, namely, there exists $(\vec u^\pm)\in \HH$ such that
\EQN{
\lim_{t\to \pm\I } \norm{\vec u(t)-U(t)\vec u^\pm}_{\HH}=0,
}
where $U(t)$ is the free propagator defined by 
\EQ{
 U(t)(\fy_s,\fy_w) := (e^{-it\Delta}\fy_s, e^{i\alpha t D}\fy_w).}
The radial symmetry assumption was later replaced by some angular regularity, see \cite{GLNW,Guo}. 

The problem for large data is much more complicated. Let $Q\in H^1(\R^3)$ be the ground state, i.e. the positive radial solution, of 
\EQN{
 -\De Q + Q = Q^3,
}
and for each $\la>0$ and $\theta\in \R$, let 
\EQN{
 Q_\la:=\la Q(\la x), \pq \vec Q_{\la,\te} :=(e^{-i\te}Q_\la,Q_\la^2), \pq  \vec Q:=\vec Q_{1,0}=(Q,Q^2). 
 }
Then for every $\la>0$ and $\te\in\R$, $\vec Q_{\la,\te+\la^2 t}$ is a standing wave solution of the Zakharov system \eqref{eq:Zak2}. These solutions are global in time but not scattering. 
For brevity, we denote their orbits in $\HH$ by 
\EQ{\label{eq:groundstate}
\cQ:=\{\vec Q_{\la,\te}\}_{\la>0,\te\in\R}, \pq \cQ_1:=\{\vec Q_{1,\te}\}_{\te\in\R}.
}
In the previous paper \cite{Zak-KM}, the authors and Wang have shown that the solutions of \eqref{eq:Zak2} with energy ``below" these standing wave solutions in the radial energy space are divided into two categories: scattering and growup. More precisely, the following results were proved in \cite{Zak-KM}: assume that $\vec u(0)\in \HH$ satisfies $E_Z(\vec u(0))M(\vec u(0))<E_Z(\vec Q_{\la,\theta})M(\vec Q_{\la,\theta})=E_S(Q)M(Q)$, then the dynamic of the solution $\uu$ is determined by a sign functional $K(u_s)$
\EQN{
K(\fy):=\int_{\R^3}[|\na\fy|^2-\tf 34|\fy|^4]dx}
and we have
\begin{itemize}
    \item if $K(u_s(0))\geq 0$ then the solution exists globally in time and scatters as $t\to\pm\I$.

    \item if $K(u_s(0))<0$ then the solution blows up in either finite or infinite time (growup), namely $\limsup_{t\to T^*-0}\norm{\vec u}_\HH=\limsup_{t\to T_*+0}\norm{\vec u}_\HH=\infty$, where $(T_*,T^*)$ is the maximal interval of existence.
\end{itemize}

In this paper, we extend the study to the solutions with energy slightly ``above'' those standing waves in the radial energy space,  
in the same way as for NLS in \cite{NS-NLS}. 
More precisely, letting
\EQ{\label{ene region}
S_{Z,\la}(\uu):=\la^{-1}E_Z(\uu)+\la M(\uu),\pq &S_Z(\uu):=S_{Z,1}(\uu)=S_S(u_s)+\tf14\|u_d\|_{L^2(\R^3)}^2\\
&S_S(u_s):=E_S(u_s)+M(u_s)
} 
we will classify the global dynamics in the following invariant set 
\EQ{\label{eq:intcond}
  \HH^\e:=\{\vec u\in \cH \mid  S_Z(\vec u) < S_Z(\vec Q) + \e\},}
for some sufficiently small $\e=\ep(\alpha)>0$. 
As long as the classification is essentially independent of $\al$ (such as the distinction among scattering, blow-up and solitons), 
we may use the rescaling \eqref{rescaling} to extend the classification 
to $\vec u_\la:=(\la u_s(\la^2t,\la x),\la^2 u_w(\la^2t,\la x))$ with the sound speed $\la\al$, 
where we have $S_{Z,\la}(\vec u_\la)=S_Z(\vec u)<S_Z(\vec Q) + \ep(\al)^2$. 
Replacing $\al$ with $\al/\la$, we may change the sound speed back to $\al$.  
Thus we may cover all initial data in 
\EQ{ \label{action bd}
 \Cu_{\la>0} \{\vec u\in \cH \mid S_{Z,\la}(\vec u) < S_Z(\vec Q) + \ep(\alpha/\la)^2\},}
where the scaling parameter $\la$ may be removed as follows. 
First, if $E_Z(\vec u)<0$ then we have the growup by \cite{Merle}, so we may restrict to the case of $E_Z(\vec u)\ge 0$.  
Then we have
\EQ{
 \pt 2\sqrt{E_Z(\vec u)M(\vec u)}=\inf_{\la>0}S_{Z,\la}(\vec u), 
 \pq S_Z(\vec Q_{1,0})=2\sqrt{E_S(Q)M(Q)}.}
If $M(\uu)>0$, then taking $\lambda=|E_Z(\uu)/M(\uu)|^{1/2}$ further transforms \eqref{action bd} to: 
\EQ{ \label{scale extension}
 \{\vec u\in\cH \mid [E_Z(\vec u)M(\vec u)]^{1/2} < [E_S(Q)M(Q)]^{1/2}+2\ep^2(\al |M(\uu)/E_Z(\uu)|^{1/2})\},}
where $[x]^{1/2}:=x|x|^{-1/2}$. 
Note that $\{\uu\in\cH \mid M(\uu)=0\}$ is 
the invariant subspace of electrical equilibrium with the trivial dynamics, namely the ion-sound free waves. 
Hence, defining any $\ep(0),\ep(\I)\ge 0$, 
we may include the cases $E_Z(\vec u)<0$ and $M(\vec u)=0$ into the above set. 

On the other hand, we should not expect uniformness of $\ep(\al)$ in $\al>0$. 
In other words, $\ep(\al)\to 0$ both as $\al\to +0$ and as $\al\to\I$. 
More precisely, in the subsonic limit $\al\to\I$, we may expect that the classification converges to the limit NLS equation, 
so that $\ep(\al)$ can be bounded below. However, it is known \cite[Theorem 10.1]{MN} that the Duhamel formula or the iteration sequence 
is not uniformly bounded in the energy space even locally in time as $\al\to\I$, so we need 
at least highly nontrivial work to get uniform estimates 
in the global dynamics. 
The other limit $\al\to+0$ is simpler but the conclusion is worse, since the limit equation 
\EQ{
 i\dot u-\De u = nu, \pq \dot n=0}
has many small bound states, such as $(u,n)=(\e e^{-i\la^2t} Q_\la, Q_\la^2)$ with $\la,\e\to+0$. 
It does not imply $\ep(+0)=0$, but the dynamics does not converge globally in time. 
Therefore, in this paper, we do not pursue the dependence on $\al>0$, regarding it as a given constant. 

To state the main results, we introduce the distance to $\cQ_1$ 
\EQ{ \label{def d0}
 d_0(\vec u):= \dist_\HH(\vec u, \cQ_1)=\inf_{\te\in\R}\|\vec u-\vec Q_{1,\te}\|_\HH.
}

\begin{thm} \label{thm:resol}
There is a function $\ep:(0,\I)\to(0,1)$ such that for any $\al>0$ and all solutions $\uu$ of \eqref{eq:Zak2} with initial data in $\cH^{\ep(\al)}$, 
we have the following trichotomy. Let $T^*\in(0,\I]$ be the forward maximal existence time of $\uu$. 
Then one of \textup{(1)-(3)} holds: 
\begin{enumerate}
    \item The solution $\uu$ is trapped by $\cQ_{1}$: $T^*=\I$ and $d_0(\uu(t))\leq C\e$ for large $t$;
    \item The solution $\uu$ scatters: $T^*=\I$ and $U(-t)\uu(t)$ converges in $\cH$ as $t\to\I$;
    \item The solution $\uu$ grows up: $\limsup_{t\to T^*}\|\uu(t)\|_\cH=\I$, 
\end{enumerate}
for some constant $C>0$ which may depend only on $\al$, where $\e\in(0,\ep]$ is any number such that $\uu \in \cH^\e$. For $t<0$, we have a similar trichotomy.
\end{thm}
Choosing $\ep$ small enough ensures that the cases (1) and (2) have no overlap, since the scattering implies $\|u_s(t)\|_{L^4_x}\to 0$ as $t\to\I$, 
while $\|u_s(t)\|_{L^4_x}\ge \|Q\|_{L^4_x}-Cd_0(\uu(t))$ by the Sobolev embedding $H^1(\R^3)\subset L^4(\R^3)$. 

Let $\T_+,\cS_+,\B_+$ be the set of initial data respectively in the cases of (1)-(3) of Theorem \ref{thm:resol}, 
and similarly $\T_-,\cS_-,\B_-$ be those for negative time. 
Any combination of those trichotomy between the forward and backward time directions is possible and hence we obtain nine different scenarios of the solutions. More precisely, we have
\begin{thm} \label{thm:9set}
For any $\al>0$, all solutions $\uu$ of \eqref{eq:Zak2} with initial data in $\cH^{\ep(\al)}$ exhibit one of the following nine distinct scenarios, with each case being attained by infinitely many data:
\begin{enumerate}
    \item Scattering for both $t\to \pm \infty$, $\cS_+\cap\cS_-$.
    \item Growup on both sides $\pm t>0$, $\B_+\cap\B_-$.
    \item Growup in $t<0$ and scattering as $t\to \infty$, $\B_-\cap\cS_+$. 
    \item Growup in $t>0$ and scattering as $t\to -\infty$, $\B_+\cap\cS_-$. 
    \item Trapped by $\cQ_1$ for $t\to \infty$ and scattering as $t\to -\infty$, $\T_+\cap\cS_-$.
    \item Trapped by $\cQ_1$ for $t\to -\infty$ and scattering as $t\to \infty$, $\T_-\cap\cS_+$.
    \item Trapped by $\cQ_1$ for $t\to \infty$ and growup in $t<0$, $\T_+\cap\B_-$.
    \item Trapped by $\cQ_1$ for $t\to -\infty$ and growup in $t>0$, $\T_-\cap\B_+$. 
    \item Trapped by $\cQ_1$ as $t\to \pm \infty$, $\T_+\cap\T_-$.
\end{enumerate}
Moreover, $\cS_\pm$ and $\B_\pm$ are open, $\T_\pm$ are $(Lipschitz)$ manifolds of codimension $1$, 
and $\T_+\cap\T_-$ has codimension $2$. All the subsets $\cS_\pm,\B_\pm,\T_\pm$   
and their intersections, namely the initial data sets respectively for (1)-(9), 
are path-wise connected sets in $\HH^{\ep(\al)}$. 
The sets for (1), (5), (6) and (9) are uniformly bounded in $\HH$, while the others are obviously unbounded by the growup. 
\end{thm}
$\T_+\cap\T_-$, $\T_+$, and $\T_-$ 
are respectively the center, the center-stable, and the center-unstable manifolds 
of $\cQ$ within $\cH^{\ep(\al)}$. 
The above theorem implies unboundedness of $\cS_\pm$ and $\T_\pm$ in $\HH$, due to the growup in the other time directions. 

\begin{rem}
The above theorem is analogous to the results for NLS \eqref{eq:NLS} obtained in \cite{NS-NLS}, 
but weaker in the trapped and growup cases.  
By finer spectral and dispersive analysis of the linearizing NLS \eqref{eq:NLS} around $Q$, it was proved in \cite{NS-NLS} that trapped by $Q$ can be enhanced to scattering to $Q$ for NLS. However, the same problem for \eqref{eq:Zak2} seems significantly more complicated. 
Also existence of any blow-up in finite time is still an open problem for the 3D Zakharov system. 
So we leave those questions for future study.
\end{rem}

As in the NLS case, the dynamics is simplified in the limit $\ep\to 0$ to the threshold. 
\begin{thm} \label{thm:thres}
All solutions $\uu$ of \eqref{eq:Zak2} with initial data in $\cH$ satisfying 
$E_Z(\vec u)M(\vec u) \le E_S(Q)M(Q)$ 
exhibit one of the following seven distinct scenarios, with each case being attained by infinitely many data:
\begin{enumerate}
    \item Scattering for both $t\to \pm \infty$,
    \item Growup on both sides $\pm t>0$,
    \item Exponentially convergent to $\cQ$ as $t\to \infty$ and scattering as $t\to -\infty$,
    \item Exponentially convergent to $\cQ$ as $t\to -\infty$ and scattering as $t\to \infty$,
    \item Exponentially convergent to $\cQ$ as $t\to \infty$ and growup in $t<0$,
    \item Exponentially convergent to $\cQ$ as $t\to -\infty$ and growup in $t>0$,
    \item The ground states in $\cQ$. 
\end{enumerate}
Moreover, the initial data sets for $(3)\cup(5)\cup(7)$ and $(4)\cup(6)\cup(7)$ are (Lipschitz) manifolds of dimension $3$. 
They are respectively the stable and the unstable manifolds of $\cQ$ in $\cH$. 
\end{thm}

Theorems \ref{thm:resol}--\ref{thm:9set} are proved in Section \ref{ss:cs mfd}, 
and Theorem \ref{thm:thres} in Section \ref{ss:smfd}. 
The proof follows the same strategy as \cite{NS-NLS} for NLS \eqref{eq:NLS}. The strategy and organization of this paper can be described as follows.  

\bigskip

\bigskip

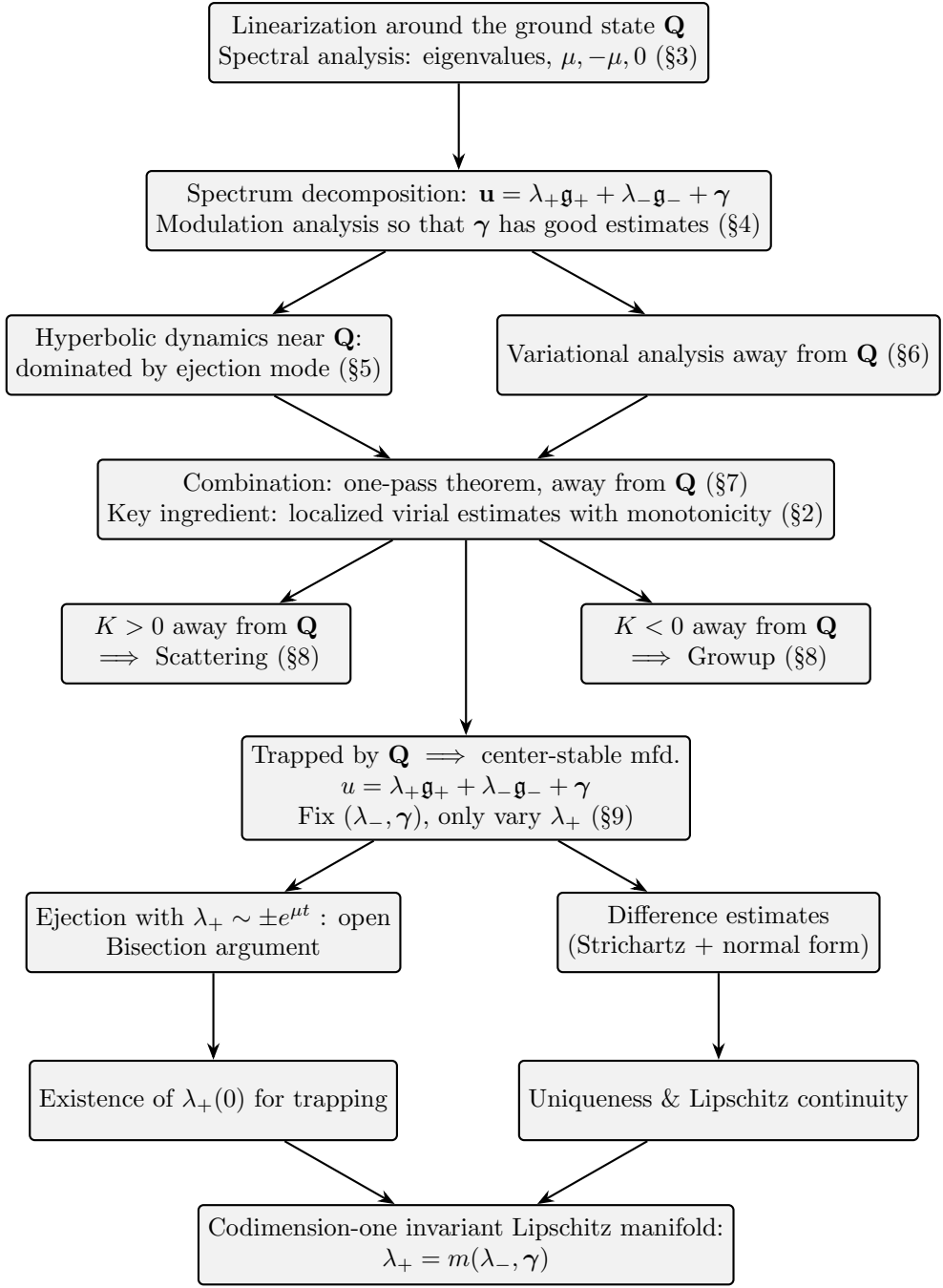
\begin{tikzpicture}[
>=Stealth,
node distance=12mm,
every node/.style={font=\small},
block/.style={
draw,
rounded corners=2pt,
thick,
minimum width=4cm,
minimum height=11mm,
align=center,
fill=gray!10
},
arrow/.style={->,thick},
darrow/.style={->,thick,dashed}
]

%%%%%%%%%%%%%%%%%%%%%%%%%%%%%%%%%%%%%%%%%%%%%%%%%%%%%%
% Main vertical flow
%%%%%%%%%%%%%%%%%%%%%%%%%%%%%%%%%%%%%%%%%%%%%%%%%%%%%%

\node[block] (A)
{Linearization around the ground state $\vec Q$\\
Spectral analysis: eigenvalues, $\mu,-\mu, 0$ (\S 3)};

\node[block,below=of A] (B)
{Spectrum decomposition: $\uu=\lambda_{+}\g_{+}+\lambda_{-}\g_{-}+\vga$\\
Modulation analysis so that $\vga$ has good estimates (\S 4)};

\node[block] (C)
at ($(B.center)+(-36mm,-20mm)$)
{Hyperbolic dynamics near $\vec Q$: \\
dominated by ejection mode (\S 5)};

\node[block] (D)
at ($(B.center)+(36mm,-20mm)$)
{Variational analysis away from $\vec Q$ (\S 6)};

%%%%%%%%%%%%%%%%%%%%%%%%%%%%%%%%%%%%%%%%%%%%%%%%%%%%%%
% Branches
%%%%%%%%%%%%%%%%%%%%%%%%%%%%%%%%%%%%%%%%%%%%%%%%%%%%%%

\node[block] (E)
at ($(D.center)+(-35mm,-20mm)$)
{Combination: one-pass theorem, away from $\vec Q$ (\S 7)\\
Key ingredient: localized virial estimates with monotonicity (\S 2)};

%%%%%%%%%%%%%%%%%%%%%%%%%%%%%%%%%%%%%%%%%%%%%%%%%%%%%%
% Continue
%%%%%%%%%%%%%%%%%%%%%%%%%%%%%%%%%%%%%%%%%%%%%%%%%%%%%%

\node[block] (F)
at ($(E.center)+(-36mm,-20mm)$)
{$K>0$ away from $\vec Q$ \\
$\implies$ Scattering (\S 8)};

\node[block] (H)
at ($(E.center)+(0mm,-40mm)$)
{Trapped by $\vec Q$  
$\implies$ center-stable mfd. \\
$u=\lambda_{+}\g_{+}+\lambda_{-}\g_{-}+\vga$\\ 
Fix $(\lambda_{-},\vga)$, only vary $\lambda_{+}$ (\S 9)};

\node[block] (N)
at ($(E.center)+(36mm,-20mm)$)
{$K<0$ away from $\vec Q$ \\ 
$\implies$ Growup (\S 8)};

\node[block] (I) 
at ($(H.center)+(-35mm,-20mm)$)
{Ejection with $\la_+ \sim \pm e^{\mu t}$ : open\\
 Bisection argument};

\node[block] (J)
at ($(H.center)+(35mm,-20mm)$)
{Difference estimates\\
(Strichartz + normal form)};

\node[block,below=of I] (K)
{Existence of $\la_+(0)$ for trapping};

\node[block,below=of J] (L)
{Uniqueness \& Lipschitz continuity};

\node[block] (M)
at ($(L.center)+(-35mm,-20mm)$)
{Codimension-one invariant Lipschitz manifold:\\
$\lambda_{+}=m(\lambda_{-},\vga)$};

%%%%%%%%%%%%%%%%%%%%%%%%%%%%%%%%%%%%%%%%%%%%%%%%%%%%%%
% Arrows
%%%%%%%%%%%%%%%%%%%%%%%%%%%%%%%%%%%%%%%%%%%%%%%%%%%%%%

%\draw[arrow] (Q)--(A);
\draw[arrow] (A)--(B);
\draw[arrow] (B)--(C);
\draw[arrow] (B)--(D);

\draw[arrow] (C)--(E);
\draw[arrow] (D)--(E);

\draw[arrow] (E)--(F);
\draw[arrow] (E)--(H);
\draw[arrow] (E)--(N);

%\draw[darrow] (G)--(H);

\draw[arrow] (H)--(I);

\draw[arrow] (H)--(J);

\draw[arrow] (I)--(K);
\draw[arrow] (J)--(L);
\draw[arrow] (K)--(M);
\draw[arrow] (L)--(M);

\end{tikzpicture}

\newpage

Compared to NLS \eqref{eq:NLS}, there are some essential difficulties for the Zakharov system.

\begin{itemize}

\item Linearized operator around the ground state.

It is essential to consider the linearized operator of the Zakharov system around the ground state. It is a system of elliptic and nonlocal operators, in contrast to the NLS case where it is an ODE (in the radial setting). Fortunately, we can identify its hyperbolic structure so that we can combine the hyperbolic dynamics around the ground state and the scattering/blowup dynamics away from the ground state.  However, the finer dynamics near the ground state requires further study. 

\item Localized virial estimates.

Another difficulty is that the virial estimate for the Zakharov system is more complicated and weaker than the NLS. The difficulty is mainly as follows: the energy only controls $L^2$ norm of the wave component, in contrast to the $H^1$ component of Schr\"odinger. There is no room to use the radial 
Sobolev embedding to control the cut-off remainder for the wave component. Moreover, the interaction energy is cubic $(n|u|^2)$ 
in contrast to quartic $(|u|^4)$ of NLS. Therefore, the virial estimate is difficult to localize without losing its monotonicity, which is essential in the proof of the one-pass theorem. Fortunately enough, as one of the main novelties of this paper, we found a new virial estimate after reformulating the Zakharov system as a local vector-valued system. 
The monotonicity of localized virial identity seems to require a unique choice of scaling exponent in the scattering case. Besides that, we need a more intricate choice of cut-off than the NLS case \cite{NS-NLS} to reduce degeneration of the Morawetz-type weight $1/|x|$.

\item Quadratic nonlinearity for long-time perturbations.

Another difficulty is caused by the quadratic nonlinearity. The nonlinear interaction structures play crucial roles in the scattering problem even for small data (\cite{GN,Zak-KM}). To prove the scattering results, we need to combine the normal form method and Kenig-Merle's concentration-compactness/rigidity method (\cite{KM}) as \cite{Zak-KM}.  

\end{itemize}

\subsection{Notation}
Here we prepare a few notations that are used throughout the paper. 
As usual, $f\lec g$ or $g\gec f$ means that there is some positive constant $C>0$ such that $f\le Cg$. 
We will fix $\al>0$ and regard it as a constant (one may assume $\al=1$ without losing essential generality). 
$f\sim g$ means $f\lec g$ and $g\lec f$. 
$f\ll g$ or $g\gg f$ means that $f/g$ is small enough, compared with some relevant constants. 
The $L^p$ norm on $\R^3$ is abbreviated as $\|f\|_p:=\|f\|_{L^p(\R^3)}$. 
The (complex and real) $L^2$ inner products on $\R^3$ are denoted by 
\EQN{
 (f|g):=\int_{\R^3}f(x)\cdot\overline{g(x)}dx, \pq \LR{f|g}:=\re(f|g).}
For any symbol $X$ with an index $0,1$, the difference and the pair are denoted by 
\EQN{
\diff{ X} := X^1-X^0, \pq X^\circ:=(X^0,X^1).}

\section{Localized virial identity}

Various virial estimates play crucial roles in the study of large-time behavior of the solutions to the Zakharov system. 
In \cite{Merle}, Merle used a virial identity to show the existence of growup solutions. In \cite{Zak-KM}, the authors used another virial identity to obtain the scattering results. 
In this section, we introduce a new family of virial identity which includes the ones used in \cite{Merle} and \cite{Zak-KM}.  In particular, our new virial identity has the advantage that it behaves better for spatial localization.  It plays an essential role in the proof of the one-pass theorem as we apply it to the solutions that do not have compactness properties. We believe it is also useful for other purposes.

\subsection{Virial identity}

We rewrite the Zakharov system \eqref{eq:Zak} (or equivalently \eqref{eq:Zak2}) 
as the following local vector system with $\bb=(-\De)^{-1}\nabla n$, $n=-\na\cdot \bb$
\EQ{\label{eq:Zak3}
 \pt \CAS{(i\p_t-\De+(\na\cdot \bb))u=0, \\ 
 (\p_0^2-\De)\bb=\na|u|^2, \pq \na\times\bb=0,}
}
where $\p_0:=\p_t/\al$. 
Denoting $\ub:=(u,\bb,\p_0\bb)=:(u_s,u_{b1},u_{b2})$ for the unknown functions in \eqref{eq:Zak3}, we have three equivalent formulations for the Zakharov system, given by \eqref{eq:Zak}, \eqref{eq:Zak2} or \eqref{eq:Zak3}.
Throughout this paper, we use $u$ or $u_s$ to denote the Schr\"odinger component which is the same in three formulations. For the wave component, we denote it by $n$, or $u_w=u_{w1}+iu_{w2}$, or $\vec b$ respectively which have the following relations:
\EQ{
 \pt u_w=n-iD^{-1}\p_0 n:=u_{w1}+iu_{w2}, 
 \pr u_{w1}=-\na\cdot\bb,\pq u_{w2}=iD^{-1}\na\cdot\p_0\bb,
 \pr u_d=u_w-|u_s|^2:=u_{d1}+iu_{d2}.
}
Note that in the notations $u_s$, $u_w$, $u_d$, we use $s,w,d$ to denote subindices, not partial derivatives of $u$. However, in this paper, we use $u_r$, $u_\rho$, and $u_\theta$ to denote partial derivatives of $u$ in those variables.
We may rewrite the energy as
\EQ{
 \ti E_Z(\ub):=E_Z(\uu)=&\int_{\R^d}\tf{|\na u|^2+|u|^2\na\cdot \bb}{2}+\tf{|\p_0 \bb|^2+|\na \bb|^2}{4}dx
 =E_S(u_s)+\tf14\|u_d\|_2^2, 
} 
and the above system \eqref{eq:Zak3} in the Hamiltonian form for $\ub\in H^1\times\dot H^1\times L^2=:\ti\cH$.
\EQ{
 \p_t \ub=\ti J \ti E'_Z(\ub), \pq \ti J:=\smat{i & 0 & 0 \\ 0 & 0 & 2\al \\ 0 & -2\al & 0}
 = -\ti J^*. }

To state our virial identity, we introduce the dilation operator and its generator
\EQ{
 S_\la^\s \fy(x):=\la^{(d+\s)/2}\fy(\la x), \pq 
 A_\s:=x\cdot\na+\tf{d+\s}{2} = -A_{-\s}^*,}
for $\s\in\R$. For any $a\in\R$ and 
$\cS_\la^a:=S_\la^0\otimes S_\la^{-a}\otimes S_\la^a$, with the generator
$\A_a:=A_0\otimes A_{-a} \otimes A_a$, we have $\ti J^* \A_a^* = \A_a \ti J$. 
Define the symplectic form by 
\EQ{
  \om(\ub,\vb) \pt:=\LR{\ti J^{-1} \ub|\vb}=\LR{-iu_s|v_s}+\tf{1}{2\al}(\LR{u_{b1}|v_{b2}}-\LR{u_{b2}|v_{b1}})
   \pr=: \om(\uu,\vv) = \LR{J^{-1}\uu|\vv}.}
Then the virial identity is given in the form 
\EQ{
 \p_t\om(\ub,\A_a \ub)\pt=\LR{\ti E_Z'(\ub)|\A_a\ub}+\LR{\U {\ti J} \ub|\A_a \ti J \ti E_Z'(\ub)}
 = 2\LR{\ti E_Z'(\ub)|\A_a\ub}
 \pr=2\ti K_Z^a(\ub), 
 \pq \ti K_Z^a(\ub) :=\p_\la \ti E_Z(\cS_\la^a \ub)|_{\la=1}
}
and we used the above identity $\A_a \ti J = \ti J^* \A_a^*$. Note that
\EQ{
\om(\ub,\A_a \ub)=\LR{-iu|A_0 u}+\tf{1}{\al}\LR{\bb|A_a\p_0 \bb}.
}
It is worth noting that introducing $\p_0=\p_t/\al$ hides the $\al$-dependence in the energy, but not in the virial. 
Since
\EQ{
\ti E_Z(\cS_\la^a \ub) 
 \pt=\int_{\R^d}\tf{\la^2|\na u_s|^2+\la^{(d-a+2)/2}|u_s|^2\na\cdot u_{b1}}{2}
 +\tf{\la^a|u_{b2}|^2+\la^{2-a}|\na u_{b1}|^2}{4}dx
 \pr= E_S(S_\la^0u_s)
 +\int_{\R^d}\tf{\la^a|u_{b2}|^2+\la^{2-a}|\na\cdot u_{b1}+\la^{(d+a-2)/2}|u_s|^2|^2}{4}dx,}
we have 
\EQ{
 \ti K_Z^a(\ub) \pt=\int_{\R^d}|\na u_s|^2+\tf{d-a+2}{4}|u_s|^2(\na\cdot u_{b1})
 +\tf{a}{4}|u_{b2}|^2+\tf{2-a}{4}|\na u_{b1}|^2dx
 \pr=K(u_s) + \tf{a}4\|u_{b2}\|_2^2+\tf{2-a}4\|u_{d1}\|_2^2 
- \tf{d+a-2}{4}\LR{u_{d1}||u|^2},}
which may be rewritten in the variable $\vec u=(u_s,u_w)$ with $u_w=u_{w1}+iu_{w2}$,
\EQ{
 \ti K_Z^a(\ub)=:K_Z^a(\vec u)=\|\na u_s\|_2^2+\tf a4\|u_{w2}\|_2^2+\tf{2-a}{4}\|u_{w1}\|_2^2-\tf{d+2-a}{4}\LR{u_{w1}||u_s|^2}.} 

So we obtain for nice solutions $(u,\bb)$ of \eqref{eq:Zak3}: For any $a\in \R$, 
\EQ{ \label{eq:vir}
 \pt\frac{d}{dt}\brk{\LR{-iu|A_0 u}+\tf{1}{\al}\LR{\bb|A_a\p_0 \bb}}=
2\ti K_Z^a(u,\bb,\p_0\bb) = 2K_Z^a(u,u_w)
 \prQ= 2K(u) + \tf{a}2\|\p_0 \bb\|_2^2+\tf{2-a}2\|u_{d1}\|_2^2 
- \tf{d+a-2}{2}\LR{u_{d1}||u|^2} 
 \prQ= 2\|\na u\|_2^2+\tf a 2\|u_{w2}\|_2^2+\tf{2-a}{2}\|u_{w1}\|_2^2-\tf{d+2-a}{2}\LR{u_{w1}||u|^2}.
 } 
It is only formal for finite energy solutions, as the left side is unbounded, 
but we will use truncated versions (Proposition \ref{prop:locvir}) 
that are justified in the energy space by the local wellposedness.  

The virial identity used in \cite{Merle} corresponds to the case $a=2-d$ while 
the one used in \cite{Zak-KM} corresponds to the case $a=1$. 
However, for its positivity, the optimal choice seems $a=0$, 
as long as we dispose of the cubic term $\LR{u_{d1} ||u|^2}$ by H\"older's inequality.  By variational analysis we can show that the virial is monotone in the regions $K(u)>0$ and $K(u)<0$. See Lemma \ref{lem:var}.

\subsection{Localized virial identity}

The virial quantity in \eqref{eq:vir} is unbounded in the energy space.  So we need to localize it in space.  As it turns out, using the new vector $\bb$ (solving a local system), one can better localize the virial identity. This is crucial for our purpose. First we introduce some operators. In the following computations of operators, we apply the usual convention, identifying functions with their multipliers, and the order of operations from right to left, while parentheses indicate the precedence of the inside. For example,
\EQ{
 \na f \fy =\na(f\fy)= (\na f)\fy + f \na \fy.}
For the logarithmic coordinate $\ro:=\log r$, $r:=|x|$ on $\R^d$, any $k\in \N$, $\s\in\R$ and $\be\in\N_0^d$, we define the following operators. Although they are defined in general, we will be working only with radial functions 
and spherical harmonics of degree $1$. 
\EQ{
 \pt \p_\ro=r\p_r=x\cdot\na, \pq r^\s \p_\ro=\p_\ro r^\s-\s r^\s, \pq \p_r^k\p_\ro=\p_\ro \p_r^k+k\p_r^k,
 \pr \na^\be \p_\ro = \p_\ro\na^\be+|\be| \na^\be, \pq x^\be \p_\ro = \p_\ro x^\be- |\be| x^\be, \pq |\be|=\be_1+\cdots \be_d,
 \pr \p_\ro^*=-\p_\ro-d=-\na\cdot x, \pq A_\s= \tf12(\p_\ro-\p_\ro^*+\s)=\p_\ro+\tf d2+ \tf\s 2 = -A_{-\s}^*, 
 \pr B_\s^\fy := \fy A_\s + A_\s \fy = \fy \p_\ro + (\s-\p_\ro^*)\fy 
 = B_0^\fy + \s\fy = -(B^\fy_{-\s})^*,
 \pr \na B_\s^\fy = (\na\fy)\p_\ro+\fy(\p_\ro+1)\na+(\s+1-\p_\ro^*)\na\fy = B_{\s+2}^\fy\na + B^{\na\fy}_{\s+1}.
}
The logarithmic coordinate is convenient for its scaling invariance, but we need to be careful about commutation with the scale dependent operators as above. 
We have
\EQ{\label{eq:ubarBu}
\re(\bar u B^\fy_\s u)=(\s-\p_\ro^*)\fy|u|^2 = \tf12[B^\fy_{\s}|u|^2 + |u|^2(\s-\p_\ro^*)\fy].
 }
For any function $f(x)$, we write $f_r:=\partial_r f$ and $f_\ro:=\partial_\ro f=r\partial_r f$.

Let $\chi:\R\to\R$ be a smooth radial function satisfying $|\chi(x)|\lec\LR{x}^{-1}$, 
and consider the tamed virial
\EQ{
 V_\chi^a(\ub) \pt:=\om(\ub,(\A_a \chi+\chi\A_a)\ub)
 =\LR{-iu|B_0^\chi u}+\tf{1}{\al}\LR{\bb|B_a^\chi\p_0 \bb}. }
Using the Hamiltonian equation, we have for any $a\in\R$
\EQ{
 \p_t V_\chi^a(\ub) \pt= \LR{\ti E_Z'(\ub)|(\A_a \chi+\chi \A_a)\ub}
 +\LR{\ti J^{-1}\ub|(\A_a \chi+\chi\A_a)\ti J \ti E_Z'(\ub)}
 \pr=2\LR{\ti E_Z'(\ub)|(\A_a \chi+\chi\A_a)\ub}
 \pr=2\LR{E_S'(u)|B^\chi_0 u} - 2\LR{u_{d1} u|B_0^\chi u}
 +\LR{\na u_{d1}|B_{-a}^\chi \bb}
 + \LR{\p_0 \bb|B^\chi_a\p_0 \bb}.
}
For the NLS part, by \eqref{eq:ubarBu} we get
\EQ{ \label{vir NLS}
 &2\LR{E_S'(u)|B^\chi_0u}\\
 =&2\LR{\na u|\na B^\chi_0 u} - 2\LR{|u|^2u|B^\chi_0 u},\\
 =&2\LR{\na u|B^\chi_2\na u+B^{\na\chi}_1 u} 
 - \LR{|u|^2|B^\chi_0|u|^2 - |u|^2\p_\ro^*\chi}\\
 =&4\LR{|\na u|^2|\chi}+2\LR{\na u|2(\na\chi)\p_\ro u + [(1-\p_\ro^*)\na\chi]u}
  +\LR{|u|^4|\p_\ro^*\chi}\\
 =&4\LR{|u_r|^2|(\p_\ro+1)\chi}+4\LR{|u_\te|^2|\chi}-\LR{\p_r|u|^2|\p_r \p_\ro^*\chi} +\LR{|u|^4|\p_\ro^*\chi},\\
 =&4\LR{|u_r|^2|(\p_\ro+1)\chi}+4\LR{|u_\te|^2|\chi} - \LR{\tf{|u|^2}{r^2}|(\p_\ro^*+2)\p_\ro \p_\ro^* \chi} 
 +\LR{|u|^4|\p_\ro^*\chi},
 }
where $\theta=\frac{x}{|x|}$, $u_\te:=\na u-\te \p_r u$ denotes the angular component of the gradient. Then 
\EQ{ \label{virZ}
 \pt\p_t V^a_\chi (\ub)=4\LR{|u_r|^2|(\p_\ro+1)\chi}+4\LR{|u_\te|^2|\chi}-\LR{\tf{|u|^2}{r^2}|(\p_\ro^*+2)\p_\ro\p_\ro^* \chi}
  +\LR{|u|^4|\p_\ro^*\chi}
 \prq-\LR{u_{d1}|B^\chi_0|u|^2-|u|^2\p_\ro^*\chi} + a\LR{|\p_0 \bb|^2|\chi}
   - \LR{u_{d1}|B^\chi_{2-a}\na\cdot \bb} - \LR{u_{d1}|B^{\na\chi}_{1-a}\cdot\bb}.}
For the two terms on the last line, we have 
\EQ{
\pt -\LR{u_{d1}|B^\chi_0|u|^2-|u|^2\p_\ro^*\chi} - \LR{u_{d1}|B^\chi_{2-a}\na\cdot \bb}
 \prq= \LR{u_{d1}|B^\chi_{2-a}u_{d1}} + \LR{u_{d1}||u|^2(\p_\ro^*+2-a)\chi}
 \prq= (2-a)\LR{|u_{d1}|^2|\chi}+\LR{u_{d1}|u|^2|(\p_\ro^*+2-a)\chi}. 
 }
Under radial symmetry, we may write $\bb=x\z(r)$
 with $\z:=x\cdot\bb/|x|^2$. Then we have 
\EQ{ \label{eq ze}
 n=-\na\cdot \bb=\p_\ro^*\z=u_{d1}+|u|^2} 
and 
\EQ{
 B^{\na\chi}_{1-a} \bb \pt= [\chi_\ro(\p_\ro+1)+(1-a-\p_\ro^*)\chi_\ro]\z
 \pr= -2\chi_\ro \p_\ro^*\z + \z (2-a-d+\p_\ro)\chi_\ro.}
Hence the last term of \eqref{virZ} becomes 
\EQ{
 -\LR{u_{d1}|B^{\na\chi}_{1-a} \bb} \pt= 
 \LR{u_{d1}|2\chi_\ro n} 
 - \LR{\p_\ro^*\z-|u|^2|\z(\p_\ro+2-d-a)\chi_\ro}
 \pr= 2\LR{u_{d1} n|\chi_\ro}-\LR{(\p_\ro+2d)\tf{|\z|^2}{2}|\up_\ro}-\LR{|u|^2\z|\up_\ro}
 \pr=2\LR{u_{d1}(u_{d1}+|u|^2)|\chi_\ro}+\LR{\tf{|\bb|^2}{2r^2}|(\p_\ro-d)\up_\ro}-\LR{|u|^2\z|\up_\ro},}
where $\up:=(d+a-2-\p_\ro)\chi$. 
Thus we obtain 
\EQ{ \label{vir0}
 \p_t V^a_\chi(\ub) \pt= 4\LR{|u_r|^2|(\p_\ro+1)\chi}+4\LR{|u_\te|^2|\chi}-\LR{\tf{|u|^2}{r^2}|(2+\p_\ro^*)\p_\ro \p_\ro^* \chi}
 +\LR{|u|^4|\p_\ro^*\chi}
 \prq +a\LR{|\p_0 \bb|^2|\chi} + \LR{|u_{d1}|^2|(2\p_\ro+2-a)\chi}-\LR{u_{d1}|u|^2|\up}
 \prq +\LR{\tf{|\bb|^2}{2r^2}|(\p_\ro-d)\up_\ro}-\LR{|u|^2\z|\up_\ro}
 \pr= 4\tilde K_Z^a(\ub)
  \prq+4\LR{|u_r|^2|\chi_\ro}-4\LR{|\na u|^2|\chi^C}-\LR{\tf{|u|^2}{r^2}|(2+\p_\ro^*)\p_\ro \p_\ro^* \chi} -\LR{|u|^4|\p_\ro^*\chi^C} 
  \prq -a\LR{|\p_0 \bb|^2|\chi^C} - \LR{|u_{d1}|^2|(2\p_\ro+2-a)\chi^C}+\LR{u_{d1}|u|^2|\up^C}
  \prq+\LR{\tf{|\bb|^2}{2r^2}|(\p_\ro-d)\up_\ro}-\LR{|u|^2\z|\up_\ro}, }
where $\chi^C:=1-\chi$, $\up^C:=(d+a-2-\p_\ro)\chi^C=d+a-2-\up$.
Thus we conclude the proof of the following crucial proposition.

\begin{prop}[Localized virial identity] \label{prop:locvir}
Assume $(u,\bb)\in C(I;\ti\cH)$ solves \eqref{eq:Zak3} on some open interval $I$. 
Let $\chi:\R^3\to\R$ be a radial function smooth in $\ro:=\log|x|\in\R$ and bounded including the derivatives in $\ro$ 
and $|x|\chi(x)\in L^\I(\R^3)$. Then for $t\in I$, 
\EQ{\label{eq:localvirial}
\p_t V_\chi^a(\ub) =& 4\tilde K_Z^a(\ub)\\
&+4\LR{|u_r|^2|\chi_\ro}-4\LR{|\na u|^2|\chi^C}-\LR{\tf{|u|^2}{r^2}|(2+\p_\ro^*)\p_\ro \p_\ro^* \chi} -\LR{|u|^4|\p_\ro^*\chi^C} \\
&-a\LR{|\p_0 \bb|^2|\chi^C} - \LR{|u_{d1}|^2|(2\p_\ro+2-a)\chi^C}+\LR{u_{d1}|u|^2|\up^C}\\
&+\LR{\tf{|\bb|^2}{2r^2}|(\p_\ro-3)\up_\ro}-\LR{|u|^2\tf{x}{r^2}\cdot \bb|\up_\ro}
}
where $\chi^C=1-\chi$, $\up=(1+a-\p_\ro)\chi$, $\up^C=1+a-\up$, and $u_{d1}=-\na\cdot\bb-|u|^2$. 
\end{prop}

\begin{rem}
We would like to compare \eqref{eq:localvirial} to the localized virial identity for NLS. For any radial solution $u\in C_t(H^1)$ of NLS \eqref{eq:NLS}, \eqref{vir NLS} implies  
\EQ{\label{eq:localvirialNLS}
&\p_t \LR{-iu|B_0^\chi u} \\
=& 4K(u)+4\LR{|u_r|^2|\chi_\ro}-4\LR{|\na u|^2|\chi^C}-\LR{\tf{|u|^2}{r^2}|(2+\p_\ro^*)\p_\ro \p_\ro^* \chi} -\LR{|u|^4|\p_\ro^*\chi^C}.
}
We see that \eqref{eq:localvirial} is much more complicated than \eqref{eq:localvirialNLS}. In particular, there are several quadratic terms involves $\bb, \dot \bb$. We only have $L^2$ control of $\nabla \bb, \dot \bb$ in the energy. So we do not have any rooms in applying radial Sobolev embedding to control the cut-off remainder. This is the main difficulty when we apply Proposition \ref{prop:locvir} for non-compact solutions.
\end{rem}

\section{Linearized operator and its spectrum}

In this section, we linearize the Zakharov system \eqref{eq:Zak2} around its ground states $\cQ_1$, using the change of variables with a phase parameter $\te\in\R$: 
\EQ{ \label{def Psi}
 \Psi_\te(\vv):=(e^{-i\te}(Q+v_s),Q^2+v_w).}
Let $\uu=\Psi_\te(\vv)$ be a solution of \eqref{eq:Zak2}, where $\te$ depends on $t$.  
Then the equation \eqref{eq:Zak2} of $\vec u$ is transformed to the following (underdetermined) system for $\vec v=(v_s,v_w)$ and $\te$:
\EQ{ \label{eq vv}
 \p_t\vv = J[\cL\vv + (\dot\te-1)(Q+v_s,0) + N(\vv)], \pq N(\vv):=(-v_s \re v_w,-\tf12|v_s|^2),}
with the linearized operator $\cL$ ($\R$-linear but not $\C$-linear) defined by 
\EQ{
 \cL(v_s,v_w) \pt:=(L_-v_s-Q\re v_w,\tf12v_w-Q\re v_s) 
  \pr= (L_+\re v_s+iL_-\im v_s-Q\re v_{d'},\tf12v_{d'}),}
where the scalar operators $L_\pm$ are defined as for NLS
\EQ{
 L_+:=-\De+1-3Q^2, \pq L_-:=-\De+1-Q^2,}
and $v_{d'}$ is defined by linearizing $u_d=u_w-|u_s|^2$, namely 
\EQ{ \label{def d'}
 v_{d'}:=v_w-2Q\re v_s.}

Differentiation of the equation for $\vec Q_{\la,\te}$ by the parameters yields the generalized kernel of the linearized operator $J\cL$: 
\EQ{
 \pt J\cL \vec Q_{(\la)} = 2\vec Q_{(\te)}, \pq J\cL \vec Q_{(\te)}=0,
 \pr \vec Q_{(\la)}:=(Q',2QQ'), \pq Q':=A_{-1}Q, \pq \vec Q_{(\te)}:=(-iQ,0).}
Next we look for real eigenvalues $\mu\not=0$. The eigenvalue problem of $J\cL$ is written as 
\EQ{ \label{eigenpb}
 \pt \mu v_s = iL_-v_s-iQ\re v_w,
 \pq \mu v_w=i\aD(v_w-2Q\re v_s),}
where the second equation is solved for $v_w$ using the Fourier multiplier
\EQ{
 v_w=\frac{2\aD}{\aD+i\mu} (Q\re v_s).}
Inserting it to the first equation, and taking its imaginary part, we obtain  
\EQ{ \label{eq v12v11}
 \mu \im v_s=(L_++R_\mu)\re v_s,}
where $R_\mu$ denotes the positive self-adjoint operator defined by 
\EQ{
 R_\mu := Q\frac{2 \mu^2}{\mu^2+(\al D)^2}Q,}
which is increasing for $\mu>0$ and continuous in the operator norm of $L^2$. $L_++R_\mu$ converges to $L_+$ as $\mu\to+0$ and to $L_-$ as $\mu\to\I$. 
Putting $f:=\sqrt{L_-}\im v_s$ (the square root is well-defined since $L_-\ge 0$), we obtain from the first equation of \eqref{eigenpb} and \eqref{eq v12v11}
\EQ{
 -\mu^2 f = \sqrt{L_-}(L_++R_\mu)\sqrt{L_-}f.}
We look for a solution of this by minimization. For each $a>0$, let 
\EQ{
 \pt b(a):=\inf\{\B_a(f) \mid f\in H^2_\rad(\R^3),\ \|f\|_2=1\}, 
  \prq \B_a(f):=\LR{\sqrt{L_-}(L_++R_a)\sqrt{L_-}f|f},}
which is non-decreasing in $a>0$ with $b(+)\le b(a)\le b(-)$, where 
\EQ{
 b(\pm) := \inf\{\B_\pm(f) \mid f\in H^2_\rad(\R^3),\ \|f\|_2=1\}, \pq \B_\pm(f):=\LR{\sqrt{L_-}L_\pm\sqrt{L_-}f|f}.}
We have $b(-)=0$ since $\B_-(f)=\|L_-f\|_2^2$ and $L_-Q=0$, while 
\EQ{
 b(+) \le -3\tf{\|Q\|_4^4}{\|xQ\|_2^2} < 0,}
because for $f:=\sqrt{L_-}r^2Q\in H^2_\rad$ we have 
\EQ{
 \pt\sqrt{L_-}f=L_-r^2Q=-4A_0Q, \pq L_+A_0Q=-2Q-Q^3, \pr \LR{2Q+Q^3|A_0Q}=\tf34\|Q\|_4^4, 
  \pq \|f\|_2^2=\LR{L_-r^2Q|r^2Q}=-4\LR{A_0Q|r^2Q}=\|2xQ\|_2^2.} 
Since $R_a\to 0$ as $a\to+0$ strongly on $L^2$, using the same $f$, we deduce that $b(a)<0$ for small $a>0$ as well.  
Then for any minimizing sequence $f_n\in H^2_\rad$ of $b(a)$ and $g_n:=\sqrt{L_-}f_n$, we have for sufficiently large $n$
\EQ{
\|L_-f_n\|_2^2 =&\LR{L_- g_n|g_n}=\B_a(f_n)+ \LR{(2Q^2-R_a)g_n|g_n} \\
\lec& \|g_n\|_2^2 = \LR{L_-f_n|f_n} \le \|f_n\|_2\|L_-f_n\|_2,}
so $\|f_n\|_2=1$ implies $\|L_-f_n\|_2\lec 1$ and 
$\|f_n\|_{H^2} \lec \norm{L_-f_n}_2+\|Q^2f_n\|_2 \lec 1$. 
Hence there is a subsequence of $f_n$ weakly convergent to some $f_\I \in H^2_\rad$ with $\|f_\I\|_2\le 1$, 
$g_n\to\sqrt{L_-}f_\I=:g_\I$ weakly in $H^1$, and $Q^2g_n\to Q^2g_\I$ strongly in $L^2$. Hence 
\EQ{
 \B_a(f_\I) \le \liminf_{n\to\I}\LR{(L_++R_a)g_n|g_n} = b(a)<0,}
which implies $f_\I\not=0$ and $f(a):=f_\I/\|f_\I\|_2 \in H^2_\rad$ is a minimizer of $b(a)$, satisfying 
\EQ{
 b(a)f=\sqrt{L_-}(L_++R_a)\sqrt{L_-}f, \pq \|f\|_2=1.}
Moreover, the above estimates imply a uniform bound $\|f\|_{H^2}\lec 1$. 
Hence the continuity of $R_a$ in the operator norm on $L^2$ implies continuity of $b(a)$ as long as $b(a)<0$. 
Then its monotonicity implies unique existence of $\mu>0$ satisfying  
\EQ{
 b(\mu) = -\mu^2,}
then the corresponding minimizer solves 
\EQ{
 -\mu^2 f = \sqrt{L_-}(L_++R_\mu)\sqrt{L_-}f, \pq \|f\|_2=1.}
The equation implies that $f\perp Q$ and $f\in H^\I(\R^3)$. Let 
\EQ{ \label{def gsw}
 \pt g_s:=\mu^{-1}(\mu+i(L_++R_\mu))\sqrt{L_-}f, 
 \pq g_w:=\frac{2\aD}{\aD+i\mu} Q\sqrt{L_-}f.}
Tracing back the above procedure, one can check that $\vec g^\pm$ defined by
\EQ{ \label{def g+-}
 \vec g^+:=(g_s,g_w), \pq \vec g^-:=(\bar g_s, \bar g_w), } 
are eigenfunctions of $J\cL$, namely $J\cL\vec g^\pm= \pm\mu\vec g^\pm$. These eigenfunctions give the stable and unstable directions around the ground state, and thus play essential roles in our analysis, especially for the dynamics around the solitons.

\section{Action expansion and modulation}

Next we expand the energy functional around the ground states using the linearized operator, 
with a choice (modulation) of the parameter which enables us to control each component in the spectral subspaces. 
Let $\vec u=\Psi_\te\vv$ as above. Then we have 
\EQ{
  S_Z(\vec u)-S_S(Q)\pt=\frac 12\LR{\LL\vec v|\vec  v} + C(\vec v), \pq C(\vec v):=-\tf12\LR{v_wv_s|v_s},}
where $C$ is the super-quadratic part,  satisfying $|C(\vec v)|\lec \|\vec v\|_\cH^3$. 
It is simpler than the NLS case because the Zakharov Hamiltonian contains only quadratic and cubic terms. 
Note also $N(\vec v)=C'(\vec v)$ as the Fr\'echet derivative in the equation \eqref{eq vv}. 
The quadratic part may be expanded as 
\EQ{
 \LR{\LL\vec v|\vec  v} \pt= \LR{L_-v_s|v_s}+\tf12\|v_w\|_2^2-2\LR{Qv_s|\re v_w}
  \pr= \LR{L_+\re v_s|\re v_s}+\LR{L_-\im v_s|\im v_s}+\tf12\|v_{d'}\|_2^2.
}

The spectral decomposition for $J\LL$ is naturally given by using the symplectic form and we can decompose $\vec v(t)$ accordingly: 
\EQ{ \label{def la+-}
 \pt \vec v (t) =: \la_+ (t)\g^+ + \la_-(t)\g^- + \vga(t),
 \pq \la_\pm:=\om(\vec v,\mp\g^\mp),}
where $\g^\pm$ are normalized versions of the eigenfunctions in \eqref{def g+-}
\EQ{
 \pt \g^\pm:=|\om(\vec g^+,\vec g^-)|^{-1/2} \vec g^\pm, \pq \om( \g^+, \g^-)=-1,}
and the normalizing factor can be computed using \eqref{def gsw}
\EQ{
 \om(\vec g^+,\vec g^-)\pt=2\LR{\re g_s|\im g_s} + \LR{\re g_w|(\al D)^{-1}\im g_w}
 \pr= -2\mu\|f\|_2^2-\mu\|\tf{2\aD}{(\aD)^2+\mu^2}Q\sqrt{L_-}f\|_2^2<0.}
Then we have 
\EQ{ \label{gaorthg}
 \om(\vga,\g^\pm)=0, \pq \om(J\LL \vga, \g^\pm)=\mp\mu\om(\vga,\g^\pm)=0,}
and also
\EQ{ \label{gaorth0}
 \om(\vga,\vec Q_{(\la)})=\om(\vec v,\vec Q_{(\la)}), \pq \om(\vga,\vec Q_{(\te)})=\om(\vec v,\vec Q_{(\te)}),}
since the generalized eigenfunctions of eigenvalues $\la_1,\la_2\in\C$ are orthogonal to each other with respect to $\om$ 
unless $\la_1+\bar\la_2=0$. The energy expansion is rewritten for this decomposition 
\EQ{ \label{action exp}
 S_Z(\vec u)-S_S(Q)=-\mu\la_+\la_-+\frac 12\LR{\LL\vga|\vga} +C(\vec v).}

Another coordinate for $(\la_+,\la_-)$ is given by 
\EQ{ \label{def la12}
 \pt \g^1:=\tf{1}{\sqrt{2}}(\g^++\g^-)=\sqrt{2}\re\g^+, 
 \pq \g^2:=\tf{1}{\sqrt{2}}(\g^+-\g^-)=\sqrt{2}i\im\g^+,
 \pr \la_1:=\tf{1}{\sqrt{2}}(\la_++\la_-)=\om(\vec v,\g^2), \pq \la_2:=\tf{1}{\sqrt{2}}(\la_+-\la_-)=\om(\vec v,-\g^1),}
so that we can also write 
\EQ{ 
 \pt \vec v=\la_1\g^1+\la_2\g^2+\vga, 
 \pq  S_Z(\vec u)-S_S(Q)=\frac 12\left[-\mu\la_1^2+\mu\la_2^2+\LR{\LL\vga|\vga}\right] +C(\vec v),}
while $\g^1$ and $\g^2$ satisfy 
\EQ{
 J\LL\g^1=\mu\g^2, \pq J\LL\g^2=\mu\g^1, \pq \om(\g^1,\g^2)=1,
  \pq 0=\om(\g^j,\vec Q_{(\la)})=\om(\g^j,\vec Q_{(\te)}),}
for $j=1,2$. 

Now we choose $\theta(t)\in\R$ of $\uu=\Psi_\te(\vv)$, cf.~\eqref{def Psi}, such that 
\EQ{ \label{orth cond0}
 \om(\vec v,\vec Q_{(\la)})=0.}
Then we have also $\om(\vga,\vec Q_{(\la)})=0$. 
Such choice of $\theta$ is unique as long as $\vec u$ is close to $\cQ_1$ and $\vv$ is small. 
More precisely, we have the following. 
Recall the distance $d_0$ to $\cQ_1$ defined in \eqref{def d0}. 
\begin{lem} \label{lem:orth}
There exists $\de_\star>0$ such that for any $\vec u=(u_s,u_w)\in\cH$ with $d_0(\vec u)\le \de_\star$, 
there is a unique $\te\in \R/2\pi\Z$ such that $\vec v:=\Psi_\te^{-1}(\uu)$ satisfies \eqref{orth cond0} 
and $\re(e^{i\te} u_s|Q')<0$. 
Moreover, $\|\vec v\|_\cH\sim d_0(\vec u)$, 
$\te$ depends smoothly on $\vec u\in\HH$, and $\te=0$ when $\vec u=\vec Q$. 
\end{lem}
The condition $\re(e^{i\te} u_s|Q')<0$ is needed for the uniqueness in $\R/2\pi\Z$, 
and it is negative because $(Q|Q')=(Q|A_{-1}Q)=-M(Q)<0$. 
The other $\te$ satisfying the orthogonality does not make $\vv$ small. 
\begin{proof}
The orthogonality condition is expanded into the components as 
\EQ{ \label{orth cond01}
 0=\om(\vec v,\vec Q_{(\la)}) \pt=\LR{e^{i\te}u_s-Q|iQ'}+\LR{(2i\aD)^{-1}(u_w-Q^2)|2QQ'}
  \pr=\im e^{i\te}(u_s|Q') + \LR{(i\aD)^{-1}(u_w-Q^2)|QQ'}.}
Let $\te_0\in\R$ such that $\|\vec u-\vec Q_{1,\te_0}\|_\cH=d_0(\vec u)$. Then we have
\EQ{
 \pt |(u_s|Q')+e^{-i\te_0}M(Q)| \le d_0(\vec u)\|Q'\|_{H^{-1}}, 
 \pr |\LR{(i\aD)^{-1}(u_w-Q^2)|QQ'}|\le d_0(\vec u) \|(\aD)^{-1}(QQ')\|_{L^2}. }
Hence if $d_0(\vec u)\le\de_\star$ is small enough, then $|\LR{(i\aD)^{-1}(u_w-Q^2)|QQ'}|$ is much smaller than $|(u_s|Q')|$, 
and so there is a unique $\te\in\R/2\pi\Z$ satisfying \eqref{orth cond0} and $\re e^{i\te}(u_s|Q')<0$. 
Moreover, $e^{i\te_0}(u_s|Q')=-M(Q)+O(d_0(\vec u))$ implies $|e^{i\te}-e^{i\te_0}|\lec d_0(\vec u)$ and 
\EQ{
 \|\vec v\|_\cH = \|\vec u-\vec Q_{1,\te}\|_\cH \le \|\vec u-\vec Q_{1,\te_0}\|_\cH + |e^{i\te}-e^{i\te_0}|\|Q\|_{H^1}
  \lec d_0(\vec u) \le \|\vec u-\vec Q_{1,\te}\|,}
where the last inequality is by definition of $d_0$. Hence $\|\vec v\|_\cH\sim d_0(\vec u)$. 
Since $\vec u\mapsto (u_s|Q'),\LR{(i\aD)^{-1}(u_w-Q^2)|QQ'}$ is continuous, $\vec u\mapsto \te\in \R/2\pi\Z$ 
is also continuous. 
In fact, the explicit form implies that it is real analytic in $\uu\in\HH$. 
\end{proof}

Under the orthogonality condition \eqref{orth cond0}, $\cL$ is coercive on $\vga$, 
since it is bigger than the quadratic form for NLS with the same number of orthogonality conditions. 
More precisely
\begin{lem}
If $\cH\ni \vga=(\ga_s,\ga_w)$ satisfies $0=\om(\vga,\g^2)=\om(\vga,\vec Q_{(\la)})$, then we have 
\EQ{
 \LR{\LL\vga|\vga}\sim \|\ga_s\|_{H^1}^2+\|\ga_{d'}\|_{L^2}^2 \sim \|\vga\|_{\HH}^2,}
 where $\ga_{d'}:=\ga_w-2Q\re \ga_s$ as defined in \eqref{def d'}. 
\end{lem}
\begin{proof}
We split it into the real and the imaginary parts: 
\EQ{
 \pt \LR{\cL \vga|\vga}=\LR{\cL \ga_1|\ga_1}+\LR{\cL i\ga_2|i\ga_2}, 
  \prq \ga_1=(\ga_{s1},\ga_{w1}):=\re \vga, \pq \ga_2=(\ga_{s2},\ga_{w2}):=\im\vga.}

For the imaginary part,  we have the orthogonality
\EQ{ \label{gaorth1}
  0=\om(\vga,\vec Q_{(\la)})= \LR{\im \vga|(Q',(\aD)^{-1}QQ')},}
and the coercive estimate for the NLS linearized operator: 
\EQ{ \label{L-est}
 \|\fy\|_{H^1}^2 \lec \LR{L_-\fy|\fy}+\LR{\fy|Q'}^2.}
To prove the latter, let $\fy=aQ+\psi$ such that $\LR{\psi|Q}=0$. 
Since $L_-Q=0$ and $L_->0$ on $Q^\perp$, we have $\|\psi\|_{H^1}^2 \lec \LR{L_-\psi|\psi}=\LR{L_-\fy|\fy}$. 
On the other hand, we have $a\LR{Q|Q'}=\LR{\fy|Q'}-\LR{\psi|Q'}$, $\LR{Q|Q'}=-M(Q)\not=0$, and so 
$|a|^2 \lec |\LR{\fy|Q'}|^2+\LR{L_-\fy|\fy}$. 
Hence $\|\fy\|_{H^1}^2 \lec |a|^2+\|\psi\|_{H^1}^2 \lec |\LR{\fy|Q'}|^2 + \LR{L_-\fy|\fy}$, as desired. 
Injecting \eqref{gaorth1} into \eqref{L-est} yields
\EQ{
 \|\ga_{s2}\|_{H^1}^2 \sim\LR{L_-\ga_{s2}|\ga_{s2}}+\LR{(\al D)^{-1}\ga_{w2}|QQ'}^2,}
and thus we obtain $\|\ga_2\|_\cH^2 \sim \LR{L_-\ga_{s2}|\ga_{s2}} + \tf12\|\ga_{w2}\|_2^2 = \LR{\cL i\ga_2|i\ga_2}$. 
 
 For the real part, we have the orthogonality 
 \EQ{
  0 = \mu\om(\vga,\g^2) =\om(\vga,J\cL \g^1)=-\LR{\cL \g^1|\ga_1},
   \pq \LR{\cL \g^1|\g^1}=\om(\mu \g^2,\g^1)=-\mu<0.}
If the desired estimate $\|\ga_1\|_\cH^2 \lec \LR{\cL \ga_1|\ga_1}$ fails, 
then there is a sequence $\vv_n\in\cH$ such that $\bar\vv_n=\vv_n$ (real-valued), 
$\LR{\cL\g^1|\vv_n}=0$, $\|\vv_n\|_\cH=1$ and so weakly convergent to some limit $\vv_\I$, 
and $\limsup_{n\to\I}\LR{\cL\vv_n|\vv_n}\le 0$. 
Since it is a compact perturbation of positive form, we have $\LR{\cL\vv_\I|\vv_\I}\le 0$. 
Then $\cL$ is non-positive definite on the two dimensional subspace $V$ spanned by $\g^1$ and $\vv_\I$. 
Since for real-valued $\vv\in\cH$ we have 
\EQ{
  \LR{\cL\vv|\vv}=\LR{L_+v_s|v_s}+\tf12\|v_{d'}\|_2^2 \ge \LR{L_+v_s|v_s},}
the quadratic form of $L_+$ is also non-positive on $V$. 
Since $L_+$ has only one negative eigenvalue and it is positive on the radial subspace orthogonal to that eigenfunction, 
$V$ must be one-dimensional for the $v_s$ component, but then it must contain a non-zero vector with $v_s=0$, 
for which $\cL$ is positive, which is a contradiction. 
\end{proof}
As an application of the above lemma, we deduce that for $0\le c\le 4$, 
\EQ{
 \mu\om((Q,cQ^2),\g^2)\pt=-\LR{\g^1|\cL(Q,cQ^2)}
 =c\LR{\g^1_s|Q^3}+(1-\tf c2)\LR{\g^1_w|Q^2},}
is non-zero and so with a fixed sign, since 
\EQ{
 \om((Q,cQ^2),\vec Q_{(\la)})=0, \pq \LR{\cL(Q,cQ^2)|(Q,cQ^2)} = c(\tf c2-2)\|Q\|_4^4 \le 0.} 
Therefore, by changing the sign of $f$ if necessary, we may assume 
\EQ{ \label{sign f}
 c\LR{\g^1_s|Q^3}+(1-\tf c2)\LR{\g^1_w|Q^2}>0}
uniformly for $0\le c\le 4$, so in particular $\LR{\g^1_s|Q^3}>0$ and $\LR{\g^1_w|Q^2}>0$.

Now we can define the linearized energy norm $\|\vec v\|_E$ by 
\EQ{ \label{def E}
 \pt \|\vec v\|_E^2=\mu(\la_1^2+\la_2^2)+\LR{\cL\vga|\vga}=\mu(\la_+^2+\la_-^2)+\LR{\cL\vga|\vga},}
on the orthogonal subspace 
\EQ{
 \pt \HH_\perp:=\{\vec v\in \HH\mid \om(\vec v,\vec Q_{(\la)})=0\}.}
The above lemmas imply the equivalence $\|\vec v\|_E\sim\|\vec v\|_\cH\sim d_0(\vec u)$. 
Since $|C(\vec v)|\lec \|\vec v\|_\cH^3$, there is $\de_0\in(0,\de_\star]$, where $\de_\star$ is the constant in Lemma \ref{lem:orth}, 
such that 
\EQ{
 d_0(\vec u)\le \de_0 \implies |C(\vec v)|<\frac{1}{4}\|\vec v\|_E^2 \sim d_0(\vec u)^2.} 
Fixing such $\de_0$, we now define the nonlinear distance $d_Q(\vec u)$, using the component $\la_1$, 
\EQ{
 d_Q(\vec u)^2\pt:=\chi(d_0(\vec u))[S_Z(\vec u)-S_S(Q)+\mu\la_1^2]
 +[1-\chi(d_0(\vec u))]d_0(\vec u)^2,}
where $\chi\in C^\I(\R)$ is a non-increasing function satisfying $\chi(t)=1$ for $t\le \de_0/2$ and $\chi(t)=0$ for $t\ge \de_0$. Then 
\begin{lem}\label{lem:deltaE}
There exists  $\de_E>0$ such that for any $\vec u\in\cH$ we have
\EQ{ \label{def deE}
 \pt d_Q(\vec u)\le\de_E \implies d_0(\vec u)\le\de_0/2 
 \pr\implies d_Q(\vec u)^2=S_Z(\vec u)-S_S(Q)+\mu\la_1^2=\frac{\|\vec v\|_E^2}{2}+C(\vec v) \in [\tf14\|\vec v\|_E^2,\tf34\|\vec v\|_E^2], }
where $\vv=\Psi_\te^{-1}(\uu)\in\HH_\perp$ is given by Lemma \ref{lem:orth}. 
Hence $d_Q(\vec u)\sim d_0(\vec u)$ uniformly for all $\vec u\in\cH$. 
\end{lem}

\section{Ejection process}

In this section we prove that when a solution $\vec u$ is getting away from the neighborhood of $\cQ_1$, the evolution is essentially dominated by the unstable mode $\g^+$, exponentially growing. First statement is a basic and static property of the neighborhood. 
\begin{lem}[Eigenmode dominance] \label{lem:eidom}
Let $\delta_E$ be as in Lemma \ref{lem:deltaE}.  Then for any $\vec u\in\HH$ satisfying 
\EQ{ \label{cond eidom}
 2[S_Z(\vec u)-S_S(Q)]\le d_Q(\vec u)^2 \le \de_E^2,}
we have 
\EQ{
 \tf14 \mu\la_1^2 \le d_Q(\vec u)^2 \le 2\mu\la_1^2,}
where $\la_1$ is defined by \eqref{def la12} for $\vv\in\HH_\perp$ given by Lemma \ref{lem:orth}.
\end{lem}
\begin{proof}
Obvious from \eqref{def deE} and the definition of the $E$ norm \eqref{def E}. 
\end{proof}

Next we consider the dynamics of solutions getting away from $\cQ_1$, for which we need the equation for each component of the decomposition in \eqref{def la+-}
\EQ{
 \vec u = \Psi_\te(\vv), \pq \vec v=\la_+\g^++\la_-\g^-+\vga, \pq \la_\pm=\om(\vec v,\mp\g^\mp)}
where $\te(t)$ is determined by Lemma \ref{lem:orth} for $d_0(\vec u)\le\de_\star$ or $d_Q(\vec u)\le\de_E$. 
Differentiating \eqref{orth cond0}, and injecting the equation \eqref{eq vv} of $\vec v$, we obtain 
\EQ{
 0\pt=\p_t\om(\vec v,\vec Q_{(\la)})=\om(\p_t \vec v,\vec Q_{(\la)})
 \pr= -\om(\vga,2\vec Q_{(\te)})+(\dot\te-1)[-\om(\vec Q_{(\te)},\vec Q_{(\la)})+\LR{v_s|Q'}]+\LR{N(\vec v)|\vec Q_{(\la)}}.}
Since $\om(\vec Q_{(\te)},\vec Q_{(\la)})=-\LR{Q|Q'}=M(Q)>0$, and 
$M(Q)-\LR{v_{s}|Q'}=-\LR{e^{i\te}u_s|Q'}>0$ in the region $d_0(\vec u) \le \de_\star$ of Lemma \ref{lem:orth},
we obtain 
\EQ{ \label{eq te}
 \dot\te-1=[M(Q)-\LR{v_{s}|Q'}]^{-1}[-2\om(\vga,\vec Q_{(\te)})+\LR{N(\vec v)|\vec Q_{(\la)}}]=O(\|\vga\|_\HH+\|\vec v\|_\HH^2),}
as long as $d_0(\vec u)<\de_\star$. It also implies that $\te\in C^1$ on such time intervals. 
We could make it $O(\|\vec v\|_\HH^2)$ by introducing modulation of the scaling $\la$ of solitons $\vec Q_{\la,\te}$ 
together with orthogonality against $\vec Q_{(\te)}$, 
or by imposing mass constraint $M(\vec u)=M(\vec Q_{\la,\te})$. 
It would be needed to prove the (conditional) asymptotic stability, 
but otherwise it is simpler not to impose such constraints. 

Inserting the above equation into \eqref{eq vv}, we obtain a closed equation of $\vec v$
\EQ{ \label{eq v}
 \pt\p_t\vec v = J[\LL\vec v + \ti N(\vec v)], 
 \pq \ti N(\vec v):=N(\vec v)+\frac{-2\om(\vga,\vec Q_{(\te)})+\LR{N(\vec v)|\vec Q_{(\la)}}}{M(Q)-\LR{v_{s}|Q'}}(Q+v_s,0).}
which makes sense in the region $d_0(\vec u) \le \de_\star$, 
containing the region $d_Q(\vec u)<\de_E$ of Lemma \ref{lem:eidom}. By Sobolev and H\"older, we have 
\EQ{
 \|J\ti N(\vec v)\|_{L^{3/2}_x} \lec \|\vec v\|_{\HH}^2 + \|\vga\|_\HH.}

Obviously the regularity on the left side is not sufficient to get a closed estimate directly for $\vec v$ from the equation, since we suffer from the same derivative loss as in the original Zakharov system. However it does not matter for the ejection process (Lemma \ref{lem:deltaX}) or the one-pass Theorem \ref{thm:onepass}, because we need only to control the eigenmodes, for which very weak topology in $x$ is enough, in addition to the Hamiltonian structure. The equations for $\la_\pm$ are 
\EQ{ \label{eq la}
 \pt \p_t \la_\pm = \pm\mu\la_\pm + \ti N_\pm(\vec v), 
 \pq \ti N_\pm(\vec v):=\LR{\ti N(\vec v)|\mp\g^\mp}=O(\|\vec v\|_\HH^2),}
where the linear part of $\ti N(\vec v)$ does not contribute since $\LR{Q|g_s}=0$. 
It also implies $\la_\pm,\la_1,\la_2\in C^1$ on time intervals where $d_0(\vec u)\le\de_\star$. 
Combined with the eigenmode dominance Lemma \ref{lem:eidom}, these yield the following. There exists $\de_X\in(0,\de_E)$ such that if a solution $\vec u$ satisfies \eqref{cond eidom} and $d_Q(\vec u)\le\de_X$ at $t=0$, as well as 
\EQ{ \label{cond eject}
 \la_+(0) \sim \la_1(0),}
then we have 
\EQ{
 \la_\pm(t) = e^{\pm \mu t}\la_\pm(0) + O((e^{\mu t}\la_+(0))^2), \pq \la_1(t)\sim\la_+(t),}
as long as $d_Q(\vec u)\le\de_X$ or $|e^{\mu t}\la_\pm(0)|\ll\de_X$. 

The condition \eqref{cond eject} determines that the solution $\vec u$ is ejected away from $\cQ_1$, otherwise it may well be approaching $\cQ_1$, as all the other conditions are invariant for the time inversion. 
\eqref{cond eject} is fulfilled provided that 
\EQ{
 \p_td_Q(\vec u)|_{t=0}\ge 0.}
This is because for $d_Q(\vec u)<\de_E$ we have 
\EQ{
 \p_td_Q(\vec u)^2 \pt= \p_t[S_Z(\vec u)-S_S(Q)+\mu\la_1^2]=2\mu\la_1\dot\la_1
 \pr=2\mu\la_1(\mu\la_2+\LR{\ti N(\vec v)|\g^2})=2\mu^2\la_1\la_2+O(\la_1^3),}
and so $\p_td_Q(\vec u)\ge 0$ implies that $(\sign\la_1)\la_2\gec-\la_1^2$ and $\la_+=(\la_1+\la_2)/\sqrt 2\sim\la_1$. 

To see that the dispersive component $\vga$ remains in a small order, we use the energy difference. Let 
\EQ{
 S_\g(\la):=S_Z(\vec Q+\la_+\g^++\la_-\g^-)=S_S(Q)-\mu\la_+\la_-+C(\la_+\g^++\la_-\g^-),}
then we have 
\EQ{
 \p_t S_\g(\la)\pt=-\mu(\dot\la_+\la_-+\la_+\dot\la_-)+\LR{N(\la_+\g^++\la_-\g^-) \mid \dot\la_+\g^+ + \dot\la_-\g^-}
 \pr=\LR{N(\la_+\g^++\la_-\g^-)-\ti N(\vec v) \mid \dot\la_+\g^+ + \dot\la_-\g^-}
 \pr=O(\|\vga\|_\HH\la_1^2+\la_1^4).}
On the other hand we have 
\EQ{
 S_Z(\vec u)-S_\g(\la)\pt=\frac 12\LR{\LL\vga|\vga}+C(\vec v)-C(\vec v-\vga)
 \pn\sim \|\vga\|_\HH^2 + O(\|\vga\|_\HH \la_1^2),}
Hence integration in $t$ using the exponential growth of $\la_1$ yields
\EQ{
 \|\vga\|_{L^\I(0,T;\HH)} \lec \|\vga(0)\|_\HH + \la_1(T)^2,}
or in a different form
\EQ{
 \|\vga(t)\|_\HH \lec \|\vga(0)\|_\HH + O((e^{\mu t}\la_+(0))^2).}

In order to continue the dynamical analysis beyond $d_Q(\vec u)<\de_X<\de_E$ and into the variational region of $\HH$, we also need the exponential behavior 
of the virial quantities. Using that $K^a_Z(\vec Q)=0$, we may expanded
\EQ{
 \pt K^a_Z(\vec u)=\LR{(K_Z^a)'(\vec Q)|\vec v}+O(\|\vec v\|_{H^1}^2),}
where the linearized term equals
\EQ{
 \LR{(K_Z^a)'(\vec Q)|\vec v}\pt=\LR{-2\De Q-3Q^3|v_s}-\tf{a+1}{4}\LR{v_{d'}|Q^2}
 \pr=-\LR{2Q+Q^3|v_s}-\tf{a+1}{4}\LR{v_w-2Qv_s|Q^2}
 \pr=-[\tf{1-a}{2}\LR{Q^3|\g^1_s}+\tf{a+1}{4}\LR{Q^2|\g^1_w}]\la_1 + O(\|\ga_s\|_2),}
and the coefficient on $\la_1$ is negative by \eqref{sign f} for $-3\le a\le 1$. Hence we conclude that 
\EQ{
 -K_Z^a(\vec u) \sim \la_+ + O(|\la_+(0)|+\la_+^2) = e^{\mu t}\la_+(0) + O(|\la_+(0)|+(e^{\mu t}\la_+(0))^2),}
in the ejection dynamics. In particular, if 
\EQ{
 |\la_+(0)|\sim|\la_1(0)|\sim d_Q(\vec u(0))\ll\de_X,}
then $\sign K_Z^a(\vec u)=-\sign\la_+=-\sign\la_1$ when $d_Q(\vec u(t))\sim\de_X$. 
Thus we obtain the following lemma. 
Recall that $d_Q(\uu(t))$ is $C^1_t$ when $d_0(\uu(t))<\de_\star$ as observed in \eqref{eq la}. 
\begin{lem}\label{lem:deltaX}
There exist constants $0<\delta_X<\delta_E$ and $C_*,T_*>0$, where $\delta_E$ is given in Lemma \ref{lem:deltaE}, with the following properties: let $\vec u(t)$ be a local solution of the Zakharov system \eqref{eq:Zak2} in $\HH$ on an interval containing $0$, satisfying
\EQ{ \label{eject cond ene}
R=:d_Q(\vec u(0))\leq \delta_X, \quad S_Z(\vec u)<S_S(Q)+R^2/2, \pq \p_t d_Q(\vec u(t))|_{t=0} \ge 0.}
Then $\vec u$ is extended as long as $d_Q(\vec u(t))\leq \delta_X$, where $\lambda_1(t), \lambda_+(t)$ do not change the sign, 
satisfying for $-3\le a\le 1$,
\EQ{
\pt d_Q(\vec u(t))\sim |\lambda_1(t)|\sim |\lambda_+(t)|\sim e^{\mu t}R,
\pr |\la_\pm-e^{\pm\mu t}\la_\pm(0)| \lec |\la_+|^2, \pq \|\vga\|_\HH \lec \|\vga(0)\|_\HH + |\la_+|^2,
\pr -\sgn(\lambda_1(t)) K^a_Z(\vec u(t))\ges (e^{\mu t}-C_*)R.
}
Moreover, $d_Q(\vec u(t))$ is increasing for $t\geq T_*R$, and $|d_Q(\vec u(t))-R|\le C_*R^3$ for $0\leq t \leq T_*R$.
\end{lem}

The trapped solutions are those staying eventually in the neighborhood of $\cQ_1$. 
As a consequence of the ejection lemma, we have the following characterization.  
Note that we do not need either dispersion or virial monotonicity for its proof.
\begin{lem} \label{lem:trap}
Let $\e\in(0,\de_X)$ and $\vec u\in C((T_*,\I);\cH)$ be a solution of \eqref{eq:Zak2} in $\cH^\e$ satisfying 
\EQ{ \label{basin attra}
 T_*<\exists T<\forall t,\ d_Q(\vec u(t))<\de_X, }
where $\de_X>0$ is the constant given by Lemma \ref{lem:deltaX}, 
and let $d_u(t):=d_Q(\vec u(t))$.
Then either $T_*=-\I$ and $\sup_{t\in\R}d_u(t)<2\e$, or there exist $T_*<T_X<T_\e<\I$ such that 
$d_u(t)$ is decreasing on $(T_X,T_\e)$ from $d_u(T_X)=\de_X$ to $d_u(T_\e)=2\e$, 
$d_u(t) \sim e^{\mu(t-T_X)}\de_X\sim|\la_-(t)|$ on $(T_X,T_\e)$ and $d_u(t)<2\e$ for $t>T_\e$. 
In particular $S_Z(\vec u)\ge S_Z(\vec Q)$. Moreover, if $S_Z(\vec u)=S_Z(\vec Q)$, then 
either $\vec u(t)=\vec Q_{1,\te+t}$ for some $\te\in\R$, or 
$d_u(t)\sim e^{\mu(T_X-t)}\de_X \sim |\la_-(t)|$ for $t\ge T_X$. 
All the same statements hold for the other time direction $t\to-\I$. 
\end{lem}
\begin{proof}
Let $B:=\sup_{t>T_*}d_u(t)$. If $B<2\e$ then $T_*=-\I$ and $\vec u$ is in the first case. 
Hence we assume $B\ge 2\e$. Define $T_X,T_O$ by 
\EQ{
 \pt T_X:=\inf\{t>T_* \mid d_u(t)<\de_X\}, \pq T_O:=\inf\{t>T_* \mid d_u(t)<\sqrt{2}\e\},
  \pr T_\e:=\sup\{t>T_* \mid d_u(t)\ge 2\e\}.}
Then $T_*\le T_X<T_O<\I$, $d_u(t)<\de_X$ for $t>T_X$ and $d_u(t)\ge\sqrt{2}\e=d_u(T_O)$ for $T_*<t\le T_O$.  
If Lemma \ref{lem:deltaX} applies at any $t>T_X$, the conclusion that $d_u(t)$ reaches $\de_X$ 
would contradict, so $\p_t d_u(t)<0$ for $T_X<t<T_O$. 
Hence $T_X<T_\e<T_O$ with $d_u(T_\e)=2\e$ and we may apply Lemma \ref{lem:deltaX} backward in time from any $t\in(T_X,T_O)$, 
which implies the exponential behavior on this interval. 

If $S_Z(\vec u)\le S_Z(\vec Q)$, the above holds for all $\e>0$, 
hence $d_u(t)\to 0$ as $t\to\I$, and so $S_Z(\vec u)=S_Z(\vec Q)$. 
If $\sup d_u(t)=0$ then $\vec u(t)=\vec Q_{1,\te+t}$ for some $\te\in\R$, 
otherwise $d_u(t)\sim e^{\mu(t-T_X)}\de_X$ for $T_X<t<T_\e\to\I$ as $\e\to+0$. 
\end{proof}

To identify the trapped solutions with $S_Z(\vec u)=S_Z(\vec Q)$ as the stable manifold, 
we need a Lipschitz estimate for the difference of (such) solutions, see Lemma \ref{lem:unst mfd}.  

\section{Variational analysis away from the ground state} \label{s:var}

In this section we prove some variational analysis when the solution stays away from the ground state. 

\subsection{Variational analysis}
First we extend the variational bounds \cite[Lemma 2.4]{Zak-KM} below the ground states 
to include the scaling parameter $a$, by essentially the same proof.

\begin{lem}[Variational estimates] \label{lem:var}
Let $\fy\in H^1(\R^3)$, $\la>0$ and $\nu>0$ satisfy
\EQ{
 E_S(\fy)+\la^2 M(\fy) + \tf{\nu^2}{4} \le E_S(Q_\la)+\la^2M(Q_\la).}
Then we have 
\EQ{
 \CAS{ K(\fy) \ge 0,\ a\le\tf 53 \implies 4K(\fy)+(2-a)\nu^2 > \sqrt{6}\nu\|\fy\|_4^2, \\
  K(\fy) \le 0,\ -1\le a\le 1 \implies 4K(\fy)+(2-a)\nu^2 < -(a+1)\nu\|\fy\|_4^2.}}
\end{lem}
The following proof indicates that $\sqrt{6}$ in the coefficient may be replaced with $a+1$ for $a\le a_*$ 
with some $a_*\in(\sqrt{6}-1,5/3)$. 
This threshold $a_*$ may be defined precisely by a cubic equation, 
but the above version suffices, as we will use $K^a_Z$ only for $a\le 1$. 
\begin{proof}[Proof of Lemma \ref{lem:var}]
We give just a sketch because it is the same as in our previous paper \cite{Zak-KM}. 
We may assume $K(\fy)\not=0$, as well as $\la=1$ by rescaling. $Q$ minimizes $S_S=E_S+M$ among nonzero functions satisfying $K=0$.
Let $\mu:=\tf43\|\na\fy\|_2^2/\|\fy\|_4^4$, so that $K(S_\mu^0\fy)=0$ with $\sign(\mu-1)=\sign K(\fy)$, and so
\EQ{
 \tf{\nu^2}{4} \le S_S(Q)-S_S(\fy) \le S_S(S_\mu^0\fy)-S_S(\fy) = \tf{(\mu-1)^2(\mu+2)}{8}\|\fy\|_4^4.
 }
Let $\te:=|\mu-1|\|\fy\|_4^2/\nu$. Then the above inequality implies $\te\ge\te_0:=\sqrt{\tf{2}{\mu+2}}$, 
while the desired estimates are reduced to 
\EQ{
 f(\te,\mu):=3\te + \tf{(2-a)(\mu-1)}{\te} > \CAS{\sqrt{6} &(\mu>1)\\
  a+1 &(0<\mu<1).}}
For any fixed $\mu\in(0,1)$, $f$ is increasing in $\te>0$, 
while for $\mu>1$, $f$ has a unique local minimal point $\te=\te_*=\sqrt{(2-a)(\mu-1)/3}$ in $\te>0$.  
Hence its minimum for $\te\ge\te_0$ is attained at $\te=\te_*$ iff $\mu>1$ and $\te_*>\te_0$, namely $(2-a)(\mu-1)(\mu+2)>6$, and then 
\EQ{
 \min_{\te\ge\te_0}f(\te,\mu) = f(\te_*,\mu)= 2\sqrt{3(2-a)(\mu-1)}.}
Since $\te_*>\te_0$ implies for $x:=(2-a)(\mu-1)$ to satisfy $\tf{x^2}{2-a}+3x>6$, which does not hold at $x=1/2$ if $a\le\tf{35}{18}$, we deduce that $x>1/2$ and hence 
\EQ{
 f(\te_*,\mu) > 2\sqrt{3/2} = \sqrt{6}.}
In the other case we have 
\EQ{
 \pt \min_{\te\ge\te_0}f(\te,\mu) = f(\te_0,\mu) =: b(\mu),
 \pq \tf{\p b^2}{\p\mu} = \tf{3\be}{2(\mu+2)^2}, 
 \pr \be(\mu):=(c\mu^2+c\mu+6-2c)(c\mu^2+3c\mu+2c-2), \pq c:=2-a.}
By direct computation, 
we may easily check that if $c\ge 0$ then $\be''(\mu)>0$ for all $\mu\ge 0$, 
if $c\ge 1/3$ then $\be(1),\be'(1)\ge 0$, and if $1\le c\le 3$ then $\be(0),\be'(0)\ge 0$. Hence 
\EQ{
 \pt  a\le \tf53,\ \mu>1 \implies b(\mu)>b(1)=f(\sqrt{2/3},1)=\sqrt{6},
  \pr -1\le a\le 1,\ \mu>0 \implies b(\mu)>b(0)=f(1,0)=a+1,}
which are the desired lower bounds respectively for $\mu>1$ and for $0<\mu<1$. 
\end{proof}

For the dynamics away from the ground states $\cQ_1$, we can exploit that the ground state is the minimal zero for $K$, in order to extend the variational estimates slightly beyond the ground energy, i.e.~$S_Z(\vec u)<S_S(Q)+\e^2$. 
\begin{lem}\label{lem:epV}
There exist $c_K\in(0,\I)$ and increasing functions $\e_V,\ka_V:(0,\I)\to(0,1)$ 
such that for any $\de>0$ and $\vec u\in\HH$ satisfying $S_Z(\vec u)<S_S(Q)+\e_V(\de)^2$ and $d_Q(\vec u)\ge\de$, 
we have the following. \\
(1) If $K(u_s)\ge 0$ and $0\le a\le 1$ then
\EQ{
  \pt K(u_s) \ge \min(\ka_V(\de),c_K\|\na u_s\|_2^2),
  \pr K^a_Z(\vec u) \ge \min(\ka_V(\de),c_K(\|\na u_s\|_2^2+\tf{a}{4}\|u_{w2}\|_2^2+\tf{2-a}{4}\|u_{w1}\|_2^2)).}
(2) If $K(u_s)<0$ and $-1\le a\le 1$ then $\max(K(u_s),K_Z^a(\vec u)) \le-\ka_V(\de)$. 
\end{lem}
In particular, $\sign K(u_s)=\sign K^a_Z(\vec u)$ is preserved as long as $d_Q(\vec u)\ge\de$. Note however that our main theorem includes the fact that there exist solutions which change $\sign K^a_Z(\vec u)$ with $S_Z(\vec u)<S_S(Q)+\e^2$ for any $\e>0$. In the previous paper \cite{Zak-KM}, the above estimate below the ground energy $S_Z(\vec u)<S_S(Q)$ was proved for $a=1$.
As the following proof shows, the upper bound on $a$ in the case (1) may be extended with any number below $\sqrt{6}-1$. 
\begin{proof}[Proof of Lemma \ref{lem:epV}]
The estimate on $K(u_s)$ and the lower bound on $K_Z^a(\uu)$ are reduced to the NLS estimate above $Q$ and the Zakharov estimate below $\vec Q$. 
Indeed, \cite[Lemma 3.4]{NS-NLS} yields $\e_0(\de),\ka_0,\ka_1(\de)>0$ such that 
for any $u\in H^1_\rad$ with $S_S(u)<S_S(Q)+\e_0(\de)^2$  and $d_S(u):=\inf_{\te\in\R}\|u-e^{-i\te}Q\|_{H^1}\ge\de$, 
we have 
\EQ{ \label{NLS K bd}
 \CAS{K(u) \ge 0 \implies K(u)\ge\min(\ka_1(\de),\ka_0\|\na u\|_2^2),
  \\ K(u)<0 \implies K(u) \le -\ka_1(\de).}}
To transfer the distance condition to the Zakharov $d_Q$, suppose that $d_S(u_s)\le 1$. 
Since $d_Q(\vec u)^2 \sim d_0(\vec u)^2=d_S(u_s)^2+\|u_w-Q^2\|_2^2$, 
$\|u_w-Q^2\|_2 \le \|u_d\|_2+\||u_s|^2-Q^2\| \lec \|u_d\|_2 +d_S(u_s)$, 
$S_Z(\uu)=S_S(u_s)+\tf14\|u_d\|_2^2$ and $|S_S(u_s)-S_S(Q)|\lec d_S(u_s)^2$, 
we have $d_Q(\uu) \lec \|u_d\|_2+d_S(u_s)$. 
Hence there is a constant $c_0\in(0,1)$ such that 
\EQ{
  \pt S_Z(\uu)-S_S(Q) \le c_0 d_Q(\vec u)^2 
   \pr\implies \tf 14 \|u_d\|_2^2 \le S_Z(u)-S_S(Q)+Cd_S(u_s)^2
 \le \tf18\|u_d\|_2^2 + Cd_S(u_s)^2 
  \pr\implies d_Q(\uu)\sim d_S(u_s).} 
Then requiring $\e_V(\de)<c_0\de^2$ ensures the energy constraint $S_Z(\uu)<S_S(Q)+\e_V(\de)^2$ 
to imply $d_S(u_s) \ge c_0\de$ and hence, via \eqref{NLS K bd}, 
\EQ{ \label{Kus bd}
 \CAS{K(u_s) \ge 0 \implies K(u_s)\ge\min(\ka_1(c_0\de),\ka_0\|\na u\|_2^2),
  \\ K(u_s)<0 \implies K(u_s) \le -\ka_1(c_0\de).}}
Choosing $\ka_V(\de)\le\ka_1(c_0\de)$ and $c_K\le\ka_0$, we obtain the desired estimate on $K(u_s)$. 

In the case of (1), we have 
\EQ{ \label{KZ2K}
 K_Z^a(\vec u)-\tf{2-a}{12}\|u_{d1}\|_2^2 - \tf{a}{4}\|u_{d2}\|_2^2
  \pt\ge K(u_s)+\tf{2-a}{6}\|u_{d1}\|_2^2-\tf12\|u_{d1}\|_2\|u_s\|_4^2
  \pr\ge K(u_s)-C\|u_s\|_4^4 \ge K(u_s)-C\|u_s\|_2\|\na u_s\|_2^3,}
where the last inequality is by Gagliardo-Nirenberg. 
Also note that $K(u_s)\ge 0$ with the uniform bound on $S_S(u_s)\le S_Z(\uu)$ implies a uniform bound on $u_s$ in $H^1$. 
Hence if $\|\na u_s\|_2\ll 1$, then the lower bound $K(u_s)\ge \ka_0\|\na u_s\|_2^2$ in \eqref{Kus bd} 
yields a similar bound on $K_Z^a(\uu)$. 
Otherwise, $K(u_s)$ is bounded below by \eqref{Kus bd}, which is dominant in \eqref{KZ2K} if $\|u_{d1}\|_2\ll 1$. 
Otherwise, requiring $\e_V(\de)^2\le\tf{1}{12}\|u_d\|_2^2$ yields
\EQ{
 S_S(u_s)+\tf16\|u_d\|_2^2=S_Z(\vec u)-\tf{1}{12}\|u_d\|_2^2 \le S_S(Q),}
and hence Lemma \ref{lem:var} with $\nu:=\tf{2}{\sqrt{6}}\|u_{d1}\|_2$ implies 
that the first line of \eqref{KZ2K} is non-negative, 
so $\|u_{d1}\|_2^2$ on the left side of \eqref{KZ2K} yields a lower bound on $K_Z^a(\uu)$. 
Thus we obtain the desired bound in the case (1).

Next we consider the case (2) of $K(u_s)<0$, 
where the last argument does not work, since Lemma \ref{lem:var} does not have any room in the coefficient $a+1$ 
(in contrast to $\sqrt{6}>2$ in the previous case). 
Suppose that the conclusion fails. Then there exist $\de>0$ and a sequence $\vec u_n\in\HH$ satisfying $K(u_{ns})<0$, 
\EQ{
 \pt \lim_{n\to\I} S_Z(\vec u_n) \le S_S(Q), \pq \sup_{n\in\N}d_Q(\vec u_n)\ge\de,
 \pq \lim_{n\to\I} K_Z^a(\vec u_n)=0,}
where $a\in[-1,1]$ may depend on $n$ but convergent. Since
\EQ{ \label{def G}
G^a_Z(\vec u):=& S_Z(u)-\tf{2}{5-a}K_Z^a(\vec u)\\
=& \tf{1}{2}\norm{u_s}_2^2+\tf{1-a}{10-2a}\|\na u_s\|_2^2+\tf{5-3a}{20-4a}\|u_{w2}\|_2^2+\tf{1+a}{20-4a}\|u_{w1}\|_2^2
}
is bounded on $\vec u_n$, we see that $\vec u_n$ is bounded in $\HH$ if $|a|<1$. 
At $a=\pm 1$, they are degenerate, but we may consider another linear combination, e.g., 
\EQ{
  \pt S_Z(\vec u)-\tf14K_Z^a(\vec u)
   \pr=\tf14\|\na u_s\|_2^2+\tf12\|u_s\|_2^2+\tf{2+a}{16}\|u_{w1}\|_2^2+\tf{4-a}{16}\|u_{w2}\|_2^2 - \tf{3+a}{16}\LR{u_w||u_s|^2},}
where the last term is bounded by $\|u_w\|_2\|u_s\|_2^{1/2}\|\na u_s\|_2^{3/2}$.
Combining its bound with that by $G^a_Z$, we see that $\{\vec u_n\}_n\subset\cH$ is bounded also for $a=\pm 1$. 
Hence, after extracting a subsequence if necessary, we may assume that $\vec u_n$ weakly converges to some $\vec u_\I\in\HH$. Then by compactness of the radial Sobolev embedding, $u_{ns}\to u_{\I s}$ strongly in $L^4$, and so
\EQ{
 S_Z(\vec u_\I) \le S_S(Q),\pq G^a_Z(\vec u_\I) \le S_S(Q), \pq K(u_{\I s}) \le 0.}
If $u_{\I s}=0$, then $K(u_{n s})<0$ requires $K(u_{n s})\to 0$, $\|\na u_{ns}\|_2\to 0$, but then $K(u_{ns})\ge 0$ for large $n$ 
and contradiction. Hence $u_{\I s}\not=0$. 
If $K(u_{\I s})=0$, then $S_S(u_{\I s})\le S_Z(\vec u_\I)\le S_S(Q)$ and the variational characterization of $Q$ 
imply that $u_{\I s}=e^{-i\te}Q$ for some $\te\in\R$, and $u_{\I d}=0$, namely $\vec u_\I=\vec Q_{1,\te}$. 
Then we have the norm convergence and so the strong convergence of $\vec u_n\to \vec u_\I$ in $\cH$, 
but then $d_Q(\vec u_\I)\ge\de>0$, a contradiction. 

Hence $K(u_{\I s})<0$, and Lemma \ref{lem:var} with $\nu=\|u_d\|_2$ implies $K_Z^a(\vec u_\I)<0$. Let 
\EQ{
 \la_*:=\inf\{\la_0\in[0,1] \mid \la_0\le \la \le 1 \implies K_Z^a(\la u_\I)\le 0\}. }
Then $\la_*\in[0,1)$ and $K_Z^a(\la u_\I)\le 0$ for $\la_*\le\la<1$, and
\EQ{
 S_S(\la u_{\I s}) \le S_Z(\la \vec u_\I) = (G_Z^a + \tf{2}{5-a}K_Z^a)(\la \vec u_\I)
  \le G_Z^a(\la \vec u_\I) < G_Z^a(\vec u_\I) \le S_S(Q). }
Hence the variational property of $Q$ for NLS implies $K(\la u_{\I s})<0$ if $0<\la<1$ and $\la\ge\la_*$, 
and $K_Z^a(\la u_\I)<0$ by Lemma \ref{lem:var}.  
Since the $K<0$ region is away from $0$, it implies $\la_*>0$, 
but then $K_Z^a(\la_* \vec u_\I)<0$, contradicting the definition of $\la_*$.  
\end{proof}

Combining the results in the previous section, we can define the sign function $\fS$.
\begin{lem}\label{lem:signS}
There exists a constant $\de_S\in(0,\de_X]$ such that the following holds, where $\de_X$ is given in Lemma 5.2. 
Let $0<\delta\leq \delta_S$ and 
\EQ{
\HH_\delta:=\{\vec u\in \HH \mid S_Z(\vec u)<S_Z(Q)+\min(d_Q(\vec u)^2/2, \varepsilon_V(\delta)^2)\}
}
where $\varepsilon_V(\delta)$ is given by Lemma \ref{lem:epV}.  Then there exists a unique continuous function $\fS:\HH_\delta \to \{\pm 1\}$ satisfying
\EQ{
\fS(\vec u)=
\CAS{
-\sign \lambda_1, \quad \mbox{ if } d_Q(\vec u)\leq \delta_E;\\
\sign K(u_s), \quad \mbox{ if } d_Q(\vec u)\geq \delta,
}
}
where we set $\sign 0=+1$. 
Moreover, there is an absolute constant $B_H\in(0,\I)$ such that 
for all $\uu\in\HH_\de$ with $\fS(\vec u)=+1$, we have $\|\uu\|_\HH<B_H$. 
\end{lem}
Lemma \ref{lem:deltaX} gives a lower bound 
$-\sign(\la_1) K(u_s) \ge C_1d_Q(u)-C_2R$ for some constants $C_1,C_2>0$. 
Choosing $\de_S\le\de_X$ such that $C_1\de_X>C_2\de_S$ ensures its positivity when $d_Q(\uu)$ is close enough to $\de_X$, 
and then by continuity $-\sign(\la_1)=\sign K(u_s)$ on the whole intersection $\de\le d_Q(\uu)\le\de_E$. 
The uniform bound $B_H$ is obtained from \eqref{def G} with $K_Z^a\ge 0$ and $a=1/3$ for example: 
\EQ{
 \tf{1}{14}\|\uu\|_\HH^2 \le G_Z^{1/3}(\uu) \le S_Z(\uu) \le S_S(Q)+\e_V(\de_S)^2. }

\subsection{Variational estimate with cut-off}
For $K>0$, we will need the cut-off inside $K_Z$, acting on the variables $u,u_{d1},\p_0 b$ as 
\EQ{
 K_{\chi_1,\chi_2,\chi_3}^a \pt:=K_Z^a(\chi_1 u,\chi_2 u_{d1},\chi_3\p_0b)
 \pr= \|\chi_1 \na u\|_2^2-\tf{d}{4}\|\chi_1 u\|_4^4
 +\LR{\tf{|u|^2}{r^2}|\chi_1(2+\p_\ro^*)\p_\ro\chi_1}
 \prq+\tf{2-a}{4}\|\chi_2u_{d1}\|_2^2+\tf{a}{4}\|\chi_3\p_0b\|^2-\tf{d+a-2}{4}\LR{|u|^2u_{d1}|\chi_1^2\chi_2},}
for which we need a variational estimate. 
For $0\le a<\min(2,d-2)$,  
\EQ {
 S_Z(\cS^a_\la\vec u)=\tf{\|\na u\|_2^2}{2}\la^2+\tf{\|u\|_2^2}{2}
 +\tf{\|u_{w1}\|_2^2}{4}\la^{2-a}+\tf{\|u_{w_2}\|_2^2}{4}\la^a -\tf12\LR{u_{w1}||u|^2}\la^{\frac{d+2-a}{2}},}
as a function of $\la\in(0,\I)$, is non-decreasing if $\LR{u_{w1}||u|^2}\le 0$, 
otherwise it has a unique maximal point $\la=\la_*>0$, since 
\EQ{
 \tf{d+2-a}{2}> 2=\max(2,2-a,a) \ge \min(2,2-a,a) \ge 0,}
and then $K_Z^a(\cS^a_\la\vec u)$ is positive for $0<\la<\la_*$ and negative for $\la>\la_*$. 
Moreover, $G_Z^a(\cS^a_\la\vec u)$ is non-decreasing for $\la>0$, 
where 
\EQ{
 G_Z^a \pt:=S_Z-\tf{2}{d+2-a}K_Z^a
 \pr=\tf{d-2-a}{2(d+2-a)}\|\na u\|_2^2+\tf12\|u\|_2^2
  \pn+\tf{d-2+a}{4(d+2-a)}\|u_{w1}\|_2^2+\tf{d+2-3a}{4(d+2-a)}\|u_{w2}\|_2^2
 \pr=\tf{d-2-a}{2(d+2-a)}\|\na u\|_2^2+\tf12\|u\|_2^2 +\tf{d-2+a}{4(d+2-a)}\|u\|_4^4
  \prq+\tf{d-2+a}{4(d+2-a)}\|u_{d1}\|_2^2+\tf{d+2-3a}{4(d+2-a)}\|u_{b2}\|_2^2 +\tf{d-2+a}{2(d+2-a)}\LR{u_{d1}||u|^2}
}
was defined in \eqref{def G}. 
Hence we have the variational characterization of the Zakharov ground state for $0\le a<\min(2,d-2)$
\EQ{ \label{min of G}
 \pt\inf\{G_Z^a(\vec u)\mid \vec u\in\cH,\ (u,u_{w1},au_{w_2})\not=0,\ K_Z^a(\vec u)\le 0\}
 \pr=\inf\{S_Z(\vec u) \mid \vec u\in\cH,\ u\not=0,\ K_Z^a(\vec u)=0\} 
 = S_Z(\vec Q)=S_S(Q).}
\begin{lem} \label{lem:Gest}
For $0\le a<1$ and $0<\de<1$, there is $\ka_G(\de)>0$ such that 
for all $\uu\in\cH$ satisfying $G_Z^a(\uu)<S_S(Q)-\de$, we have 
$K_Z^a(\uu) \ge \ka_G(\de)(\|\na u\|_2^2+\|u_{d1}\|_2^2+a\|\p_0b\|_2^2)$. 
\end{lem}
\begin{proof}
The bound on $G_Z^a$ implies a uniform bound $\|\uu\|_\cH\lec 1$. 
If $\|\na u\|_2\ll 1$ is small enough, then 
the Gagliardo-Nirenberg $\|\fy\|_4^4\lec\|\na\fy\|_2^3\|\fy\|_2$ 
implies that the nonlinear part of $K_Z^a$ is much smaller, 
namely $\|u\|_4^4\ll\|\na u\|_2^2$ and $|\LR{u_{d1}||u|^2}|\le\|u\|_4^2\|u_{d1}\|_2
 \ll \|\na u\|_2^2+\|u_{d1}\|_2^2$, so we obtain the desired estimate. 

If it fails for $\|\na u\|_2\sim 1$, then there is a sequence $\uu_n\in\cH$ 
satisfying $G_Z^a(\uu_n)<S_S(Q)-\de$, $\|\na u_n\|_2\sim 1$ and $K_Z^a(\uu_n)\to 0$. 
The uniform bound in $\cH$ allows us to replace it with a subsequence weakly convergent to some $\uu_\I\in\cH$, 
and the radial Sobolev embedding implies the strong convergence of the nonlinear parts $\|u\|_4^4$ 
and $\LR{u_{d1}||u|^2}$. 
Hence the weak lower semicontinuity of norms implies $G_Z^a(\uu_\I)\le S_S(Q)-\de$ 
and $K_Z^a(\uu_\I)\le 0$. 
Then \eqref{min of G} implies $(u_{\I s},u_{\I w1},au_{\I w2})=0$, 
but the strong convergence of the nonlinear part with $\|\na u_n\|_2\sim 1$ implies 
$K_Z^a(\uu_n)\sim 1$ for large $n$, a contradiction. 
\end{proof}

We also need the weighted radial Sobolev
\begin{lem}
There is an absolute constant $C>0$ such that for all $\te\in[0,1]$, all $\chi\in C^1((0,\I))$ 
satisfying $0\le\chi\le 1$ and $(1+\p_\ro)\chi\ge 0$, and any radial $u\in H^1(\R^3)$, we have 
\EQ{ \label{wei-L4}
 \LR{|u|^4|\chi^2 r^\te} \le C\|u\|_2^{1+\te}\|u_r\|_2^{1-\te}\LR{|u_r|^2|\chi^2}.}
\end{lem}
For NLS \cite{NS-NLS}, it was used only with $\te=1$. It seems we need a smaller $\te>0$ for the Zakharov. 
As long as $\te$ is positive, the rescaled version of the above \eqref{wei-L4} with $m>0$ 
\EQ{ \label{wei-L4-m}
 \LR{|u|^4|\chi^2(\tf{r}{m})(\tf{r}{m})^\te} \lec m^{-\te}\|u\|_{H^1}^2\LR{|u_r|^2|\chi^2(\tf{r}{m})}}
yields a uniform decay as $m\to\I$ for $0<\te\le 1$ and bounded $u\in H^1_r$. 
\begin{proof}
It follows immediately from 
\EQ{
 \pt \int_0^\I |u|^4 r^\te \chi^2 r^2 dr = \int_0^\I f(s) \int_s^\I |u|^4 r^\te dr ds,
 \pr \int_s^\I |u|^4 r^\te dr  \lec \|u\|_2^{1+\te}\|u_r\|_2^{1-\te}\int_s^\I |u_r|^2 dr \pq(0\le\te\le 1),}
where $f:=\p_r(r\chi)^2=2r\chi(1+\p_\ro)\chi\ge 0$. 
Denoting $\|u\|_{>s}^2:=\int_s^\I|u(r)|^2dr$, the latter follows from partial integration and Schwarz
\EQ{
 \pt  |u(s)|^2s^\te = \int_s^\I 2\LR{u,u_r}s^\te dr \lec \|ur^\te\|_{>s}\|u_r\|_{>s} \lec \|u\|_2^\te \|u_r\|_2^{1-\te}\|u_r\|_{>s},
 \pr  \|u\|_{>s}^2 = \int_s^\I 2\LR{u,u_r}(r-s)dr \lec \|u\|_2 \|u_r\|_{>s},}
where we used $\|ru\|_{>0}\sim\|u\|_2$ and $\|u\|_{>0}\lec\|u_r\|_2$. 
\end{proof}
 
Going back to any solution $u\in\HH^\e$ of \eqref{eq:Zak2} in the region $d_Q(\vec u)\ge \de$ and $\e\le\e_V(\de)$ with $K(u)\ge 0$, 
the variational Lemma \ref{lem:epV} gives lower bounds on $K(u_s)$ and $K_Z^a(\uu)$. 
If $\|\na u\|_2\ll 1$, then the nonlinear part is negligible in $K_\chi^a \sim \|\na\chi_1 u\|_2^2+\|\chi_2 u_{d1}\|_2^2+a\|\chi_3\p_0 b\|_2^2$. 
If $\|\na u\|_2\sim 1$, then we have some $\ka_2(\de)\sim \ka_V(\de)>0$ such that
\EQ{
 G_Z^a(\vec u)=S_Z(\vec u)-\tf{2}{d+2-a}K_Z^a(\vec u) < S_S(Q)-\ka_2(\de),}
provided that $\e>0$ is small enough depending on $\de>0$. 

Let $\ck G_Z^a(u,u_{d1},\p_0 b):=G_Z^a(\uu)$. 
By the cut-off $(u,u_{d1},\p_0 b)\mapsto(\chi_1 u,\chi_2 u_{d1},\chi_3 \p_0 b)$, 
all the terms in $\ck G_Z^a$ (with $0\le a<d-2$) decrease except 
\EQ{ \label{hardy-wei}
 \pt \|\na(\chi_1 u)\|_2^2 = \|\chi_1\na u\|_2^2 + \LR{\tf{|u|^2}{r^2}|f_H},  
 \pq f_H:=-\chi_1(d-2+\p_\ro)\p_\ro\chi_1.}
If $|f_H|\le B r^\te$ for some $B>0$ and $\te\in(0,2]$ then  
\EQ{
 |\LR{\tf{|u|^2}{r^2}|f_H(\tf xm)}| \le B\LR{\tf{|u|^2}{r^2}|(\tf rm)^\te} \lec B m^{-\te}\|u\|_{H^1}^2,}
which is uniformly small as $m\to\I$. 
Hence, choosing $m\gg 1$ large enough depending on $\de$, we obtain 
\EQ{
 \ck G_Z^a(\chi_1(\tf xm) u,\chi_2(\tf xm) u_{d1},\chi_3(\tf xm) u_{b2}) < S_S(Q)-\ka_2(\de)/2,}
then Lemma \ref{lem:Gest} yields
\EQ{
 K_\chi^a \ge \ka_G(\ka_2(\de)/2)(\|\na(\chi_1(\tf xm) u)\|_2^2+\|\chi_2(\tf xm) u_{d1}\|_2^2+a\|\chi_3(\tf xm) \p_0 b\|_2^2).}
Moreover, \eqref{hardy-wei} together with Hardy 
implies $\|\chi_1\na u\|_2\lec (1+B)\|\na(\chi_1u)\|_2$, 
provided that $|\p_\ro\chi_1|+|\p_\ro^2\chi_1|\le B\chi_1$.
Thus we conclude 
\begin{lem}\label{lem:varcut}
For any $\te,B>0$ and $a\in[0,1)$, there exist non-decreasing functions $\e_\chi,\ka_\chi:(0,1)\to(0,1)$ 
and a non-increasing function $m_\chi:(0,1)\to(1,\I)$ such that for any $\chi=(\chi_1,\chi_2,\chi_3)\in C^2(0,\I)\times C^0(0,\I)^2$ 
satisfying $0\le \chi_j\le 1$ for $j=1,2,3$ and $|\p_\ro\chi_1|+|\p_\ro^2\chi_1|\le B\min(r^\te,|\chi_1|)$, 
for any $\de\in(0,1)$, $m\ge m_\chi(\de)$, and $\vec u\in\cH$ satisfying 
$S_Z(\vec u)<S_S(Q)+\e_\chi(\de)^2$, $d_Q(\vec u)\ge \de$ and $K(u)\ge 0$, we have
\EQ{\label{eq:varcut}
 K_\chi^a(\vec u) \ge \ka_\chi(\de)(\|\chi_1(\tf xm)\na u\|_2^2+\|\chi_2(\tf xm) u_{d1}\|_2^2+a\|\chi_3(\tf xm) \p_0 b\|_2^2). }
\end{lem}
Note that the dependence of $\e_\chi,\ka_\chi,m_\chi$ on $\chi$ is only through $\te,B$. 
Actually, $\e_\chi$ is independent of $\te,B$. 
The dependence on $a$ is uniform except the limit $a\to 1-0$, which follows from the fact that both sides of \eqref{eq:varcut} are affine in $a$. 
Anyway, the dependence on $\te,B,a$ does not matter in our argument, as they are all fixed to some specific values, 
as well as the cut-off functions.

\section{One-pass theorem}

In this section, we prove the one-pass theorem. 
The theorem roughly states that if the solution starts and leaves away from a neighborhood of $\cQ_1$, then it can not return. To prove that, we need to combine the hyperbolic structure close to $\cQ_1$ presented in Lemma \ref{lem:deltaX}  with the variational structure away from $\cQ_1$ given in Lemma \ref{lem:epV}. 
The crucial ingredients are the localized virial estimate proved in Section 2 and to prove its monotonicity. The main difficulty in proving the monotonicity lies at the facts that we do not have any precompactness of the solution in time, nor any room to use the radial Sobolev embedding for the wave component to control the cut-off remainder terms. Fortunately enough, by carefully choosing the cut-off functions and analyzing the error estimates, we can absorb all the cut-off remainders by the leading virial estimates in Lemma \ref{lem:varcut}.

\begin{thm}\label{thm:onepass}
There exist constants $\e_*,R_*>0$ with $\e_*<R_*/2<\de_X/4$ such that the following hold, where $\de_X>0$ is given by Lemma \ref{lem:deltaX}. 
Let $\vec u\in C([0,T^*);\HH)$ be a forward maximal solution of \eqref{eq:Zak2} in $\cH^\e$ satisfying $d_Q(\vec u(0))<R$ for some $\e\in (0,\e_*]$ and $R\in (2\e,R_*]$. Then one has either \textup{(1)} or \textup{(2)} of the following. 
Let $d_u(t):=d_Q(\vec u(t))$. 
\begin{enumerate}
    \item $T^*=\infty$ and $d_u(t)<R$ for all $t\geq 0$.
    \item There exist $t_*,T_X\in(0,T^*)$ such that $d_u(t)<R+R^2$ on $0\le t<t_*$, $d_u(t)$ is increasing for $t_*\le t\le T_X$ from $d_u(t_*)=R+R^2$ to $d_u(T_X)=\de_X$, and $d_u(t)>R_*$ on $T_X<t<T^*$.
\end{enumerate}
In the latter case, $\fS(\uu(t))\in \{\pm 1\}$ does not change sign on $t_*\leq t<T^*$. 
\end{thm}
\begin{proof}
We choose $\e_*, R_*>0$ together with other constants $\de_*\in(0,\de_S)$ and $q\in(0,1/2]$ such that
\EQ{
 \pt 0<\e_*<\min(\e_V(\de_*),\e_\chi(\de_*)), \pq 2\e_*< R_* < \min(1/C_*,\de_*/2,1/m_\chi(\de_*)),
  \prQ \de_*<\de_X/C_0, \pq R_*< \min(e^{-C_0/\ka_V(\de_*)},\ka_\chi(\de_*)^{1/(2q)}/C_0),}
for some large constant $C_0>1$, 
where $C_*$, $\e_V$ and $\delta_S$ are as in Lemmas \ref{lem:deltaX}, \ref{lem:epV} and \ref{lem:signS}, 
and $\e_\chi,m_\chi,\ka_\chi$ are as in Lemma \ref{lem:varcut} for some constants $B,\te>0$ (specifically we may choose $B=2$, $\te=1$). 
$q$ is a power in our specific weight in the scattering case, cf.~\eqref{est chi-mu}. 
It may be fixed to $q=1/6$, but we keep it as a parameter to clarify its role and freedom. 
The requirements for smallness of $R_*,\de_*$, or for largeness of $C_0$, are specified later at \eqref{vir in B}, \eqref{small R*-bup} and \eqref{small R*-scat}. 
Although we have several conditions, none of them are competing each other. 
In practice, we should fix the constants in the order 
\EQ{
 \de_X,\de_S \to \de_* \to m,\ell \to R_* \to \e_*,}
where the large parameters $m,\ell>1$ will appear in spatial cut-off arguments. 
Intuitively, $\de_X$ and $\delta_*$ define the threshold for the distance to $\cQ_1$. 
If $R\leq d_u(t)\leq \de_X$, then we use Lemma \ref{lem:deltaX}, and if $R\leq d_u(t)\ge\delta_*$, then we use Lemma \ref{lem:epV}. The two regions have ``large" overlapping (in the logarithmic scale) because of $\de_*\ll\de_X$. 

The proof is by contradiction, where the main target is the last part of (2), namely $d_u(t)>R_*$ for $T_X<t<T^*$, 
while the other part of (2) follows from the ejection Lemma \ref{lem:deltaX}. Let
\EQ{
 \pt T_0:=\inf\{t\in[0,T^*) \mid d_u(t) \ge R\},  \pq T_X:=\inf\{t\in[0,T^*) \mid d_u(t) \ge \de_X\}.}

Suppose that (1) does not hold. Then we have $T_0\in(0,T^*)$, $d_u(T_0)=R$, $d_u(t)<R$ for $0\le t<T_0$, 
and $\p_t d_u(T_0)\ge 0$. Then Lemma \ref{lem:deltaX} applied from $t=T_0$ implies $T_0<T_X<T^*$, $d_u(T_X)=\de_X$, 
$d_u\sim e^{\mu(t-T_0)}R$ on $[T_0,T_X]$ and increasing for $T_0+T_*R\le t\le T_X$. Let 
\EQ{
 T_1:=\sup\{t\in[T_0,T_X] \mid d_u(t)<R_*\}, \pq t_*:=\inf\{t\in[T_0,T_X] \mid d_u(t)\ge R_*+R_*^2\}.}
Then $T_0<T_1<t_*<T_X$. Since $R_*+C_*R_*^3<R_*+R_*^2$, we deduce that 
$d_u(t)$ is increasing from $d_u(t_*)=R_*+R_*^2$ to $d_u(T_X)=\de_X$ on $[t_*,T_X]$, 
$d_u(t)<R_*+R_*^2$ on $[T_0,t_*]$, and $d_u(t)\ge R_*$ on $[T_1,T_X]$. 
If (2) holds, then let $T_3:=T^*$. Otherwise, let 
\EQ{
 T_3:=\inf\{t\in[T_X,T^*) \mid d_u(t)\le R_*\} \in (T_X,T^*).}
Then $d_u(T_3)=R_*$ and $\p_t d_u(T_3)\le 0$. In both cases, we have $d_u(t)>R_*$ on $T_1<t<T_3$. 
Hence by Lemma \ref{lem:signS}, $s:=\fS(\vec u(t))\in \{\pm 1\}$ is constant on $(T_1,T_3)$. 

\bigskip

{\bf Step 1.} There exist disjoint open intervals $I_1, I_2\etc \subset (T_1,T_3)$, such that 
\begin{itemize}
\item on each $I_j$, there exists $\tau_j\in \bar I_j$ such that for $t\in I_j$
\EQ{ \label{hyperbolic intervals}
 d_u(t) \sim e^{\mu|t-\tau_j|}d_u(t_j), \quad s K_Z(\vec u(t))\ges (e^{\mu|t-\tau_j|}-C_*)d_u(t_j).
}
\item on $I'=(T_1,T_3)\setminus \bigcup I_j$, we have $d_u(t)\geq \de_*$ and hence by Lemma \ref{lem:epV}
\EQ{ \label{K bd pos}
 K_Z^a(\vec u) \geq \min(\ka_V(\de_*),c_K[\|\na u_s\|_2^2+\tf{a}{4}\|u_{w2}\|_2^2+\tf{2-a}{4}\|u_{w1}\|_2^2])} 
if $s=+1$ and $0\le a\le 1$, and 
\EQ{ \label{K bd neg}
 K_Z^a(\vec u) \le -\ka_V(\de_*)}
if $s=-1$ and $|a|\le 1$. 
\end{itemize}

Indeed, Lemma \ref{lem:deltaX} implies \eqref{hyperbolic intervals} for $I_1:=[T_1,T_X]$ with $\ta_1:=T_1$. 
If $T_3<T^*$, then applying the lemma backward from $t=T_3$, we obtain the last interval $I_n$ with $\ta_n:=T_3=\max I_n$. 
For any $\tau\in (T_1,T_3)$ where $d_u(\ta)\in(R_*,\de_*)$ is a local minimum, we can apply the lemma both forward and backward in time, obtaining an open interval $I_\tau\subset (T_1,T_3)$ so that $d_u(\partial I_\tau)=\{\delta_X\}$ and for $t\in I_\tau$
\EQ{
 d_u(t) \sim e^{\mu |t-\tau|}d_u(\tau), \quad s K^a_Z(\vec u(t))\ges (e^{\mu |t-\tau|}-C_*)d_u(\tau).
}
Then all those intervals including $I_1$ and $I_n$ above 
are mutually disjoint with length $\log(\de_X/\de_*)\lec |I_j| \lec \log(\de_X/R_*)$, and $d_u(t)\ge\de_*$ in the complement. 

\bigskip

{\bf Step 2.} Virial estimate in the case $s=-1$.

For $2-d<a<\min(d-2,\tf{d+2}{3})$, 
$G_Z^a$ in \eqref{def G} is equivalent to the energy norm 
\EQ{
 G^a_Z(\uu) := S_Z(\uu)  - \tf{2}{d+2-a}K_Z^a(\uu) \ge c_{a,d} \|\vec u\|_\cH^2,
}
for some constant $c_{a,d}>0$ depending on $a,d$. 
Then we can bound $K_Z^a$ from above 
\EQ{
 K_Z^a(\uu) \le -\tf{d+2-a}{2}[c_{a,d}\|\vec u\|_\cH^2 - S_S(Q)-\e^2]}
for all $\vec u\in\cH^\e$. Combined with \eqref{K bd neg}, it implies $K^a_Z(\vec u) \lec -\ka_V(\de_*)(\|\vec u\|_\cH^2+1)$ 
for $|a|<1$ and $t\in I'$, where the implicit constant depends on $a$.  
To fix the constants, we now choose $a=0$. Then 
\EQ{ \label{K<0 est}
 S_Z(\vec u)<S_S(Q)+\e_V(\de_*)^2,\ d_Q(\vec u)\ge \de_* \implies K^0_Z(\vec u) \lec -\ka_V(\de_*)(\|\vec u\|_\cH^2+1).}

To localize the virial identity in this case, we choose $\chi=\y(\log(r/m))$ for some $\y\in C^\I(\R)$ and $m,\ell\gg 1$ depending on $\de_*>0$, satisfying 
\EQ{
 \y' \le 0\le\y\le 1, \pq \y(\ro)=\CAS{1 &(\ro\le 2),\\ 0 &(\ro \ge \ell)},
 \pq |\y'|+|\y''|+|\y'''|\lec 1/\ell \ll \ka_V(\de_*).}
We have from \eqref{vir0} 
\EQ{
 \p_t V_\chi^0 (\tilde \uu) \pt=4K_Z^0-\LR{4|u_r|^2+2u_{d1}^2|\chi^C-\chi_\ro}
 \prq- \LR{\tf{|u|^2}{r^2}|(\p_\ro^*+2)\p_\ro^*\p_\ro \chi} +\tf12\LR{\z^2|(\p_\ro-d)\up_\ro} 
 \prq-\LR{|u|^4|\p_\ro^*\chi^C+\up^C} +\LR{|u|^2n|\up^C}-\LR{|u|^2\z|\up_\ro},}
with $\up:=(1-\p_\ro)\chi$, $\zeta=x\cdot \bb/|x|^2$, where we expanded $u_{d1}=n-|u|^2$ in the term of $|u|^2u_{d1}$. 

In the region \eqref{K<0 est}, we may discard the leading quadratic terms of $|u_r|,u_{d1}^2$ since $\chi^C-\chi_\rho>0$. 
Since the cut-off of $|u/r|^2$ is supported on $r>2m$, 
that term is bounded by $m^{-2}\|u\|_2^2 \lec m^{-2}$ 
(the assumption that $\vec u$ is once close to $\vec Q$ implies $M(u)\sim M(Q)\sim 1$). 
Similarly we can dispose of the $|u|^4$ term by the radial Sobolev 
$H^1_r\subset r^{-1}L^\I$
\EQ{
 |\LR{|u|^4|\p_\ro^*\chi^C+\up^C}|
 \lec \|u\|_{L^4(r>2m)}^4 \lec m^{-2}\|u\|_2^3\|\na u\|_2 \lec m^{-2}\|\vec u\|_\cH,}
and also 
\EQ{
 |\LR{|u|^2n|\up^C}-\LR{|u|^2\z|\up_\ro}|
 \lec \|u\|_{L^4(r>2m)}^2\|n\|_2 \lec m^{-1}\|\vec u\|_\cH^2,}
where we used the Hardy inequality $\|\z\|_2\lec\|\p_\ro^*\z\|_2=\|n\|_2$. Indeed,
\EQ{
 \pt\int_0^\I |\fy|^2 r^{d-1}dr = \int_0^\I |r^d\fy|^2 r^{-d-1}dr
 = [|r^d\fy|^2\tf{r^{-d}}{-d}]_{+0}^\I - \int 2\re \ba{r^d\fy}(\p_r r^d\fy)\tf{r^{-d}}{d} dr
 \pr\le \lim_{r\to+0}\tf{1}{d}r^d|\fy(r\te)|^2
 +\tf{2}{d}\sqrt{\int_0^\I |\fy|^2 r^{d-1}dr\int_0^\I |r^{1-d}\p_rr^d\fy|^2 r^{d-1} dr}.}
Since the boundary term is zero, we thus obtain, for any $\fy\in L^2(\R^d)$ with $\p_\ro\fy\in L^2(\R^d)$, 
\EQ{
 \int_{\R^d} |\fy|^2 dx \le (\tf{2}{d})^2 \int_{\R^d} |\p_\ro^*\fy|^2 dx.}
For the term of $\z^2$, we use the Hardy inequality with the smallness of the derivatives of $\y$: 
\EQ{
 |\LR{\z^2|(\p_\ro-d)\up_\ro}| \lec \|\z\|_2^2\|(\p_\ro-d)\up_\ro\|_\I
 \lec \|\vec u\|_\cH^2/K.}
Thus they are all absorbed by the variational estimate \eqref{K<0 est} by taking $m$ and $K$ sufficiently large, namely 
\EQ{ \label{m cond-B}
 1/m+1/\ell \ll \ka_V(\de_*).}
Then we have, at all $t\in I'$,
\EQ{\label{eq:monovirial}
\p_t V_\chi^{0}(\ti u) \lec -\ka_V(\de_*).
}

Now suppose for contradiction that both (1) and (2) fail. Then we get from Step 1 that $T_3<T^*$ and 
\EQ{ \label{vir in B}
 V_\chi^{0}(\ti u)(T_1)-V_\chi^{0}(\ti u)(T_3) 
 \gec \sum_{j=1}^n \int_{I_j} (e^{\mu |t-t_j|}-C_*)d_u(t_j)dt + \int_{I'} \ka_V(\de_*)dt \gec \delta_X,}
where the exponential integrals on $I_j$ are positive and $\sim\de_X$, 
provided that $\tf{\inf d_u(I_j)}{\sup d_u(I_j)} \le \de_*/\de_X\ll 1$ is small enough. 

On the other hand, since $d_u(t)=R_*$ at $t=T_1,T_3$ and $Q$ has exponential decay at infinity, 
we get for $t=T_1,T_3$ that
\EQ{ 
 |V_\chi^a(\ti u)(t)|=|\LR{-iu|B_0^\chi u}+\tf{1}{\al}\LR{\bb|B_a^\chi\p_0 \bb} |\lec R_*+m e^\ell R_*^2 \sim R_*,
}
provided that $R_*\le e^{-\ell}/m$, then we get a contradiction between \eqref{vir in B} and \eqref{small R*-bup}, 
thus finishing the proof in the case of $\fS(\uu(T_X))=-1$. 
The largeness requirement \eqref{m cond-B} on $m,\ell$ is translated to the smallness of $R_*$: 
\EQ{ \label{small R*-bup}
 R_* \le e^{-C/\ka_V(\de_*)},}
for a large constant $C>0$. 
Note that this exponential smallness is needed to treat the cut-off error of the wave component, so it differs from the NLS case \cite{NS-NLS}. It is somewhat natural as the wave part is only an obstruction for the virial-blowup argument, and we have only the homogeneous norm ($\|u_w\|_{L^2_x}$) for the wave component.  

\bigskip

{\bf Step 3.} Virial estimate in the case $s=1$.

In the region $K>0$, to apply the virial is much more delicate. First we do not have uniform lower bound similar to \eqref{eq:monovirial}, as it goes to $0$ when the solution is scattering. Second, the cut-off quadratic term has negative sign, which is opposite to $K$, so we can not simply drop them. We need to use the cut-off variational estimate given in Lemma \ref{lem:varcut} to absorb all the cut-off remainder terms. 
The situation is similar to NLS \cite{NS-NLS}, but it is even harder because of the cubic nonlinear energy 
$|u|^2n$, which has slower order of decay in terms of the main quadratic energy, 
than the NLS energy $|u|^4$. 

In particular, the simple cut-off weight $1/(1+r)$ in \cite{NS-NLS} does not seems to work, 
because the cut-off error of $|u|^2n$ becomes bigger than the main quadratic part,  
at least for general radial functions. 
It is also because of the degeneration of quadratic term for the weight $1/|x|$, 
which is common among the Morawetz-type estimates. 
Hence in order to maximize the use of \eqref{eq:varcut}, we need to design carefully the weight $\chi$. 
It should be asymptotic to $1/|x|$, but the difference from $1/|x|$ plays a crucial role. 

Now we choose a specific $\chi$ for the Zakharov system. 
For any $\mu(r)\ge 0$, a non-negative solution $\chi(r)$ of 
\EQ{
 (1+\p_\ro)\chi = \mu^2}
is given by $\chi(r) = \tf 1r\int_0^r \mu^2(s)ds \ge 0$. 
If $\mu$ is bounded and decreasing, then so is $\chi$, 
and $\chi\sim 1/r$ as $r\to\I$ iff $\int_0^\I \mu^2(r)dr\in(0,\I)$.  
For explicit computation, we choose specifically  
\EQ{ \label{est chi-mu}
 \pt \mu:=(1+r)^{-1/2-q} \sim \chi^{1/2+q}, 
 \pq \chi:=\tf1r\int_0^r (1+s)^{-1-2q}ds =\tf{1-(1+r)^{-2q}}{2q r}\sim \tf{1}{1+r},
 \pr \chi-\mu^2=\tf1r\int_0^r[(1+s)^{-1-2q}-(1+r)^{-1-2q}] ds
 \prQ=(1+2q) \tf1r\int_0^r \tf{t}{(1+t)^{2+2q}}dt  \ge (\tf12+q)\tf{r}{(1+r)^{2+2q}} = -\tf12\p_\ro(\mu^2),
}
for a fixed $q\in(0,1/2]$, whose smallness will be specified below. 
$q=1/6$ is enough, but $q=1/2$ corresponding to $\chi=\tf{1}{1+r}$ seems too large for the following argument to work.  
With the specific choice of $\chi$ and $a=0$, we have $\up=(d-2-\p_\ro)\chi$ and 
\EQ{
 \p_t V_\chi^0 \pt= 4K_{\mu,\chi_2,0}^0  + 2\LR{u_{d1}^2|\mu^2-\chi_2^2} +\tf12\LR{\z^2|(\p_\ro-d)\up_\ro}
 \prq-\LR{\tf{|u|^2}{r^2}|(\p_\ro^*+2)\p_\ro\p_\ro^*\chi+2\p_\ro(2+\p_\ro^*)\mu^2+4|\mu_\ro|^2}
 \prq+\LR{|u|^4|d\mu^4-(d+\p_\ro)\chi} -\LR{|u|^2u_{d1}|\up-\mu^2\chi_2}-\LR{|u|^2\z|\up_\ro},}
where we should choose $\chi_2<\mu$ (later), in order for the remaining $u_{d1}^2$ term to dominate the $\z^2$ term.  
The coefficient of $|u/r|^2$ is in general (using $(1+\p_\ro)\chi=\mu^2)$
\EQ{
 \pt(d-3)\p_\ro^*\p_\ro\chi-(4+\p_\ro^*)\p_\ro\mu^2-4|\mu_\ro|^2 
 \pr= -2(\mu\mu_\ro+|\mu_\ro|^2-\mu\mu_{\ro\ro})+(d-3)\p_\ro(\p_\ro^*\chi+\mu^2).}
For $\mu=(1+r)^{-1/2-q}$ and $d=3$, it is equal to
\EQ{
 (1+2q)(\tf{r}{1+r})^2\mu^2.} 
The coefficient of $\z^2$ term is estimated 
\EQ{
 (\p_\ro-3)\up_\ro \pt= (\p_\ro-3)\p_\ro(1-\p_\ro) \chi = [8-(8-5\p_\ro+\p_\ro^2)(1+\p_\ro)]\chi 
 \pr\ge (\p_\ro-\p_\ro^2)(\mu^2) 
 \pn=-(1+2q)(2+2q)(\tf{r}{1+r}\mu)^2, }
where we used \eqref{est chi-mu}. 
The coefficients of $|u|^4$, $|u|^2u_{d1}$ and $|u|^2\z$ are all bounded by
\EQ{
 \pt |3\mu^4-(3+\p_\ro)\chi| +  |\up-\mu^2\chi_2| + |\up_\ro| \lec \tf{r}{(1+r)^2},}
provided that $\chi_2=1+O(r)$ as $r\to+0$ and $\chi_2\lec 1$. 
Thus we obtain, using Lemma \ref{lem:varcut} after the rescaling $\chi\mapsto \chi(\tf{r}{m})$ 
(and also for $\mu,\chi_2$ in the same way), 
\EQ{
 \p_t V_\chi^0  \pt\ge 4\ka_\chi(\de_*)(\|\mu\na u\|_2^2+\|\chi_2 u_{d1}\|_2^2)+ 2\LR{u_{d1}^2|\mu^2-\chi_2^2}
  \prq+(1+2q)\LR{|\tf{u}{r}|^2|(\tf{r/m}{1+r/m}\mu)^2}
   -(1+2q)(1+q)\LR{\z^2|(\tf{r/m}{1+r/m}\mu)^2}
  \prq-C\LR{|u|^4+|u|^2(|u_{d1}|+|\z|)|\tf{r/m}{(1+r/m)^2}},}
for some absolute constant $C>0$. 
We may check the condition of the lemma for $\chi$ with $B=2$ and $\te=1$, 
using the equations $\chi_\ro=\mu^2-\chi\le 0$ and \eqref{est chi-mu}. 
Indeed, we have $|\chi_{\ro\ro}|\le|\chi_\ro|\le\min(r,\chi)$. 
Denoting $\La_m:=\tf{r/m}{(1+r/m)^2}$ and noting that
\EQ{
\Lambda_1^{2q} \sim \min(r^{2q},r^{-2q}) 
 \les (r^{2q}+r) \mu(r)^2,
}
then using \eqref{wei-L4-m}, we have for $m\ge 1$
\EQ{
 \pt \LR{|u|^4|\La_m^{2q}} \lec m^{-2q}\|u\|_2^2 \|\mu u_r\|_2^2, 
 \pr \LR{|u|^2(|u_{d1}|+|\z|)|\La_m} \lec \sqrt{\LR{|u|^4|\La_m^{2q}} \LR{u_{d1}^2+\z^2|\La_m^{2-2q}}},}
which are absorbed by the leading quadratic part for large $m$, provided that 
\EQ{ \label{Holder wei}
 \La_m^{2-2q} \lec \chi_2^2, \pq \LR{\z^2|\La_m^{2-2q}} \lec \LR{n^2|\La_m^{2-2q}}\lec \LR{u_{d1}^2|\chi_2^2}+\LR{|u|^4|\La_m^{2-2q}}, }
where the second inequality follows from $n=\p_\ro^*\z$ and the Hardy inequality 
(see below), and the last inequality follows from the others and $n=u_{d1}+|u|^2$. 
More precisely, the largeness requirement on $m$ is 
\EQ{ \label{m cond-S}
 m\ge m_\chi(\de_*), \pq m^{-2q} \ll \ka_\chi(\de_*).}

For the remaining $\z^2$ term, it suffices to have  
the Hardy inequality (at $m=1$)
\EQ{ \label{hardy des}
 (1+2q)(1+q)\LR{\z^2|(\tf{r}{1+r}\mu)^2} \le 2\LR{|\p_\ro^*\z|^2|\mu^2-\chi_2^2}, \pq \mu^2-\chi_2^2 \lec \La_1^{2q}.} 
To see the constant of the Hardy inequality, a direct computation yields
\EQ{
 \LR{\z^2|f}  = \LR{(r^3\z)^2|r^{-4}f}_{(0,\I)} 
 \pt= [-(r^3\z)^2 F]_0^\I - 2\LR{r^5\z\p_\ro^*\z| F}_{(0,\I)} 
 \pr\le 2\sqrt{\LR{\z^2|f} \LR{|\p_\ro^*\z|^2|r^6F^2/f}},} 
where $F(r):=\int_r^\I s^{-4}f(s)ds$ for general $f\ge 0$. 
In particular, $f:=\La_m^{2-2q}$ yields the second inequality in \eqref{Holder wei}. 
For \eqref{hardy des}, we choose 
\EQ{
 f(r):=\tf{r}{1+r}\mu^2, \pq 
 F(r)=\int_r^\I s^{-3}(1+s)^{-2-2q}ds \le \tf12 r^{-3}f(r),}
and so we obtain, for all $q>0$ (actually $q=0$ is not critical here),  
\EQ{ \label{hardy-z}
 \LR{\z^2|\tf{r}{1+r}\mu^2}  \le \LR{|\p_\ro^*\z|^2|\tf{r}{1+r}\mu^2}. }
Now we choose $\chi_2\ge 0$ such that 
\EQ{
 \chi_2^2 = \mu^2 - (1-q)\tf{r}{1+r}\mu^2.}
Requiring $q\le 1/4$, we have 
\EQ{
 \mu^2-\chi_2^2 \le \La_1, 
 \pq \chi_2^2/q \ge \mu^2 \ge (1+r)^{-1-2q} \ge \La_1^{2-2q}.}
In order to get \eqref{hardy des} from \eqref{hardy-z}, we need 
\EQ{
 (1+2q)(1+q) \le 2(1-q),}
which holds for small $q>0$ (for example $q=1/6$). 

Therefore, we have at all $t\in I'$
\EQ{
\p_t V_\chi^0\geq \ka_\chi(\de_*) (\|\mu\na u\|_2^2+\|\chi_2 u_{d1}\|_2^2)\geq 0.
}
Hence we get from Step 1 that if both (1) and (2) fail then
\EQ{ \label{vir in S}
 V_\chi^{0}(\ti u)(T_3)-V_\chi^{0}(\ti u)(T_1) \gec   \sum_{j=1}^n \int_{I_j} (e^{\mu |t-t_j|}-C_*)d_u(t) dt \gec \de_X,
}
where the exponential integrals are positive as in the previous case $s=-1$ if $\de_*/\de_X\ll 1$ is small enough. 
On the other hand, as in Step 2, we get for $t=T_1,T_3$
\EQ{ \label{small R*-scat}
|V_\chi^a(\ti u)(t)|=|\LR{-iu|B_0^\chi u}+\tf{1}{\al}\LR{\bb|B_a^\chi\p_0 \bb} |\lec R_*+mR_*^2 \sim R_*,}
provided that $R_*\le 1/m$, contradicting \eqref{vir in S}. 
The largeness condition \eqref{m cond-S} on $m$ is translated to the smallness of $R_*$: 
\EQ{
 R_* \le 1/m_\chi(\de_*), \pq R_*\ll \ka_\chi(\de_*)^{1/(2q)}.}
Note that the high power $1/(2q)$ comes from the new cut-off function to handle the lower degree interactions, so it differs from the NLS case \cite{NS-NLS}, where $q=1/2$. 
Thus we complete the proof in both cases. 
\end{proof}

\section{Scattering and growup away from the ground state} \label{s:away}

In this section we prove the scattering-growup dichotomy for solutions staying away from the ground states. 
Actually the growup is immediate from the virial estimate, so the main task is to prove the scattering, 
combining the ideas and estimates in \cite{Zak-KM} and \cite{NS-NLS}. The main ingredients in \cite{Zak-KM} are the combinations of the method of normal form transform, radial generalized Strichartz estimates and Kenig-Merle's concentration compactness/rigidity method (\cite{KM}). 

\begin{prop}\label{prop:sc}
There exist constants $0<\e_*<R_*<\de_X$ such that the following holds. Let $\uu\in C([0,T^*),\HH)$ be a forward maximal solution of \eqref{eq:Zak2} satisfying $S_Z(\vec u)<S_S(Q)+\e_*^2$ and $d_Q(\vec u(t))>R_*$ for all $t\in[0,T^*)$. 
If $\fS(\uu(0))=+1$ then $T^*=\I$ and $\uu$ scatters as $t\to \infty$. 
If $\fS(\uu(0))=-1$ then $\limsup_{t\to T^*-0}\|\uu(t)\|_\HH=\I$. 
We have a similar result for $t\le 0$. 
\end{prop}
Those $\e_*,R_*$ may be regarded as the same as in Theorem \ref{thm:onepass}, 
though we add another smallness condition \eqref{small ep scat} on $\e_*$ in addition to those in the previous section. 
We assume that $R_*,\de_*$ are the same (or satisfy the same conditions) as there. 

First we dispose of the growup case $\fS(\uu(0))=-1$. 
If $T^*<\I$, then the local wellposedness implies $\|\uu(t)\|_\HH\to\I$ as $t\to T^*-0$, 
so we may assume $T^*=\I$. 
Define the intervals $I_j\subset(0,\I)$ and their complement $I'=(0,\I)\setminus\Cu I_j$ 
as in Step 1 of the proof of Theorem \ref{thm:onepass} with $T_1=0$ and $T_3=\I$. 
Then for all $T\in(0,\I)\cap I'$, we have in the same way as \eqref{vir in B}
\EQ{ \label{vir gup}
 V_\chi^0(\ti u)(T_1)-V_\chi^0(\ti u)(T) 
 \gec \sum_{I_j\subset (T_1,T)} \int_{I_j}(e^{\mu |t-t_j|}-C_*)d_u(t)dt 
  + \int_{I' \cap (T_1,T)} \ka_V(\de_*)dt,}
where the right side goes to $\I$ as $T\to\I$, while the left side is bounded by $me^\ell(\|\uu(T)\|_\HH+\|\uu(T_1)\|_\HH)^2$. 
Thus we obtain $\limsup_{t\to T^*}\|\uu(t)\|_\HH=\I$. 

In the rest of this section, we prove the scattering part of Proposition \ref{prop:sc}. 
We reproduce the notations and the resolution spaces used in \cite{Zak-KM}. Fix a small number $0<\io \ll 1$ and let
\EQ{ \label{def S}
 \pt \fX_1:=\LR{D}^{-1}L^2_t \dot B^{\fr 25-\io}_{q_s,2}, \pq \fY_1:= L^2_t \dot B^{-\fr 14-\io}_{q_w,2}, 
 \pr \tf{1}{q_s}:=\tf 3{10}-\tf\io3,
 \pq \tf{1}{q_w}:= \tf 14-\tf\io3,
 \pq S:= L^\I_t\HH \cap (\fX_1\times \fY_1),
 \pr X:=L^\I_t B^{-\fr 12-\io}_{\I,2}, \pq Y:=L^\I_t(\dot{B}^{-\fr 32-\io}_{\I,2}+\dot{B}^{-\fr 32+\io}_{\I,2}),
  \pq Z:=X\times Y.}
$S$ consists of the radially-improved Strichartz norms in $H^1\times L^2$ for the Schr\"odinger and wave equations,  
and we measure smallness of dispersed components by using $Z$. 
Those spaces are increasing in $\io>0$ and optimal as $\io\to+0$.  
For any solution $\uu\in C([0,T^*),\cH)$ of \eqref{eq:Zak2} with the maximal existence time $T^*\in(0,\I]$, 
the scattering as $t\to\I$ is equivalent to $\uu \in S([0,T^*))$ (which implies $T^*=\I$). 
The proof is similar to \cite[Lemma A.1]{Zak-KM}: 
If $\uu$ is scattering, then its $Z([T,\I))$ norm is decaying as $T\to\I$, which allows us to solve the equation 
in $S(T,T')$ uniformly for $T<T'<\I$, leading to $\uu\in S(T,\I)$. 
The other direction is obvious from the Duhamel estimate in \cite{GN}. 
Actually, for the scattering, smallness of $X(T,\I)$ norm of $u_s$ is enough in $\HH$-bounded regions, 
since all the nonlinear terms contain $u_s$. 
Then the Sobolev embedding $\dot H^1\subset \dot B^{-1/2}_{\I,2} \subset B^{-1/2-\de}_{\I,2}$ implies 
the following small-data scattering. 
\begin{lem} \label{lem:small-kine scat}
For any $B\in(0,\I)$, there exist $\e_S(B),C_S(B)>0$ such that for all initial data $\uu(0)\in\HH$ 
satisfying $\|\uu(0)\|_\HH\le B$ and $\|\na u_s(0)\|_2\le\e_S(B)$, 
the unique solution $\uu$ of \eqref{eq:Zak2} is global and scattering as $t\to\pm\I$ 
with a uniform bound $\|\uu\|_S\le C_S(B)$. 
\end{lem}
Using the above lemma, we impose another smallness on $\e_*$: 
\EQ{ \label{small ep scat}
  6\e_*^2< \min(\ka_V(\de_*),c_K\e_S(B_H)^2, S_S(Q)),}
where $\ka_V$, $c_K$ are given by Lemma \ref{lem:epV} and $B_H$ by Lemma \ref{lem:signS}, 
while $\de_*\in(0,\de_S)$ is the same as in the proof of Theorem \ref{thm:onepass}. 

Let $\sU$ be the collection of all solutions $\uu$ of \eqref{eq:Zak2} on $[0,\infty)$ satisfying for all $t\ge 0$
\EQ{ \label{def U}
S_Z(\vec u)<S_S(Q)+\e_*^2, \pq
d_Q(\vec u(t))\geq R_*, \pq \fS(\uu(t))=+1.
}
Lemma \ref{lem:signS} implies $\|\uu\|_\HH<B_H$ in $\sU$.
For each $E>0$, we define $N(E)$ as
\EQ{
N(E):=\sup\{\|\uu\|_{S([0,\infty))} \mid \uu \in \sU,\ S_Z(\uu)\leq E\}.
}
In \cite{Zak-KM} it was proved that $N(E)<\infty$ for $E<S_S(Q)$.  We will extend this upper bound to $S_S(Q)+\e_*^2$. 
We define
\EQ{
E^*=\sup \{E>0 \mid N(E)<\infty\}.
}
The main task is to show $E^*\ge S_S(Q)+\e_*^2$, which means all the solutions in $\sU$ are scattering\footnote{If $\sU$ is also bounded in $S([0,\I))$ then $E^*=\I$. Since $\sU$ is increasing in $\e_*>0$, it is indeed the case unless we choose the maximal possible $\e_*$, whose existence is unclear from our result.}
as $t\to\I$. 
By contradiction, we assume 
\EQ{\label{eq:Estar}
E^*=S_S(Q)+\e^2<S_S(Q)+\e_*^2
}
for some $0\leq \e<\e_*$. Then Proposition \ref{prop:sc} follows from the following lemmas.

\begin{lem}[Existence of critical element]\label{lem:existence}
Under the assumption \eqref{eq:Estar}, there exists a solution $\uu$ of \eqref{eq:Zak2} in $\sU$ satisfying 
$S_Z(\uu)=E^*$ and $\norm{\uu}_{S([0,\I)}=\infty$. We call such $\uu$ \emph{a critical element}. 
\end{lem}

\begin{lem}[Precompactness of critical elements]\label{lem:precomp}
If $\uu$ is a critical element of the above lemma, then $\{\uu(t)\mid t\ge 0\}$ is precompact in $\cH$. 
\end{lem}

\begin{lem}[Rigidity]\label{lem:rigidity}
The critical element of Lemma \ref{lem:existence} does not exist.
\end{lem}
\begin{proof}
Let $\uu$ be a critical element. We can show that it contradicts to the virial estimates as in the step 3 of the proof of Theorem \ref{thm:onepass}. Indeed, the proof is much easier thanks to the precompactness property: we can easily localize the virial estimates with decay of cut-off errors as $m\to\I$ uniformly in time, 
independent of the choice of cut-off. 
\end{proof}

It only remains to prove Lemmas \ref{lem:existence}--\ref{lem:precomp}. 
Here we need the profile decomposition and global perturbation lemma established in \cite{Zak-KM}. 

\begin{lem}[Profile decomposition {\cite[Lemma 4.2, Remark 4.1]{Zak-KM}}] \label{free prof}
For any bounded sequence $\{\vec f_n\}_{n\in\N}\subset\cH$, there
are a subsequence $\vec f_n'$, $\bar{J}\in \N_0 \cup \{\infty\}$, a
bounded sequence $\vec f^j\in\cH$, and sequences $\{t_n^j\}_{n\in\N}\subset\R$ for $j\in\N$ with $j\le\bar J$, 
such that the following holds: For any $j,n\in\N$ and $J\in\N_0$ with $j,J\le\bar{J}$, let 
\EQ{
 \pt \vec v_n(t):=U(t)\vec f_n', 
 \pq \vec v_n^j(t):=U(t-t_n^j)\vec f^j, 
 \pq \vec v_n^{>J} := \vec v_n - \sum_{j=1}^J \vec v_n^j.}
Then for any $j,k\in\{1 \ldots J\}$, we have
$t_\I^j:=\lim_{n\to\I} t_n^j \in \{0,\pm\I\}$,
\EQ{ \label{separation}
  j\not=k\implies \lim_{n\to\I}|t_n^j-t_n^k|=\I,}
\EQ{ \label{weak conv}
  \pt (\pe{\vec v_n})(t_n^j),  (\pe{\vec v_n})(0) \to 0\ \text{weakly in $\cH$} \text{ as } n\to \I,}
and
\EQ{ \label{smallness}
 \lim_{J\to \bar{J}}\limsup_{n\to\I}
  \|\vec v_n^{>J}\|_Z =0.}
Moreover,
\eqref{separation}--\eqref{weak conv} implies the linear orthogonality
\EQ{
 \pt \lim_{n\to\I}\|m\vec v_n(0)\|_{\cH}^2-\sum_{j=1}^J\|m\vec v_n^j(0)\|_{\cH}^2-\|m\pe{\vec v_n}(0)\|_{\cH}^2 = 0,}
for any bounded Fourier multiplier $m$ on $\cH$, as well as the nonlinear orthogonality
\EQ{ \label{eq:nonlinear orth}
 \pt \lim_{n\to\I} E_S(v_{ns}(0))-\sum_{j=1}^J E_S(v_{ns}^j(0))-E_S(\pe v_{ns}(0))=0,
 \pr \lim_{n\to\I} K(v_{ns}(0))-\sum_{j=1}^J K(v_{ns}^j(0))-K(\pe v_{ns}(0))=0,
 \pr \lim_{n\to\I} E_Z(\vec v_n(0))-\sum_{j=1}^J E_Z(\vec v_n^j(0))-E_Z(\pe{\vec v_n}(0))=0,
 \pr \lim_{n\to\I} K_Z^a(\vec v_n(0))-\sum_{j=1}^J K_Z^a(\vec v_n^j(0))-K_Z^a(\pe{\vec v_n}(0))=0,}
for any $a\in\R$. The same orthogonality holds also along $t=t_n^j$ instead of $t=0$.
\end{lem}

For a fixed $j$, such a sequence of free solutions $\{\lp{\vec v_n^j}\}_{n\in\N}$ is called a {\it free concentrating wave}. Given a free concentrating wave in the form
\EQ{
 \lp{\vec v}_n(t)=U(t-t_n)\lp{\vec f}, \pq t_\I=\lim_{n\to\I}t_n\in\{0,\pm\I\},
}
we associate to it the {\it nonlinear profile} $\np{u}$, defined as the solution of \eqref{eq:Zak2} satisfying 
\EQ{
 \vec u=U(t)\lp{\vec f} -i\int_{-t_\I}^t U(t-s)(u_s\re u_w,\al D|u_s|^2)(s)ds.
}
The nonlinear profile is obtained by solving the initial data problem (if $t_\I=0$)
or by solving the final data problem (if $t_\I=\pm\I$), see \cite{Zak-KM}. 
In the latter case, $\uu$ is scattering as $t\to-t_\I$. 
We call $\np{u}_n(t):=\np{u}(t-t_n)$ the {\it nonlinear concentrating wave} associated with
$\lp{\vec v}_n(t)$. By the above construction we have 
\EQ{ \label{NP init}
 \pt\|\lp{\vec v}_n(0)-\np{u}_n(0)\|_{\cH}
  =\|\np{u}(-t_n)-U(-t_n)\lp{\vec f}\|_{\cH} + o(1)\to 0}
as $n\to\I$. 

We need a time-local version of the global perturbation \cite[Lemma 4.3]{Zak-KM}. 
\begin{lem}[Global perturbation on intervals] \label{NL profile}
For each free concentrating wave $\lp{\vec v_n^j}$ in Lemma \ref{free prof}, let $\np u_n^j$ be the associated 
nonlinear concentrating wave. Let $\vec u_n$ be the sequence of nonlinear solutions of \eqref{eq:Zak2} with 
$\vec u_n(0)=\vec f_n'$ and $T_n\in(0,\I]$. If $\limsup_{n\to\I}\|\np u_n^j\|_{S([0,T_n))}<\I$ for all $j\in\N$ 
with $j\le\bar J$, then
\EQ{ \label{NP decop}
 \lim_{J\to\bar J}\limsup_{n\to\I}\|\vec u_n-\sum_{j=1}^J\np u_n^j-\vec v_n^{>J}\|_{S([0,T_n))}=0.}
\end{lem}
\begin{proof}
The case of $T_n\equiv \I$ was in \cite{Zak-KM} (the decomposition \eqref{NP decop} 
is not explicitly stated, but obvious from the proof in \cite[Section 4.3]{Zak-KM}). 

If $\{T_n\}$ is bounded and $t_\I^j=\pm\I$, then $\|\np u_n^j\|_{Z(0,T_n)}=\|\np u^j\|_{Z(-t_n^j,T_n-t_n^j)}\to 0$ as $n\to\I$ 
by the scattering construction of those profiles, so that we may ignore their contributions, 
except at most one profile with $t_\I^j=0$. 

If $T_n\to\I$, then the orthogonality $|t_n^j-t_n^k|\to\I$ (for $j\not=k$) implies that 
there is at most one $j$ such that $\limsup_{n\to\I}|t_n^j-T_n|<\I$, which implies $t_\I^j=\I$.  
Passing to a subsequence if necessary, we may assume that $T_n-t_n^j\to\pm\I$ for all the other $j$. 
If $t_\I^j=-\I$ or $T_n-t_n^j\to-\I$ (which may happen only if $t_\I^j=\I$), then $\|\np u_n^j\|_{Z(0,T_n)}\to 0$ as above. 
If $t_\I^j\in\{0,\I\}$ and $T_n-t_n^j\to\I$, then $\limsup_{n\to\I}\|\np u_n^j\|_{S(0,T_n)}<\I$ 
implies $\|\np u^j\|_{S(0,\I)}<\I$, so $\np u^j$ is scattering as $t\to\I$. 

Thus in all the cases we obtain, for all $j\not=k$, 
\EQ{ 
 \CAS{t_n^j-t_n^k \to \I \implies \|\np u_n^j\|_{Z((0,T_n)\cap(-\I,(t_n^j+t_n^k)/2))} \to 0,\\
  t_n^j-t_n^k \to -\I \implies \|\np u_n^j\|_{Z((0,T_n)\cap((t_n^j+t_n^k)/2,\I))} \to 0.}}
This is what we need to discard interactions among profiles in the proof of \cite[Lemma 4.9]{Zak-KM}, 
and the rest of proof works in the same way as in the case of $T_n\equiv\I$. 
\end{proof}

\begin{proof}[Proof of Lemma \ref{lem:existence}]
By the definition of $E^*$, we know then there exists a sequence of $\uu_n\in \sU$ satisfying 
\EQ{\label{eq:criticalapprox}
S_Z(\uu_n)\to E^*, \quad \norm{\uu_n}_{S([0,\I))}\to \infty.}
Then Lemma \ref{lem:small-kine scat} together with Lemma \ref{lem:signS} implies for large $n$, 
\EQ{
 \inf_{t\ge 0}\|\na u_{ns}(t)\|_2 > \e_S(B_H). }
Hence the variational lower bound on $K(u_s)$ 
in Lemma \ref{lem:epV} becomes bigger than $6\e_*^2$ under the condition \eqref{small ep scat}.  
Then using Lemma \ref{lem:trap} and translation in $t$, we may assume 
\EQ{ \label{d-K bd}
d_Q(\uu_n(0))\geq \delta_X, \quad K(\uu_n(0))>6\e_*^2.
}
Now apply the profile decomposition in Lemma \ref{free prof} to the sequence $\{\uu_n(0)\}$.  We get
a family of free concentrating waves $\{\lp{\vec v_n^j}\}_{n\in\N}$, the associated nonlinear profiles $\{\np u^j\}$ with maximal intervals of existence $I^j$ and nonlinear concentrating waves $\{\np u_n^j\}_{n\in\N}$. 
We will show that $\np u^j=0$ for all $j\ge 2$ as well as
\EQ{ \label{rem vanish}
 \lim_{n\to \infty} \norm{\vec v_n^{>1}}_{\HH}=0.
}

Since \eqref{d-K bd} implies $G(u_{ns}(0))<S_S(Q)-\e_*^2$, where 
\EQ{
 G(\fy):=S_S(\fy)-\tf13 K(\fy)=\tf16\|\na \fy\|_2^2+\tf12\|\fy\|_2^2,}
the orthogonality \eqref{eq:nonlinear orth} with \eqref{NP init} implies $G(\np u^j_{ns}(0))<S_S(Q)-\e_*^2/2$ 
for all $j$ and large $n$, 
then the variational characterization of $Q$
\EQ{
 S_S(Q) = \inf\{G(\fy) \mid H^1\ni\fy\not=0,\ K(\fy)\le 0\} }
implies $K(\np u^j_{ns}(0))\ge 0$, hence $0\le S_S(\np u^j_{ns}(0)) \le S_Z(\np u^j)$. 
Then using the orthogonality again, we obtain from $S_Z(\uu_n)\to S_S(Q)+\e^2$
\EQ{
0\leq S_Z(\np u^j)\leq S_S(Q)+\e^2, \quad \forall j\geq 1.}

If $S_Z(\np u^j)< S_S(Q)$ for all $j$, then the result in \cite{Zak-KM} implies scattering of $\np u^j$ with $\|\np u^j\|_S<\infty$ for each $j$,  
and hence Lemma \ref{NL profile} implies boundedness of $\norm{\uu_n}_S$, contradicting \eqref{eq:criticalapprox}.
Therefore, there exists at least one $\np u^j$, say $\np u^1$, satisfying $S_S(Q)\leq S_Z(\np u^1)\leq E^*$. Then $S_Z(\np u^j)\le \e^2<S_S(Q)$ for all $j\geq 2$, due to the orthogonality \eqref{eq:nonlinear orth} for $S_Z$ together with $S_Z(\np u^j)\ge 0$. 
So those $\np u^j$ are small scattering solutions for $j\geq 2$. 
On the other hand, $\np u^1$ cannot scatter as $t\to\I$ because of Lemma \ref{NL profile}. 
Hence by definition of $E^*$, we have $S_Z(\np u^1)=E^*$, 
and the orthogonality implies that $\vec v_n^j=\np u_n^j\equiv 0$ for all $j\ge 2$ and we have \eqref{rem vanish}.  
To say that $\np u^1$ is a critical element (up to a time translation), 
it remains to show $\np u^1\in \sU$, namely \eqref{def U}. We argue case by case. 

{\bf Case 1.} $t_\infty^1=-\infty$. 
This cannot happen, as it means scattering of ${\np u^1}$ as $t\to\I$. 

{\bf Case 2.} $t_\infty^1=\infty$. 
In this case, ${\np u^1}$ scatters as $t\to -\infty$. 

\pq {\bf Case 2a.} $d_Q(\np u^1(t))>2\e_*$ for all $t\in I^1$.
Then $\fS(\np u^1)=+1$ for all $t\in I^1$, by continuity of the sign function (Lemma \ref{lem:signS}). 
Therefore the one-pass Theorem \ref{thm:onepass} implies that 
$\np u^1$ exists globally and satisfies $d_Q(\np u^1(t))>R_*$ for large $t>T'$. Then $\np u^1(t+T')\in \sU$ and so $\np u^1(t+T')$ is a critical element.

\pq {\bf Case 2b.} There exists $t_*\in I^1$ such that $d_Q(\np u^1(t_*))\leq 2\e_*$. 
Then for any $T\in I^1$ with $T>t_*$, we have 
$\|\np u_n^1\|_{S((-\I,T-t_n^1))}=\|\np u^1\|_{S((-\I,T))}<\I$, 
so Lemma \ref{NL profile} implies that $\uu_n(t_*-t_n^1)=\np u^1(t_*)+o(1)$ in $\HH$ as $n\to\I$. 
This cannot happen as $d_Q(\uu_n(t))\geq R_*>2\e_*$ for all $t\ge 0$ and $n$, 
while $t_n^1\to t_1^\I=-\I$. 

{\bf Case 3.} $t_\infty^1=0$. 
In this case, we have $d_Q(\np u^1(0))> \delta_*$ and so $\fS(\np u^1(0))=+1$ since $K(\np u^1(0))\geq 0$.

\pq {\bf Case 3a.} $d_Q(\np u^1(t))>2\e_*$ for all $t\geq 0$.
This is similar to Case 2a.  We have $\fS(\np u^1)=+1$ for $t\geq 0$. Therefore $\np u^1$ exists globally and satisfies $d_Q(\np u^1(t))>R_*$ for large $t>T'$. Then $\np u^1(t+T')$ is a critical element.

\pq {\bf Case 3b.} $d_Q(\np u^1(t_*))\leq 2\e_*$ for some $t_*>0$.
This is similar to Case 2b. For any $T_0,T_1\in I^1$ with $T_0<0<t_*<T_1$, we have 
$\|\np u_n^1\|_{S([T_0-t_n^0,T_1-t_n^0))}=\|\np u^1\|_{S([T_0,T_1))}<\I$, 
so Lemma \ref{NL profile} implies that $\uu_n(t_*-t_n^1)=\np u^1(t_*)+o(1)$ in $\HH$ as $n\to\I$, 
contradicting $d_Q(\uu_n(t))\ge R_*>2\e_*$ for all $t\ge 0$ and $n$, while $t_n^1\to 0$. 

Thus in all cases, we obtain a critical element $\np u^1(t-T')$ for some $T'\in\R$.  
\end{proof}

Finally we show the precompactness property.
\begin{proof}[Proof of Lemma \ref{lem:precomp}] 
We apply the above argument to the sequence of the solutions $\{\uu(t+\tau_n)\}$ for arbitrary $\tau_n\to \infty$. 
Then there exists only one nonlinear profile $\np u^1$ with $t^1_\I\in\{0,\I\}$. 
If $t_\infty^1=\infty$, then $\uu(t+\tau_n)$ is bounded in $S((-\I,0])$, 
contradicting the fact that $\norm{\uu}_{S([0,\infty))}=\infty$. So $t_\infty^1=0$, and hence 
$\uu(\ta_n)\to\vec f$ in $\HH$ along a subsequence of $\ta_n$ 
for some $\vec f\in\HH$, which proves that $\{\uu(t)\}_{t\geq 0}$ is precompact in $\HH$.
\end{proof}

\section{Manifold of trapped solutions}
To prove that the set of trapped solutions is a Lipschitz manifold in the energy space  
of codimension $1$, we argue in two parts:
\begin{enumerate}
\item Existence is proven by the elementary topological argument, 
using the stability of the ejection process: Since solutions exiting from the neighborhood 
have $\la_+ \sim\pm\de_X$, this property is stable for initial perturbation,  
and both the sign cases $\pm$ exist, 
there must be another case within the connected set of $\la_+(0)$, 
namely trapped solutions without exiting. 
Since it requires to tune only one parameter $\la_+(0)\in\R$, the set is at least of codimension $1$. 
\item Uniqueness and regularity for the above choice of $\la_+(0)$ 
(with respect to the other part of initial data) is proven 
by exponential growth of the difference of two solutions 
whose difference is not initially small in $\la_+$. 
Higher regularity may be shown by the same property of the derivatives 
(with respect to the initial data).  
\end{enumerate}
The first part is immediate from the dynamic descriptions of the ejection process. 
The second part is essentially a linearized version of that, 
but the energy argument is not enough to control the difference of the dispersive part $\vga$, 
because the conservative structure is destroyed 
and the energy norm is too weak to control the nonlinear terms. 
Hence we need to use the bilinear estimate as in \cite{Zak-KM}, 
which is based on the normal form and the radial improved Strichartz estimate. 
As we do not prove the scattering to the ground states, 
it suffices to have estimates locally in time, so the potential terms may be treated as regular perturbation 
(without time decay). 
We still need to show that the dispersive part does not grow 
as fast as the unstable mode, for which we estimate the evolution of the linearized energy 
using those Strichartz norms and the non-resonance in duality.

\subsection{Bilinear estimates}
We recall the bilinear estimates of \cite{GN} using the (radial) Strichartz estimates. 
In addition to the radial endpoint Strichartz norms $\fX_1,\fY_1$ in \eqref{def S}, we use 
\EQ{
 \pt \fX_0:=L^\I_t H^1_x, \pq \fY_0:=L^\I_t L^2_x, 
 \pr \|u\|_{\fX_+}:=\|i\dot u-\De u\|_{L^2_t(L^2\cap L^{3/2})_x}, 
 \pq \|N\|_{\fY_+}:=\|i\dot N+\aD N\|_{L^2_tL^{3/2}_x}, 
 \pr \fX:=\{u\in \fX_0\cap \fX_1 \mid \|u\|_{\fX_+}<\I\}, \pq \|u\|_{\fX}:=\|u\|_{\fX_0\cap\fX_1}+\|u\|_{\fX_+}, 
 \pr \fY:=\{u\in \fY_0\cap \fY_1 \mid \|u\|_{\fY_+}<\I\}, \pq \|u\|_{\fY}:=\|u\|_{\fY_0\cap\fY_1}+\|u\|_{\fY_+},
 \pr \fZ_0:=\fX_0\times \fY_0=L^\I_t\cH, \pq \fZ_1:=\fX_1\times \fY_1, \pq \fZ:=\fX\times \fY.}
The following linear estimate is immediate from the radial Strichartz estimate 
\begin{lem} \label{lem:lin}
For any radial $\fy\in H^1_\rad(\R^3)$, $\psi\in L^2_\rad(\R^3)$, and radial $f(t,x):I\times\R^3\to\C$ on any interval $I\subset\R$ and $t_0\in I$, we have 
\EQ{
 \pt \|\int_{t_0}^t e^{-i\De(t-s)}f(s)ds\|_{\fX(I)}
  \lec \|f\|_{L^2_t(L^{3/2}\cap L^2)_x(I)\cap L^1_tH^1_x(I)},
 \pr \|e^{-it\De}\fy\|_{\fX} \lec \|\fy\|_{H^1}, \pq  \|e^{it\aD}\psi\|_{\fY} \lec \|\psi\|_{L^2}.}
\end{lem}
The above Duhamel estimate is used only for interactions with solitons where we have regularity room. 
For the main bilinear interactions, we have 
\begin{lem} \label{lem:bil}
For any radial functions $u(t,x),v(t,x),N(t,x):I\times\R^3\to\C$ 
on interval $I=[0,T)$ for any $T\in(0,\I]$, we have 
\EQ{
 \pt\|\int_0^t e^{-i\De(t-s)}(u\re N)(s)ds\|_{\fX(I)} 
 \prq\lec \|N\|_{\fY(I)}(\|u\|_{(\fX_1\cap \fX_+)(I)}+\min(1,T^{1/6})\|u\|_{\fX_0(I)})
   +\|N(0)\|_{L^2_x}\|u(0)\|_{H^1_x},}
and
\EQ{
 \pt\|\int_0^t e^{\aD(t-s)}\aD(u\bar v)(s)ds\|_{\fY(I)}
 \prq\lec \min_{j=0,1}\|u\|_{\fX_j(I)}\|v\|_{\fX_{1-j}(I)} + \|u\|_{(\fX_1+\fX_+)(I)}\|v\|_{(\fX_1+\fX_+)(I)}
 \prQ+ \min(1,T^{1/2})\|u\|_{\fX(I)}\|v\|_{\fX(I)} + \|u(0)\|_{H^1_x}\|v(0)\|_{H^1_x}.}
\end{lem}
Note that the right sides are bounded simply by $\|N\|_{\fY(I)}\|u\|_{\fX(I)}$ and 
$\|u\|_{\fX(I)}\|v\|_{\fX(I)}$ respectively, 
but the above refined version is useful for the potential perturbation for small $T>0$. 
\begin{proof}
Let $b_1(N,u):=\int_0^t e^{-i\De(t-s)}(Nu)(s)ds$ 
and $b_2(u,v):=\int_0^t e^{i\aD(t-s)}\aD(u\bar v)(s)ds$. 
The estimate on $b_1(\bar N,u)$ is the same as $b_1(N,u)$ and omitted. 
The equation part of the norm on the left side is easily bounded by 
\EQ{
 \pt\|b_1\|_{\fX_+} = \|Nu\|_{L^2_t(L^{3/2}\cap L^2)_x} \le 
 \|N\|_{L^\I_tL^2_x}\|u\|_{L^2_t(L^6\cap L^\I)_x} \lec \|N\|_{\fY_0} \|u\|_{\fX_1},
 \pr\|b_2\|_{\fY_+} = \|\aD(u\bar v)\|_{L^2_t L^{3/2}_x}
 \lec \|u\|_{L^2_tW^{1,6}_x}\|v\|_{L^\I_tH^1_x} \lec \|u\|_{\fX_1}\|v\|_{\fX_0},}
while those norms are needed for the normal form argument, which is needed only for the interactions with the wave component $N$ 
in high frequency. Let $(fg)_{XL}$, $(fg)_{LX}:=(gf)_{XL}$, $\Om(f,g)$ and $\ti\Om(f,g)$ be the bilinear Fourier multipliers as in \cite[Section 2]{GN}. 
Roughly speaking, $(fg)_{XL}$ is the high-low interaction avoiding the resonant frequency $\sim\al$, 
while $\Om(f,g)\approx\LR{D}^{-1}D^{-1}(fg)_{XL}$ and $\ti\Om(f,g)\approx\LR{D}^{-1}D^{-1}(fg)_{XL+LX}$ 
are obtained by partial integration on the bilinear phase in time. Let 
\EQ{
 b_{1XL}:=\int_0^t e^{-i\De(t-s)}(Nu)_{XL}(s)ds,
 \pq b_{2XL}:=\int_0^t e^{i\aD(t-s)}\aD(u\bar v)_{XL+LX}(s)ds.}
Then, as in \cite[Lemmas 3.5--3.7]{GN}, we have
\EQ{
 \|b_{1XL}\|_{\fX_0\cap \fX_1} \pt\lec \|\Om_b(N,u)\|_{\fX_0\cap \fX_1}  
 \prq+\|\int_0^t e^{-i\De(t-s)}[\Om(i\dot N+\aD N,u)+\Om(N,i\dot u-\De u)]ds\|_{\fX_0\cap \fX_1},}
where $\Om_b(N,u):=\Om(N,u)-e^{-it\De}\Om(N,u)(0)$ and the second term is bounded by
\EQ{
 \pt\|\Om(i\dot N+\aD N,u)\|_{L^1_tH^1_x} \lec \|N\|_{\fY_+}\|u\|_{\fX_1},
 \pr\|\Om(N,i\dot u-\De u)\|_{L^2_tH^{1,6/5}_x} \lec \|N\|_{\fY_0} \|u\|_{\fX_+}.}
The boundary term is bounded in $\fX_1$ by 
\EQ{
 \|\Om_b(N,u)\|_{\fX_1} \lec \|N\|_{\fY_0}\|u\|_{\fX_1} + \|N(0)\|_{L^2_x}\|u(0)\|_{H^1_x}.}
The $\fX_0$ norm may be bounded in the same way, 
but to get a smallness factor for short time, we use instead 
\EQ{
 \|\Om(f,g)\|_{H^{s_0+s_1+1/2}_x} \lec \|f\|_{H^{s_0}}\|g\|_{H^{s_1}}, \pq 
 s_0\in[-\tf12,0],\ s_1\in[\tf12,1].}
Then we have at each $t>0$, exploiting the regularity room, 
\EQ{
 \pn\|\Om_b(N,u)\|_{H^1_x}
 \pt\lec \|\Om(N,u)-\Om(N,u)(0)\|_{H^1_x} 
 \pn+ \|(e^{-it\De}-1)\Om(N,u)(0)\|_{H^1_x}
 \pr\lec \|N(t)-N(0)\|_{H^{-1/2}_x}\|u(0)\|_{H^1_x}
 + \|N(t)\|_{L^2_x}\|u(t)-u(0)\|_{H^{1/2}_x}
 \prq+\min(1,t^{1/4})\|\Om(N,u)(0)\|_{H^{3/2}_x},}
where the time difference is bounded by using $L^{3/2}_x\subset\dot H^{-1/2}_x$ 
and interpolation of $H^s_x$, 
\EQ{
 \pt \|N(t)-N(0)\|_{H^{-1/2}_x} \lec \min(1,t^{1/2})(\|N\|_{\fY_0}+\|N\|_{\fY_+}),
 \pr \|u(t)-u(0)\|_{H^{1/2}_x} \lec \min(1,t^{1/6})(\|u\|_{\fX_0}+\|u\|_{\fX_+}).}
Thus we obtain 
\EQ{
 \pt\|\Om_b(N,u)\|_{\fX_0}
 \pn\lec \min(1,T^{1/6})\|N\|_{\fY_0\cap \fY_+}\|u\|_{\fX_0\cap \fX_+}.}
For the other term, we have similarly 
\EQ{
 \|b_{2XL}\|_{\fY_0\cap \fY_1} \pt\lec \|D\ti\Om_b(u,v)\|_{\fY_0\cap \fY_1}
 \prq+\|\int_0^t De^{i\aD(t-s)}[\ti\Om(i\dot u-\De u,v)+\ti\Om(u,i\dot v-\De v)]ds\|_{\fY_0\cap \fY_1},}
where $\ti\Om_b(u,v):=\ti\Om(u,v)-e^{it\aD}D\ti\Om(u,v)(0)$ and the second term is bounded by 
\EQ{
 \pt\|D\ti\Om(i\dot u-\De u,v)+D\ti\Om(u,i\dot v-\De v)\|_{L^1_tL^2_x}
 \pn\lec \|u\|_{\fX_+}\|v\|_{\fX_1} + \|u\|_{\fX_1}\|v\|_{\fX_+},}
while the boundary term is bounded by
\EQ{
 \|D\ti\Om_b(u,v)\|_{\fY_1} \pt \lec \min_{j=0,1}\|u\|_{\fX_j}\|v\|_{\fX_{1-j}} + \|u(0)\|_{H^1_x}\|v(0)\|_{H^1_x}, 
 \\ \|D\ti\Om_b(u,v)\|_{\fY_0}
 &\lec \|u(t)-u(0)\|_{L^2_x}\|v\|_{\fX_0}+\|u\|_{\fX_0}\|v(t)-v(0)\|_{L^2_x}
 \prq+\min(1,t^{1/2})\|D\ti\Om(u(0),v(0))\|_{H^1_x}
 \pr\lec \min(1,t^{1/2})\|u\|_{\fX_0\cap \fX_+} \|v\|_{\fX_0\cap \fX_+}. }

The remaining parts of $b_1,b_2$, namely $b_1-b_{1XL}$ and $b_2-b_{2XL}$, 
are treated in the same way as \cite[Lemmas 3.2--3.3]{GN}. 
To gain the power of $T\in(0,1)$ for the resonant terms, we use a different exponent $1/q_L:=1/4+\ka/3$ 
and the radial Strichartz estimate 
\EQ{
 \|\int_0^t e^{-i\De(t-s)}(Nu)_{\al L}(s)ds\|_{\fX(I)}
 \pt\lec \|(Nu)_{\al L}\|_{L^2_tL^{q_L'}_x(I)} \pr\lec \|N\|_{\fY_1(I)}\|u\|_{L^6_tL^2_x(I)}
  \lec T^{1/6}\|N\|_{\fY_1(I)}\|u\|_{\fX_0(I)},}
using that $(Nu)_{\al L}$ is restricted to high-low interactions with high frequencies around $\al\sim 1$, and similarly 
\EQ{
 \|\int_0^t e^{i\aD(t-s)}(u\bar v)_{\al L}(s)ds\|_{\fY(I)}
  \pt\lec \|(u\bar v)_{\al L}\|_{(L^2_tL^{3/2}_x\cap L^1_tL^2_x)(I)}
   \pr\lec T^{1/2}\|u\|_{\fX_0(I)}\|v\|_{\fX_0(I)},}
and $(u\bar v)_{L\al}$ by symmetry. 
\end{proof}

\subsection{Difference estimates}
Consider two solutions $\vec u^j=\Psi_{\te_j}(\vv^j)$ of \eqref{eq:Zak2} near the ground state $\cQ_1$, 
namely $d_Q(\vec u^j)<\de_E$ on a time interval $I$, where $\te_j \in C^1(I)$ is given by Lemma \ref{lem:orth} 
such that $\om(\vec v^j,\vec Q_{(\la)})=0$ for each $j=0,1$. 
Then $\vec v^0,\vec v^1$ solve \eqref{eq v}, where the modulation part may be linearized as 
\EQ{
 \pt J \ti N(\vv) = J N(\vv) + \vec Q_{(\te)} M_{(\te)} \vga + \ck N(\vv), 
 \pq M_{(\te)}\fy := 2M(Q)^{-1}\om(\fy,\vec Q_{(\te)}),} 
where $\ck N(\vv)=(\ck N_1(\vv),0)$ is defined for small $\|\vv\|_\HH$, satisfying 
\EQ{
 \|\ck N(\vv)\|_\HH \lec \|\vv\|_\HH^2, \pq \|\ck N(\vv)-\ck N(\vec w)\|_\HH \lec \|\vv-\vec w\|_\HH(\|\vv\|_\HH+\|\vec w\|_\HH).}
Then the difference $\diff{\vec v}:=\vec v^1-\vec v^0$ satisfies 
\EQ{ \label{diff eq}
 \CAS{(i\p_t+L_-)\diff{v_s} - Q \re\diff{v_w} -QM_{(\te)}\diff\vga = v_s^0\re\diff{v_w}+\diff{v_s}\re v_w^1 + i\diff{\ck N_1}, \\
 (i\p_t+\aD)\diff{v_w} - 2 \aD Q \re\diff{v_s} =\re\aD(\bar v_s^0+\bar v_s^1)\diff{v_s},} }
where $\diff{\ck N_1}:=\ck N_1(\vv^1)-\ck N_1(\vv^0)$. 
Using the above linear and bilinear estimates, Lemmas \ref{lem:lin}--\ref{lem:bil}, we obtain on any interval $[t_0,t_1]\subset I$ of length $|t_1-t_0|\le 1$ 
where both $\vec u^j$ stay around $\cQ_1$, 
\EQ{
 \pt\|\diff{v_s}\|_{\fX(t_0,t_1)} \lec \|\diff{v_s}(t_0)\|_{H^1_x} + \|Q^2\diff{v_s}+QM_{(\te)}\diff\vga\|_{L^2_t(L^{3/2} \cap H^1)_x}
  \prQQQ+(|I|^{1/6}\|Q\|_{H^2_x} +\| v_s^0\|_\fX) \|\diff v_w\|_\fY 
 \prQQQ+ \|\diff{ v_s}\|_\fX\|v_w^1\|_\fY 
  +\|\diff v_w(t_0)\|_{L^2_x}\|Q\|_{H^1_x} + \|\diff\vv\|_{\fZ_0} \|\vv^\circ\|_{\fZ_0},
 \pr\|\diff v_w\|_{\fY(t_0,t_1)} \lec \|\diff v_w(t_0)\|_{L^2_x} 
 + (\| v_s^\circ\|_\fX+|I|^{1/2}\|Q\|_{H^2_x})\|\diff{ v_s}\|_\fX 
  \prQQQ+ \|Q\|_{H^1_x}\|\diff v_s(t_0)\|_{H^1_x},}
where $\vec v^\circ:=(\vec v^0,\vec v^1)$, and the second term on the right is easily bounded by 
\EQ{
 \pt \|Q^2\diff{v_s}+QM_{(\te)}\diff\vga\|_{L^2_t(L^{3/2} \cap H^1)_x}
 \lec |I|^{1/2}(\|Q\|_{H^2_x}^2+\|Q\|_{H^2_x})\|\diff{\vv}\|_{L^\I_t\HH_x}. }

Hence, using the smallness of $\|v_s^\circ\|_\fX$ as well, which follows from the same argument or the local wellposedness, 
we obtain for $|t_0-t_1|\ll 1$, 
\EQ{ \label{diff est unit}
 \|\diff\vv\|_{\fZ(t_0,t_1)} \lec \|\diff{\vec v(t_0)}\|_{\cH},}
and the same argument yields
\EQ{ \label{sum est unit}
 \|\vv^\circ\|_{\fZ(t_0,t_1)} \lec \|\vec v^\circ(t_0)\|_{\cH}.}
Both estimates are extended by iteration for $t_1\le t_0+1$ within $I$. 

The above estimates on unit time intervals are consistent with exponential growth. 
To get a better bound on the dispersive component, we use the linearized energy.  
For the spectral decomposition $\vv^j=\la_+^j\g^++\la_-^j\g^-+\vga^j$ with $\la_\pm^j:=\om(\vv^j,\mp\g^\mp)$ 
as before, we have the same equations and expansion for the difference
\EQ{ \label{L diff exp}
 \pt (\p_t\mp\mu)\diff\la_\pm = \diff{\ti N_\pm := \ti N_\pm(\vec v^1)} - \ti N_\pm(\vec v^0), 
 \pr \LR{\cL\diff{\vec v}|\diff{\vec v}}=-2\mu(\diff{\la_+})(\diff{\la_-})+\LR{\cL\diff{\vga}|\diff{\vga}},}
Hence the eigenmode part is estimated by 
\EQ{ \label{diff la est}
 \pt |(\p_t\mp\mu)\diff{\la_\pm}| \lec |\diff{\ti N_\pm}| \lec \|\diff{\vec v}\|_\cH \|\vec v^\circ\|_\HH, 
 \pr |\p_t(\diff{\la_+})(\diff{\la_-})| \lec |\diff{\ti N_+}\diff{\la_-}|+|\diff{\la_+}\diff{\ti N_-}|
 \lec \|\diff{\vec v}\|_\cH^2 \|\vec v^\circ\|_\HH. }
For the main part of energy, using the equation \eqref{diff eq}, we have
\EQ{
 \p_t\tf12\LR{\cL\diff{\vec v}|\diff{\vec v}}
 \pt=\LR{\cL{\diff{\vec v}}|(\p_t-J\cL)\diff{\vec v}} 
 \pr=\LR{iL_-\diff v_s-iQ\re\diff{v_w}|v^0_s\re\diff v_w+\diff{v_s}\re v_w^1}
  \prq+\LR{\tf12i\diff v_w-iQ\re\diff{v_s}|\re\aD[(\bar v_s^0+\bar v_s^1)\diff v_s]},}
where the modulation term with $\dot\te-1$ does not contribute. 
Those terms except $\De$ in $L_-$ are easily bounded on $[t_0,t_1]\subset I$ with $|t_0-t_1|\le 1$ by 
\EQ{
 \pt \|\LR{i(1-Q^2)\diff v_s-iQ\re\diff{v_w}|v^0_s\re\diff v_w+\diff v_s \re v^1_w}\|_{L^1_t(t_0,t_1)}
 \prq\lec \|\diff v_s\|_{\fX_0}(\|v_s^0\|_{\fX_0}\|\diff v_w\|_{\fY_0}+\|\diff v_s\|_{\fX_0}\|v_w^1\|_{\fY_0}),
 \prQ+ \|\diff{v_w}\|_{\fY_0}(\|v_s^0\|_{\fX_1}\|\diff{v_w}\|_{\fY_0}+\|\diff{v_s}\|_{\fX_1}\|v_w^1\|_{\fY_0}),}
where $\fX_1\subset L^2_tL^\I_x$ is used, and similarly, using $\fX_0\cap \fX_1\subset L^4_tH^{1,3}_x$, 
\EQ{
 \pt \|\LR{i\diff\ga_w-2iQ\re\diff{\ga_s}|\re\aD[(\bar v_s^0+\bar v_s^1)\diff v_s]}\|_{L^1_t(t_0,t_1)} 
 \lec \|\diff\vv\|_{\fZ_0}\|v_s^\circ\|_\fX\|\diff v_s\|_{\fX_1},}
as well as the interactions excepting $XL$ (see \cite[Section 2]{GN} for the notation $LH+HH+LL$) 
\EQ{
 \pt \|\LR{i\De\diff{v_s}|(\re\diff v_w,v^0_s)_{LH+HH+LL}+(\re v^1_w,\diff v_s)_{LH+HH+LL}}\|_{L^1_t(t_0,t_1)} 
 \prq\lec \|\diff{v_s}\|_{\fX_1}(\|\diff v_w\|_{\fY_0}\|v^0_s\|_{\fX_1}+\|v^1_w\|_{\fY_0}\|\diff v_s\|_{\fX_1}).}
To the $XL$ part, we apply the normal form or partial integration 
using the equation of $\Om$, namely $-D^2\Om(g,h)+\Om(\aD g,h)+\Om(g,D^2h)=(gh)_{XL}$, 
\EQ{
 \pn\int_{t_0}^{t_1} \LR{f|(gh)_{XL}}dt
 \pt=[\LR{f|\Om(g,h)}_{t_0}^{t_1}]
 \pn-\int_{t_0}^{t_1} \LR{(i\p_t-\De)f|\Om(g,h)}dt 
 \prq+ \int_{t_0}^{t_1} \LR{f|\Om((i\p_t+\aD)g,h)+\Om(g,(i\p_t-\De)h)}dt,}
for $(g,h)=(\diff{v_w},v_s^0)$ and $(v_w^1,\diff{v_s})$, 
as well as its $(i\p_t-\aD)$ version in place of $(i\p_t+\aD)$ for $(\diff{\bar v_w},v_s^0)$ and $(\bar v_w^1,\diff{v_s})$ 
(with the corresponding modification of $\Om$). Then we obtain
\EQ{
 \pt\int_{t_0}^{t_1}  \LR{i\De\diff{v_s}|(\re\diff v_w,v^0_s)_{XL}+(\re v^1_w,\diff v_s)_{XL}}dt
 \pr\lec \|\diff{v_s}\|_{\fX_0}(\|\diff v_w\|_{\fY_0}\|v^0_s\|_{\fX_0}+\|v^1_w\|_{\fY_0}\|\diff v_s\|_{\fX_0})
 \prq +\|\diff{ v_s}\|_{\fX_+}(\|\diff v_w\|_{\fY_0}\|v^0_s\|_{\fX_1}+\|v^1_w\|_{\fY_0}\|\diff{v_s}\|_{\fX_1})
 \prq +\|\diff{v_s}\|_{\fX_1}(\|\diff v_w\|_{\fY_+}\|v^0_s\|_{\fX_0}+\|v^1_w\|_{\fY_+}\|\diff{v_s}\|_{\fX_0})
 \prq +\|\diff{v_s}\|_{\fX_1}(\|\diff v_w\|_{\fY_0}\| v^0_s\|_{\fX_+}+\|v^1_w\|_{\fY_0}\|\diff{ v_s}\|_{\fX_+}).}
Thus we obtain for $|t_0-t_1|\le 1$ in $I$
\EQ{
 [\LR{\cL\diff{\vga}|\diff{\vga}}]_{t_0}^{t_1}
 \pt\lec \|\diff{\vec v}\|_{\fZ(I)}^2 \|{\vec v}^\circ\|_{\fZ(I)}
 \pn\lec \|\diff{\vec v(t_0)}\|_\HH^2\|\vec v^\circ(t_0)\|_\HH,}
where \eqref{diff est unit}--\eqref{sum est unit} are used. 
Iterating this estimate on consecutive unit intervals, we obtain for all $[t_0,t_1]\subset I$ of any length
\EQ{  \label{diff ga est}
 \|\diff{\vga(t_1)}\|_\HH^2 \lec \|\diff{\vga(t_0)}\|_\HH^2 + \|\diff{\vec v(t_0)}\|_\HH^2\|\vec v^\circ(t_0)\|_\HH 
 + \int_{t_0}^{t_1} \|\diff{\vec v(t)}\|_\HH^2\|\vec v^\circ(t)\|_\HH dt,}
while the stable mode is estimated easily by \eqref{diff la est}
\EQ{ \label{diff la- est}
 |\diff{\la_-(t_1)}| \lec e^{\mu(t_0-t_1)}|\diff{\la_-(t_0)}| + \int_{t_0}^{t_1} e^{\mu(t_0-t)}\|\diff{\vec v(t)}\|_\HH\|\vec v^\circ(t)\|_\HH dt. }
They may be combined as the center-stable part, keeping the same estimate as $\vga$. Thus we have obtained 
\EQ{ \label{ests on diff}
 \pt|(\p_t-\mu)\diff{\la_+}| \le C_0\|\vv^\circ\|_\HH \|\diff{\vv}\|_\HH,
 \pr\|\diff{\vv(t_1)}\|_\HH^2 \le C_0\|\diff{\vv(t_0)}\|_\HH^2 + C_0|\diff{\la_+(t)}|^2 
   + C_0\int_{t_0}^{t_1} \|\vec v^\circ(t)\|_\HH \|\diff{\vv(t)}\|_\HH^2 dt,}
for some constant $C_0>0$ and for any interval $[t_0,t_1]\subset I$.

\subsection{Center-stable and center-unstable manifolds} \label{ss:cs mfd}
In order to identify all trapped solutions as a center-stable manifold of $\cQ_1$ of codimension $1$ 
(and as a center-unstable manifold for $t\to-\I$ by time inversion), 
we need the dominance of the unstable mode for the difference. More precisely, we have
\begin{lem} \label{lem:diff Lip}
There exist a constant $C_D>1$ with the following properties. 
Let $I=[T_0,T_1)$ be an interval, and $\vec u^j=\Psi_{\te_j}(\vv^j)$ with $j=0,1$ 
be two solutions of \eqref{eq:Zak2} on $I$ satisfying $\om(\vv^j,\vec Q_{(\la)})=0$ and 
\EQ{ \label{la+ Lip est}
 \de:=\max_{j=0,1}\sup_{t\in I}d_Q(\vec u^j(t))<\de_E, \pq C_D\de<\mu, 
 \pq 0<C_D\de\|\diff{\vv(T_0)}\|_\HH \le |\diff{\la_+(T_0)}|, }
where $\la_\pm^j$ are defined by \eqref{def la+-} as well as $\vga^j$ for $j=0,1$. 
Then we have for all $t\in I$
\EQ{ \label{dyn diff}
 \pt |\tf{\p_t\diff{\la_+}}{\diff{\la_+}}-\mu|< \tf12\mu, 
 \pq |\diff{\la_-}| \lec |\diff{\la_-(T_0)}| + \de(\|\diff{\vv(T_0)}\|_\HH+|\diff{\la_+}|), 
 \pr \|\diff{\vga}\|_\HH \lec \|\diff{\vga(T_0)}\|_\HH + \de^{1/2}(\|\diff{\vv(T_0)}\|_\HH(t-T_0+1)^{1/2}+|\diff{\la_+}|).  }
\end{lem}
Note that $\diff\vv(T_0)=0$ is the trivial case 
where $\uu^1=(e^{-i\diff\te}u^0_s,u^0_w)$ is just a phase shift of $\uu^0$. 
The above conclusion includes the fact that $\sign\diff\la_+$ is preserved on $I$. 
\begin{proof}
It is by the bootstrapping or continuity argument, to prove estimates in the form
\EQ{ \label{bootstrap}
 \pt \|\diff{\vv(t)}\|_\HH < C_1(a_0+|\diff{\la_+(t)}|) < \tf{\mu}{2\de C_0}|\diff{\la_+(t)}|,}
for some constant $C_1>1$ to be chosen, where 
$a_0:=\|\diff{\vv(T_0)}\|_\HH$, $\de:=\|\vv^\circ\|_{L^\I_t\HH(I)}$, 
and $C_0$ is the constant in \eqref{ests on diff}. 
The first inequality is obvious at $t=T_0$. We assume that the second inequality holds at $t=T_0$, which requires
$2C_0C_1\de < \mu$. 

Assuming \eqref{bootstrap} breaks down at some $t\in I$, we may define the first breakdown time
\EQ{
 T':=\inf\{t\in I\mid \eqref{bootstrap}\text{ fails at $t$}\} \in (T_0,T_1).}
For brevity, we denote $f(t):=|\diff{\la_+(t)}|$ and $g(t):=\|\diff{\vv(t)}\|_\HH$. 
Then \eqref{ests on diff} and \eqref{bootstrap} imply on $I=[T_0,T')$
\EQ{
  f' \ge \mu f - C_0\de g > \mu f - C_0 C_1 \de(a_0+f) > \tf{\mu}{2} f.}
Hence $f$ is exponentially increasing and so the second inequality of \eqref{bootstrap} can not fail at $t=T'$. 
Concerning the first inequality, \eqref{ests on diff} yield
\EQ{
 g^2 \le C_0[a_0^2+f^2+\int_{t_0}^t \de g(s)^2 ds], }
where the integral is bounded using \eqref{bootstrap} and $f'>\tf{\mu}{2}f$
\EQ{
 \int_{t_0}^t \de g(s)^2 ds \pt< \int_{t_0}^t \de C_1^2(a_0+f(s))^2ds \pn< \int_{t_0}^t \tf{C_1}{C_0}(a_0+f(s))\tf{\mu}{2}f(s)ds
 \pr<  \int_{t_0}^t \tf{C_1}{C_0}(a_0+f(s))f'(s)ds \le \tf{C_1}{2C_0}(a_0+f(t))^2. }
Thus we obtain $g^2 < (C_0+\tf{C_1}{2})(a_0+f)^2$, 
which is better than the first inequality of \eqref{bootstrap}, provided that  
$C_0+\tf{C_1}{2} < C_1^2$. 
For example, $C_1=\sqrt{C_0+1}+1$ is enough. 
Then $T'<T_1$ is impossible, which means \eqref{bootstrap} for all $t\in I$, 
and thus we obtain the desired estimates, using \eqref{diff ga est}--\eqref{diff la- est}. 
\end{proof}

If both $\vec u^0,\vec u^1$ are trapped by $\cQ_1$ in $\cH^\e$, then the above exponential growth of $\diff{\la_+}$ would 
be a contradiction. Hence the contraposition implies 
\EQ{ \label{Lip est mfd}
 \de:=\sup_{t\ge 0}\max_{j=0,1} d_Q(\vec u^j)<\min(\de_E,\mu/C_D) \implies |\diff\la_+(t)| \lec \de \|\diff{\vec v(t)}\|_\HH. }
Since the one-pass theorem implies that trapped solutions with initial data close to $\cQ_1$ 
must stay close forever, the set of trapped solutions in a neighborhood of $\cQ_1$ must sit in a Lipschitz graph of $\la_+$. 

To make it more precise, we introduce the coordinate in the stable and unstable directions for the difference from $\vec Q_{1,0}=\vec Q$ 
\EQ{
 \pt \cH_\g^\perp:=\{\fy\in\cH \mid \om(\fy,\g^\pm)=0\}, 
 \pr \Pi:\R^2\oplus\cH^\perp_\g\to\cH,
 \pq \Pi(a_+,a_-,\fy):=a_+\g^++a_-\g^-+\fy.}
$\Pi$ is a linear homeomorphism with the inverse map
\EQ{
 \psi=\Pi(a_+,a_-,\fy) \iff a_\pm=\om(\psi,\mp\g^\mp), 
 \ \fy=\psi-a_+\g^+-a_-\g^-.}
Let $\cB_\de(Q)$ be the open neighborhood of the ground states for this metric, defined by
\EQ{
 \pt \cB_{\de}(0):=\{\fy\in\cH \mid \|\fy\|_{\cH}<\de\}, 
 \pq \cB_{\de}^{\te}(Q):=\Psi_{\te}(\cB_{\de}(0)), \pq \cB_{\de}(Q):=\Cu_{\te\in\R}\cB_{\de}^{\te}(Q). }
The small balls in the orthogonal subspace are denoted by 
\EQ{
 \cB_\de^\perp(0):=\cB_\de(0)\cap\cH^\perp_\g = \{\fy\in\cH \mid \|\fy\|_\HH<\de,\ \om(\fy,\g^\pm)=0\}.}
We have an atlas $\{(\cB_{\de}^{\te}(Q),\Phi_{\te}^{-1})\}_{\te\in\R}$ of $\cB_\de(Q)$ with the affine coordinates
\EQ{
 \pt\Phi_{\te}:\R^2\oplus\cH_\g^\perp\to\cH, 
 \pq \Phi_{\te}:=\Psi_{\te}\circ\Pi.}
The modulation Lemma \ref{lem:orth} may be restated as follows. For $\te\in\R$, let 
\EQ{
 \hat\Psi_\te \pt:= \Psi_\te^{-1}\circ \Psi_0 = \Psi_{\te+c}^{-1}\circ\Psi_c \pq (\forall c\in\R),
  \pq \hat\Psi_\te(v) = v + (e^{i\te}-1)(Q+v_s,0).}
Then for any $\fy\in\B_{\de_\star}(0)$, there is a unique $\te\in\R$ satisfying 
$|\te|\lec\de_\star$ and $\hat\Psi_\te(\fy)\in\cH_\perp$. 
In other words, $\uu:=\Psi_0(\fy)=\Psi_\te(\vv)$ is given by Lemma \ref{lem:orth}. 
Define $\Te:\B_{\de_\star}(0)\to\R$ by $\Te(\fy)=\te$. Then it is analytic with the leading linear term
\EQ{
  \pt\Te(\fy)=M(Q)^{-1}\om(\fy,\vec Q_{(\la)}) + O(\|\fy\|_\HH^2),
  \pr\diff\Te(\fy)=M(Q)^{-1}\om(\diff\fy,\vec Q_{(\la)}) + O(\|\fy\|_\HH\|\diff\fy\|_\HH),}
for any $\fy,\fy_0,\fy_1\in\B_{\de_\star}(0)$. 
Then the modulation maps to satisfy the orthogonality $\om(\cdot,\vec Q_{(\la)})=0$ 
are defined as 
\EQ{
 \hat\Psi(\vv):=\hat\Psi_{\Te(\vv)}(\vv), \pq \hat\Phi:=\Pi^{-1}\circ\hat\Psi\circ\Pi,}
then $\hat\Phi:\ti\B_\de\to\ti\B_{C\de}$ for sufficiently small $\de>0$, where
\EQ{
 \ti\B_\de := (-\de,\de)^2\times\B_\de^\perp(0).} 
Henceforth for any $(a,\fy)\in\ti\B_\de$ we use the following notation 
\EQ{
 \pt a=:(a_+,a_-)\in(-\de,\de)^2, \pq (\fy^+,\fy^-):=(\fy,\bar\fy)\in(\B_\de^\perp(0))^2, 
 \pr [a,\fy]:=\Phi_0(a,\fy)=a_+\g^++a_-\g^-+\fy,
  \pq \{a,\fy\}:=\{\Phi_\te(a,\fy)\mid \te\in\R\}.}

Thanks to the one-pass Theorem \ref{thm:onepass}, it suffices to investigate those staying close for all $t\ge 0$. 
More precisely, all trapped solutions in $\cH^\e$ as $t\to\I$ with initial data in $\cB_\de(Q)$ 
must stay in $\cB_{C\de}(Q)$ for all $t\ge 0$ with some constant $C>1$ and for any $0<\e\le\e_*$ and any $\e\ll\de\ll R_*$. 
Notice the difference from Lemma \ref{lem:trap}, which ensures settlement into the small neighborhood only for sufficiently large time. 
Thus we obtain the following classification of those solutions. 
\begin{prop} \label{prop:mfd}
There exist $\de_M\in(0,\de_X]$ and a unique function $m:(-\de_M,\de_M)\times\B_{\de_M}^\perp(0)\to(-\de_M,\de_M)$ 
such that for any $\de\in(0,\de_M]$ and the solution $\vec u$ of \eqref{eq:Zak2} with initial data $\vec u(0)\in\{a,\fy\}$ 
for any $(a,\fy)\in\ti\B_\de^\perp(0)$, we have the trichotomy
\begin{enumerate}
\item $\vec u$ scatters as $t\to\I$ iff $a_+<m(a_-,\fy)$. 
\item $\vec u$ is trapped by $\cQ_1$ as $t\to\I$ iff $a_+=m(a_-,\fy)$. 
\item $\vec u$ grows up in $t>0$ iff $a_+>m(a_-,\fy)$. 
\end{enumerate}
Moreover, it is uniformly Lipschitz continuous with $m(0,0)=0$ 
\EQ{
 |m(b^0,\fy^0)-m(b^1,\fy^1)| \lec \de(|b^0-b^1|+\|\fy^0-\fy^1\|_\HH) \pq (b^j\in(-\de,\de),\ \fy^j\in\B^\perp_\de(0)).}
The same classification for $t<0$ holds with the threshold $a_-=m(a_+,\bar\fy)$. 
Moreover, there is a unique function $\ti m:\B^\perp_\de(0)\to(-\de,\de)^2$ such that 
for any $(a,\fy)\in \ti\B_\de$, 
the above $\vec u$ is trapped both as $t\to\pm\I$ iff $a=\ti m(\fy)$. We have $\ti m(0)=0$,
\EQ{
 |\ti m(\fy^0)-\ti m(\fy^1)| \lec \de\|\fy^0-\fy^1\|_\HH \pq (\fy^j\in\B^\perp_\de(0)). }
$m,\ti m$ are weakly continuous. 
Finally, there exists $0<\e\lec\de$ such that for any solution $\vec u$ in $\cH^\e$ trapped by $\cQ_1$ as $t\to\I$,  
there exist $T>0$, $a\in(-\de,\de)$ and $\fy\in\B^\perp_\de(0)$ such that 
$\vec u(T)\in \{m(a,\fy),a,\fy\}$. 
If $\vec u\in\HH^\e$ is trapped both as $t\to\pm\I$, then there exists a unique $\fy\in\B^\perp_\de(0)$ 
such that $\vec u(0)\in\{\ti m(\fy),\fy\}$. 
\end{prop}
\begin{proof}
The latter half of statement is a consequence once $m$ is constructed. 
Indeed, the classification for $t<0$ follows by time inversion $\vec u(t)\mapsto \ba{\vec u(-t)}$. 
Then $\ti m$ is obtained by the implicit function theorem for the simultaneous equations $a_\pm=m(a_\mp,\fy^\pm)$ 
with the small Lipschitz constant $O(\de)$. 
If $\vec u\in\cH^\e$ is trapped and $0<\e\ll\de$ is small, then Lemma \ref{lem:trap} implies that $\vec u(t)\in \B_{C\e}(Q)$ 
for large $t>0$, so that we may apply the classification by $m$.  
If $\vec u$ is trapped both as $t\to\pm\I$, then Theorem \ref{thm:onepass} implies that $\vec u(t)\in \B_{C\e}(Q)$ 
for all $t\in\R$, so that the classification may be applied at any time. 
The weak continuity follows from the weak closedness of trapping in $\B_\de(Q)$, 
as the solutions are uniformly bounded by $d_Q(\uu(t))\lec\de$ for $t\ge 0$ if $a_+=m(a_-,\fy)$, 
and for $t\in\R$ if $a=\ti m(\fy)$. 
The uniform bound is inherited by the weak limit of solutions, which means trapping of the latter, so the equation of $m$ or $\ti m$ should also be inherited. 

It remains to construct $m$, for which the ejection Lemma \ref{lem:deltaX} and the difference estimate Lemma \ref{lem:diff Lip} suffice. 
In other words, the variational, virial or dispersive estimates in Sections \ref{s:var}--\ref{s:away} are not needed. 
We now fix $\de>0$ small enough such that $\Pi(\ti\B_\de)\subset B_{\de_\star}(0)$ and 
\EQ{
 \Psi_0\circ\hat\Psi(\ti\B_\de)\subset\{\uu\in\cH^\e \mid d_Q(u)<\de_E\}, }
for some $\e<\de_X/10$ (which may be much larger than those used in the previous sections). For any $(a,\fy)\in\ti\B_\de$, let 
\EQ{
 \vec w:=\Pi(a,\fy), \pq (\la,\ga):=\hat\Phi(a,\fy)}
and let $\uu$ be the solution of \eqref{eq:Zak2} with the initial data 
\EQ{
 \uu(0)=\Phi_0(\la,\ga) = \Phi_{-\te}(a,\fy),}
where $\te:=\Te\circ\Pi(a,\fy)$. Let $d_u(t):=d_Q(\uu(t))$ and 
let $T_X>0$ be the ejection time given by Lemma \ref{lem:deltaX}, namely
\EQ{
 T_X:=\inf\{t>0 \mid d_u(t)\ge\de_X\}.}
If $T_X<\I$, then there is $T_0\in(0,T_X)$ such that $d_u(T_0)\in(2\e,\de_X/2)$ 
and $\p_td_u(T_0)\ge 0$, so that we may apply Lemma \ref{lem:deltaX} at $t=T_0$, 
which implies 
\EQ{
 d_u(t) \sim e^{\mu(t-T_X)}\de_X \sim |\la_+(t)|}
for $T_0\le t\le T_X$. 
The local wellposedness implies stability of the above estimate for small initial perturbation, 
then we may apply the lemma also to the perturbed solutions, 
which implies the stability of ejection including $\sign\la_+(T_X)$. 
In other words, the initial data sets for the ejection to the positive and negative sides, namely
\EQ{
 I_\pm := \{(a,\fy) \in \ti\B_\de \mid T_X<\I,\ \pm\la_+(T_X)>0\} }
are both open, while $T_X=\I$ in the remainder 
\EQ{
 I_0:= \ti\B_\de \setminus(I_+\cup I_-). }
If $(a^j,\fy^j)\in I_0$ for $j=0,1$, then Lemma \ref{lem:diff Lip} or \eqref{Lip est mfd} implies 
\EQ{
 |\diff\la_+| \lec \de\|\diff{\hat\Psi(\ww)}\|_\HH \lec \de\|\diff{\vec w}\|_\HH.}
Since the leading term of the modulation map does not contribute to $\la$, we have 
\EQ{
 \diff\la=\diff a + O(\|\ww\|_\HH\|\diff\ww\|_\HH),}
and hence the smallness of $\de>0$ implies 
\EQ{ \label{Lip a+}
 |\diff a_+| \lec \de(|\diff a_-|+\|\diff\fy\|_\HH).}
In particular, $(0,0)\in I_0$ (the ground state) implies
\EQ{ \label{small a+ mfd}
 |a_+| \lec \de(|a_-|+ \|\fy\|_\HH)}
for any $(a,\fy)\in I_0$. 
On the other hand, if $|a_-|\ll \pm a_+ \sim\de$, then $\la_+\sim\la_1\sim\pm\de$, 
so that we may apply the ejection Lemma \ref{lem:deltaX} at $t=0$ to deduce that $(a,\fy)\in I_\pm$. 
Since $I_\pm$ are disjoint open and the complement $I_0$ is restricted to \eqref{small a+ mfd}, 
by connectedness, we deduce that 
if $\pm a_+\gg\de\|\fy\|_\HH$ then $(a,\fy)\in I_\pm$, 
and for any fixed $(a_-,\fy)$, there is at least one $a_+\in\R$ with $|a_+|\lec\de(|a_-|+\|\fy\|_\HH)$ 
and $(a,\fy)\in I_0$. Then \eqref{Lip a+} implies its uniqueness $a_+=:m(a_-,\fy)$ and Lipschitz continuity. 
Then by the connectedness again, we have $(a,\fy)\in I_\pm$ for $\pm(a_+-m(a_0,\fy))>0$. 
\end{proof}

We are now ready to prove the main results of this paper. 
\begin{proof}[Proof of Theorem \ref{thm:resol}]
Let $0<\e\le\e_*$ for the constant $\e_*$ given by Theorem \ref{thm:onepass} and Proposition \ref{prop:sc}, and 
\EQ{
 \M^\e_{cs}(\cQ_1) \pt:= \{\vec u(0)\in\cH^\e \mid \exists T<\forall t<T^*,\ d_Q(\vec u(t))<\de_X\}.}
Henceforth $\vec u$ is any solution in $\HH$ of \eqref{eq:Zak2} with the maximal existence interval $(T_*,T^*)$. 
First, Lemma \ref{lem:trap} implies 
\EQ{
 \M^\e_{cs}(\cQ_1) = \{\vec u(0)\in\cH^\e \mid \varlimsup_{t\to\I}d_Q(\vec u(t))<2\e\}, }
and those solutions in this set are in the case (1) of the theorem, as $d_Q(\uu)\sim d_0(\uu)$. 

Next, Theorem \ref{thm:onepass} implies that for any solution $\uu$ in $\HH^\e$, 
either there exists $T_X\in(0,T^*)$ such that $d_u(t):=d_Q(\uu(t))>R_*$ for all $t\in(T_X,T^*)$, 
or $T^*=\I$ with $d_u(t)<2\e$ for large $t$. 
The latter is in $\M^\e_{cs}(\cQ_1)$ or the case (1). 
Otherwise, Proposition \ref{prop:sc} implies the scattering (case (2)) if $\fS(\uu(T_X))=+1$ 
or the growup (case (3)) if $\fS(\uu(T_X))=-1$. 
\end{proof}

\begin{proof}[Proof of Theorem \ref{thm:9set}]
All the 9 cases of $3\times 3$ behavior in $t>0$ and in $t<0$ have non-empty sets in the classification by Proposition \ref{prop:mfd}, divided by the graphs $a_\pm=m(a_\mp,\fy^\pm)$.  
Choosing $\de\gec\e$ large enough within $\de\le\de_M$ (or taking $\e$ smaller if necessary) 
ensures that for all the trapped solutions $\uu\in\M^\e_{cs}(\cQ_1)$, 
$d_Q(\uu(t))<3\e$ for large $t>0$ implies $\uu(t)\in\Phi_\te(\ti\B_\de)$ for some $\te\in\R$, and hence 
\EQ{ \label{fom T+}
 \M^\e_{cs}(\cQ_1)=\T_+=\Cu\bigl\{\cZ_-\{m(a,\fy),a,\fy\}\cap\cH^\e \bigm| a\in(-\de,\de),\ \fy\in\B^\perp_\de(0) \bigr\}, }
where $\cZ_-$ denotes the maximal backward orbit for \eqref{eq:Zak2}, namely 
\EQ{
 \cZ_-\vec u(0):=\{\vec u(t)\mid T_*<t\le 0\}. } 
Since the backward flow is a local diffeomorphism of $\HH$, 
we deduce that $\M^\e_{cs}(\cQ_1)$ is a Lipschitz manifold of codimension $1$.

Similarly, trapping for $t\to-\I$ makes the center-unstable manifold of codimension $1$:
\EQ{ 
 \M^\e_{cu}(\cQ_1) \pt:= \{\vec u(0)\in\cH^\e \mid \exists T>\forall t>T_*,\ d_Q(\vec u(t))<\de_X\}
  \pr= \{\vec u(0)\in\cH^\e \mid \varlimsup_{t\to-\I}d_Q(\vec u(t))<2\e\} = \T_-
  \pr=\Cu\bigl\{\cZ_+\{a,m(a,\bar\fy),\fy\}\cap\cH^\e \bigm| a\in(-\de,\de),\ \fy\in\B^\perp_\de(0) \bigr\},}
where $\cZ_+\vec u(0):=\{\vec u(t)\mid 0\le t<T^*\}$ denotes the maximal forward orbit.

Moreover, their intersection makes the center manifold of $\cQ_1$ with codimension $2$ as the graph of $\ti m$ in the small neighborhood of $\cQ_1$: 
\EQ{ \label{fom cmfd}
 \M^\e_c(\cQ_1) \pt:= \{\vec u(0) \in\cH^\e \mid \varlimsup_{|t|\to\I}d_Q(\vec u(t))<\de_X\}
  \pr= \{\vec u(0) \in \cH^\e \mid \sup_{t\in\R} d_Q(\vec u(t))<2\e\} = \T_+\cap\T_-
  \pr= \bigl\{ \{\ti m(\fy),\fy\} \bigm| \fy\in\B^\perp_\de(0) \bigr\} \cap \cH^\e, }
where the second expression with $\sup_{t\in\R}$ rather than $\limsup_{|t|\to\I}$ follows from the one-pass Theorem \ref{thm:onepass}, which allows the last expression without any extension by the flow. 

We also obtain the characterization of the other parts of $\T_+$, namely 
\EQ{ \label{fom T- SB+}
  \pt \T_+\cap\cS_-=\Cu \bigl\{\cZ_-\{m(a,\fy),a,\fy\}\cap\cH^\e \bigm| -\de<a<\ti m_2(\fy),\ \fy\in\B^\perp_\de(0) \bigr\}, 
  \pr \T_+\cap\B_-=\Cu \bigl\{\cZ_-\{m(a,\fy),a,\fy\}\cap\cH^\e \bigm| \ti m_2(\fy)<a<\de,\ \fy\in\B^\perp_\de(0) \bigr\},}
and similar expressions for $\T_-\cap\cS_+$ and $\T_-\cap\B_+$, where the roles of $a_\pm$ are flipped. 

Moreover, the one-pass theorem and the sign functional imply that all solutions in $\cS_+\cap\B_-$ and $\cS_-\cap\B_+$ must go through the small neighborhood of $\cQ_1$. Let $\cZ:=\cZ_+\cZ_-=\cZ_-\cZ_+$ be the maximal orbit in both forward and backward time directions. Then 
\EQ{ \label{fom S+B-}
 \cS_+\cap\B_-=\Cu \bigl\{\cZ\{a,\fy\}\cap\cH^\e \bigm| (a,\fy)\in\ti\B_\de,
  \ \pm(a_\pm-m(a_\mp,\fy^\pm))<0 \bigr\}, }
and a similar formula holds for $\cS_-\cap\B_+$, where the inequalities are reversed. Hence those sets are open. 
Moreover, all solutions $\uu$ in $\cS_+\cap\cS_-$, $\cS_+\cap\T_-$, $\cS_-\cap\T_+$ and $\T_+\cap\T_-$ 
must satisfy either $d_Q(\uu(t))<2\e$ or $\uu(t)\in\HH_{\de_S}$ with $\fS(\uu(t))=+1$ for all $t\in\R$. 
Hence, using Lemma \ref{lem:signS} for the latter, we obtain a uniform bound of them in $\HH$. 

The openness of $\cS_\pm$ follows in general and without the energy constraint, from the perturbation (stability) argument around $t=\pm\I$ of scattering solutions. The blow-up set is not generally open, but its openness under the energy constraint follows from the classification and the one-pass theorem. 
Since those getting close to $\cQ_1$ 
\EQ{
 \B_+^\de \pt:=\Cu \bigl\{\cZ\{a,\fy\}\cap\cH^\e \bigm| (a,\fy)\in\ti\B_\de,
  \ a_+>m(a_-,\fy) \bigr\}}
is open, the issue is about those staying away from $Q$. 
Suppose for contradiction that $\B_+$ is not open. Then there is a sequence of solutions $\uu_n$ such that $\uu_n(0)\in\cS_+\cup\T_+$ 
convergent to some $\uu(0)\in\B_+\setminus\B_+^\de$ in $\HH^\e$. 
After some time translation, we may assume that $d_Q(\uu(0))>R_*$ with $\fS(\uu(0))=-1$. 
Then it holds also for $\uu_n(0)$ of large $n$. 
Since $\uu_n$ is either scattering or trapped as $t\to\I$, the classification implies that $\uu_n$ must go through or stay in 
the small neighborhood of $\cQ_1$. In particular, there is a sequence $t_n>0$ such that $d_Q(\uu_n(t_n))<2\e$. 
Since $\uu(0)\not\in\B_+^\de$, we have $\inf_{t\ge 0}d_Q(\uu(t))\ge 3\e$, so $t_n\to\I$. 
The one-pass Theorem \ref{thm:onepass} (backward in time) yields a sequence $T_n\in(0,t_n)$ such that 
$d_Q(\uu_n(T_n))=R_*<d_Q(\uu_n(t))$ for $0\le t<T_n\to\I$. 
Then the same argument as for \eqref{vir in B} or \eqref{vir gup} yields 
\EQ{
 V^0_\chi(\ti u_n)(0)-V^0_\chi(\ti u_n)(T_n) \gec \sum_j \int_{I^n_j}(e^{\mu(t-t_j^n)}-C_*)d_u(t)dt + \int_{I_n'}\ka_V(\de_*)dt,}
for some disjoint intervals $\{I^n_j\}_j$ in $(0,T_n)$ and $I_n':=(0,T_n)\setminus \Cu_j I^n_j$. 
As $n\to\I$, the left side is bounded, since $\uu_n(0)\to\uu(0)$ converges in $\HH$ and $d_Q(\uu_n(T_n))=R_*$, 
while the right side is divergent as $T_n\to\I$, which is a contradiction. Hence $\B_+$ is open. 

Finally, we prove the (path-wise) connectedness. Taking $\e>0$ smaller if necessary, we may assume that $0<\e\ll\de\ll\e^{2/3}$. 
For those in $\Phi_\te(\ti\B_\de)$, we may restrict the problem to $\te=0$, because all the concerned subsets are gauge invariant. 

For any $[\ti m(\fy),\fy]\in\T_+\cap\T_-$ in the formula \eqref{fom cmfd}, 
consider the curve $\uu(b):=[\ti m(b\fy),b\fy]$ with $b:1\to 0$ connecting to $\uu(0)=\vec Q$. 
To check whether it is contained in \eqref{fom cmfd}, the only non-trivial point is the energy constraint, 
which reads in the expansion \eqref{action exp} 
\EQ{
 S_Z(\uu(b))-S_S(Q)=-\mu\la_+\la_-+\tf12\LR{\cL\ga|\ga}-C(\vv)<\e^2,}
with $\la:=\ti m(b\fy)$, $\ga:=b\fy$ and $\vv:=\Pi(\la,\ga)$. Since $\LR{\cL\ga|\ga}\sim\|\ga\|_\HH^2$ 
and $C(\vv)$ is cubic with $|C(\vv)|\lec\|\vv\|_\HH^3$, the small Lipschitz constant $O(\de)$ of $\ti m$ implies 
for $0<b^0<b^1<1$ with $b^1-b^0<b^0$, 
\EQ{
 \diff{\LR{\cL\ga|\ga}} \sim (b^0\diff b)\|\fy\|_\HH^2, 
 \pq |\diff(\la_+\la_-)|\lec \de^2b^0|\diff b|\|\fy\|_\HH^2,
 \pq |\diff C(\vv)| \lec |b^0|^2|\diff b|\|\fy\|_\HH^3, }
so $S_Z(\uu(b))$ is increasing and thus $\uu(b)\subset\T_+\cap\T_-$ for $0\le b\le 1$. 
Hence $\T_+\cap\T_-$ is connected. 

Next, let $[a,\fy]\in\cS_+\cap\B_-$ in the formula \eqref{fom S+B-}, and consider a curve in the form 
$\uu(t):=[\la(t),t\fy]$ with $t:1\to 0$, satisfying $\la(1)=a$ and 
\EQ{
 \dot\la_\pm = \CAS{\pm L\de(|\dot\la_\mp|+\|\fy\|_\HH) &(\mp\la_\pm < B\de^2)\\ 0 &(\mp\la_\pm\ge B\de^2),}}
where $L,B\sim 1$ are chosen such that $L\de$ is the Lipschitz constant of $m$ and $B\de^2=\|m\|_{L^\I}$ on $(-\de,\de)\times\B_\de^\perp(0)$. 
Then $|\dot\la_\pm|\lec \de\|\fy\|_\HH$. 
Moreover, as $t$ decreases from $1$, the conditions $\pm(\la_\pm-m(\la_\mp,\fy^\pm))<0$ are preserved along $\uu(t)$, 
as well as $\la_+,\la_-\in(-\de,\de)$, and
\EQ{ \label{ene dec deform}
 \p_t S_Z(\uu(t)) = t\LR{\cL\fy|\fy} + O(\de|\la|\|\fy\|_\HH + |\la|^2\|\fy\|_\HH + t^2\|\fy\|_\HH^3)}
implies $\p_t S_Z(\uu(t))>0$, as long as $\|t\fy\|_\HH$ is large enough compared with $\de|\la|\lec\de^2$. 
In other words, we may connect $[a,\fy]=\uu(1)$ to $\uu(t_0)$ within $\cS_+\cap\B_-$ for some $t_0\in(0,1)$ 
such that $\|t_0\fy\|_\HH\sim\de^2$. 
Thus we reduced the connectedness of $\cS_+\cap\B_-$ to the region $[a,\fy]\in\cS_+\cap\B_-$ with $\|\fy\|_\HH\lec\de^2$. 

The same reduction works for the other parts in the coordinate neighborhood. More precisely, let 
\EQ{
 \pt \cN_\s^\pm:= \{ [a,\fy] \mid (a,\fy)\in\ti\B_\de,\ \sign(a_\pm-m(a_\mp,\fy^\pm))=\s\},
 \pr \cN_{\s_+,\s_-}:=\cN_{\s_+}^+\cap\cN_{\s_-}^-, \pq \cN_{\s_+,\s_-}^\e:=\cN_{\s_+,\s_-}\cap\HH^\e,}
so that $\s_\pm=+,0,-$ correspond to $\B_\pm,\T_\pm,\cS_\pm$ respectively. 
The reduction of $\cN_{-,+}^\e\subset\cS_+\cap\B_-$ is as above, and  
the reduction of $\cN_{+,-}^\e\subset\cS_-\cap\B_+$ is similar with appropriate sign changes, as well as for $\cN_{\pm,\pm}^\e$. 
For $\T_+\cap\B_-\supset\cN_{0,+}^\e\ni[a,\fy]$, the reduction is along the curve $\uu(t):=[m(b(t),t\fy),b(t),t\fy]$ with $t:1\to 0$ satisfying $b(1)=a_-$ and 
\EQ{
 \CAS{b \ge B\de^2 \implies \dot b=0,\\ b<B\de^2 \implies \dot b = -L\de(L\de(|\dot b|+\|\fy\|_\HH)+\|\fy\|_\HH).}}
Then the curve remains in $\cN_{0,+}$. Since $|\dot b|\lec\de\|\fy\|_\HH$, the same estimate \eqref{ene dec deform} implies 
that $\uu(t)\in\HH^\e$ as long as $\|t\fy\|_\HH\gec\de^2$. 
Thus we may reduce $\cN_{0,+}^\e$ to the region $\|\fy\|_\HH\lec\de^2$, and similarly $\cN_{0,-}^\e,\cN_{\pm,0}^\e$. 

For any $[a,\fy]\in\cN_{0,+}^\e$ with $\|\fy\|_\HH\lec\de^2$, consider the curve 
$\uu(s):=[m(s,\fy),s,\fy]$ starting from $s=a_-$ decreasing to $s=\ti m_2(\fy)$. 
Then the curve is in $\cN_{0,+}$ for $a_-\ge s>\ti m_2(\fy)$, and also in $\HH^\e$ because of $|\la_+\la_-|\lec\de^3\ll\e^2$ and $\|\fy\|_\HH\lec\de^2$. 
Thus we reduce $\cN_{0,\pm}$ and $\cN_{\pm,0}$ further to arbitrarily small neighborhoods of $\cN_{0,0}^\e$. 

For any $[a,\fy]\in\cN_{-,+}^\e$ with $\|\fy\|_\HH\lec\de^2$, 
consider the curve $\uu(t):=[ta,\fy]$ with $t:1\to 0$. If $a_+a_-<0$, then $S_Z(\uu(t))$ is decreasing together with $t$, 
while if $a_+a_-\ge 0$, then $S_Z(\uu(t))-S_S(Q)\lec\de^3\ll\e^2$ for all $t\ge 0$. 
Hence the curve remains in $\HH^\e$, so it can get out of $\cN_{-,+}^\e$ only by reaching the boundary $a_\pm=m(a_\mp,\fy^\pm)$, 
namely $\cN_{0,+}^\e$, $\cN_{-,0}^\e$ or $\cN_{0,0}^\e$. 
If it reaches $\cN_{0,+}^\e$ at $t=t_0>0$, then as $t\to t_0+0$, we have $\uu(t)=[m(ta_-,\fy)-h,ta_-,\fy]$ 
with some $h(t)\to+0$. Then it may be connected to a point near $\cN_{0,0}^\e$ along the curve in $\cN_{0,+}^\e$ considered above, 
namely $[m(s,\fy)-h,s,\fy]$ with small $h>0$ from $s>t_0a_-$ decreasing to $s=\ti m_2(\fy)$. 
The case of $\cN_{-,0}^\e$ is similar. Therefore, if $\uu(t)$ gets out of $\cN_{-,+}^\e$ within $t\in[0,1]$ then it may be connected 
to arbitrarily small neighborhood of $\cN_{0,0}^\e$. 
Otherwise, it is connected to $\uu(0)=[0,\fy]$ within $\cN_{-,+}^\e$. 
In this case, consider another curve $\uu(b):=[b\ti m(\fy),\fy]$ with $b:0\to 1$, which stays in $\HH^\e$ as $|\ti m(\fy)|\lec\de\|\fy\|_\HH\lec\de^3$. 
If it gets out of $\cN_{-,+}^\e$ then the same argument works as above to connect it along $\cN_{0,+}$ or $\cN_{-,0}$ to $\cN_{0,0}$. 
Thus in any case, we may deform $\cN_{-,+}^\e$ to a small neighborhood of $\cN_{0,0}^\e$. 
The same argument applies to $\cN_{+,-}^\e$ and $\cN_{\pm,\pm}^\e$. 

Thus we have reduced the connectedness problem of $\cN_{\s_+,\s_-}^\e$ for any $\s_\pm\in\{+,0,-\}$ to 
the smaller neighborhood 
\EQ{
 \cN_*:=\{\Phi_\te(a,\fy) \mid a\in \ti m(\fy)+(-\e^2,\e^2)^2,\ \fy\in\B_{\e^{4/3}}^\perp(0)\} \subset \HH^\e,}
where the distinction among $\cN_{\s_+,\s_-}^\e$ is simply given by the signs of $a_\pm-m(a_\mp,\fy^\pm)$, 
so the (path-wise) connectedness of each $\cN_{\s_+,\s_-}\cap\cN_*$ is obvious from the small Lipschitz constant of $m$, 
as well as the one-sided sets $\cN_\s^\pm\cap\cN_*$. 
Then the flow extensions \eqref{fom T+}, \eqref{fom T- SB+} and \eqref{fom S+B-} transfer the connectedness  
to $\T_+$, $\T_-\cap\cS_+$, $\T_-\cap\B_+$ and $\cS_+\cap\B_-$, 
and similarly to the other combinations, except for $\cS_+\cap\cS_-$ and $\B_+\cap\B_-$, which are not covered by the flow extensions.  

It remains to consider $\cS_+\cap\cS_-$, $\B_+\cap\B_-$, $\cS_\pm$ and $\B_\pm$. 
Let $\uu_j\in\cS_+\cap\cS_-$ be distinct two points for $j=0,1$ and let $C(t)$ with $t:0\to 1$ be a curve connecting $\uu_j$ in $\HH^\e$. 
If it is not contained in $\cS_+\cap\cS_-$, then its boundary point must belong to 
$\cS_+\cap\T_-$, $\T_+\cap\cS_-$ or $\T_+\cap\T_-$, since $\cS_\pm$ and $\B_\pm$ are open. 
If it belongs to $\T_+$, then apply the forward flow to the parts of curve close to the boundary point and within $\cS_+\cap\cS_-$, 
and if it belongs to $\T_-$, then apply the backward flow. Then we can send the boundary point to the region $d_Q(\uu(t))<2\e$, 
and so the nearby points in $\cS_+\cap\cS_-$ as well by the local wellposedness. 
Then the latter points belong to $\cN_{-,-}^\e$, which is connected as shown above. Hence $\cS_+\cap\cS_-$ is also connected. 
The same reduction works from $\B_+\cap\B_-$ to $\cN_{+,+}^\e$, from $\cS_\pm$ to $\cN_-^\pm\cap\HH^\e$, and from $\B_\pm$ to $\cN_+^\pm\cap\HH^\e$. 
Hence all the subsets of classification are path-wise connected. 
\end{proof}

It is worth noting that the maximal extension by flow works as above, because we have the classification of global dynamics in the whole neighborhood of the local manifold as in Proposition \ref{prop:mfd}, thanks to the one-pass Theorem \ref{thm:onepass}. The local construction of manifolds works much more generally only with exponential dichotomy such as described in Lemma \ref{lem:deltaX}, but not the global dynamics of solutions off the manifold. Without the latter information, the flow extension might well destroy the manifold structure by (asymptotic) self-intersection. 
The one-pass theorem also yields the smallness of the center manifold $\M^\e_c(\cQ_1) = \M^\e_{cs}(\cQ_1)\cap\M^\e_{cu}(\cQ_1)$,  
in contrast to the unboundedness of $\M^\e_{cs}(\cQ_1)$ and $\M^\e_{cu}(\cQ_1)$ (as each of them contains grow-up in the other direction). 

\subsection{Stable and unstable manifolds} \label{ss:smfd}
Using the above difference estimate, we may identify the trapped solutions on the threshold energy as the stable manifold of $\cQ_1$, 
and similarly those for $t\to-\I$ as the unstable manifold. 
The proof is indeed easier than the center-stable manifold, but it may also be derived as a special case of the latter. 
We state the case of $t\to-\I$, namely the unstable manifold. The case of $t\to\I$ is similar with the time inversion. 
\begin{lem} \label{lem:unst mfd}
There exist unique solutions $\vec Q^\pm\in C((-\I,T^*_\pm);\cH)$ of \eqref{eq:Zak2} such that 
\begin{enumerate}
\item $d_Q(\vec Q^\pm(0))=\de_X$, $d_Q(\vec Q^\pm(t))\sim e^{\mu t}\de_X$ and increasing for $t\le 0$. 
\item  $\fy_\pm:=\vec Q^\pm(0)-\vec Q$ satisfies $\om(\fy_\pm,\vec Q_{(\la)})=0$ and $\pm\om(\fy_\pm,\g^-)\sim\de_X$. 
\end{enumerate}
Moreover, $\vec Q^+$ grows up in $t>0$, and $\vec Q^-$ scatters as $t\to\I$. 
Let $\vec Q^0(t):=\vec Q_{1,t}$ and $\vec Q^{\s,\te}:=(e^{-i\te}Q^\s_s,Q^\s_w)$. 
Then for every solution $\vec u\in C((-\I,T^*);\cH)$ of \eqref{eq:Zak2} satisfying $S_Z(\vec u)=S_Z(\vec Q)$ and $\limsup_{t\to-\I}d_Q(\vec u(t))<\de_X$, there exist $\te,T\in\R$ and $\s\in\{\pm,0\}$ such that $\vec u(t)=\vec Q^{\s,\te}(t-T)$ for all $t<T^*$. 
\end{lem}
The constant $\de_X$ may be replaced with any one smaller.
\begin{proof}
First we prove the uniqueness modulo $\te,T$. 
Lemma \ref{lem:trap} implies $d_u(t):=d_Q(\vec u(t))\to 0$ exponentially and monotonically as $t\to-\I$, 
and the case of $d_u(t)=0$ is obvious, so we may assume $d_u(T_X)=\de_X$ for some $T_X\in\R$, 
$d_u(t)\sim |\la_+(t)| \sim e^{\mu(t-T_X)}\de_X$ and increasing for $t\le T_X$. 
Moreover, Lemma \ref{lem:deltaX} applied from $t\to-\I$ implies 
\EQ{
 \la_+(t) \sim \pm\de_X e^{\mu(t-T_X)}, \pq |\la_-(t)|+\|\vga(t)\|_\HH \lec (e^{\mu(T_X-t)}\de_X)^2.} 
If $\vec u^0$ and $\vec u^1$ are both such solutions satisfying $\la_+^0(T_1)=\la_+^1(T_1)=\pm\de_1$ 
for some $0<\de_1\lec\de_X$ at some $T_1<T_X$, 
then on any interval $(T_0,T_1)$ we have from \eqref{diff la est} and \eqref{diff ga est}, 
ignoring the exponential decay for $\la_-$, 
\EQ{
 \pt |\diff{\la_+}| \lec \int_t^{T_1} e^{\mu(t-s)}\de_1e^{\mu(s-T_1)}\|\diff{\vv(s)}\|_\HH ds \lec \de_1\|\diff{\vv}\|_{L^\I_t\HH},
 \pr |\diff{\la_-}| \lec e^{\mu(T_0-T_1)}\de_1 + \int_{T_0}^t \de_1e^{\mu(s-T_1)}\|\diff{\vv(s)}\|_\HH ds
 \lec e^{\mu(T_0-T_1)}\de_1+\tf{\de_1}{\mu}\|\diff{\vv}\|_{L^\I_t\HH},
 \pr \|\diff{\ga}\|_\HH^2 \lec e^{2\mu(T_0-T_1)}\de_1^2 + \int_{T_0}^t \de_1e^{\mu(s-T_1)}\|\diff{\vv(s)}\|_\HH^2 ds
 \pn\lec e^{2\mu(T_0-T_1)}\de_1^2 + \tf{\de_1}{\mu} \|\diff{\vv}\|_{L^\I_t\HH}^2.}
Hence if $\de_1/\mu$ is small enough for the last term to be absorbed by the left side, then  
\EQ{
 \|\diff{\vv(s)}\|_{L^\I_t\HH(T_0,T_1)} \lec \de_1 e^{\mu(T_0-T_1)} \to 0 \pq (T_0\to-\I),}
which implies $\vec v^0(t)=\vec v^1(t)$ for $t<T_1$. 
Then the modulation equation \eqref{eq te} implies $\dot\te^0(t)=\dot\te^1(t)$ for $t<T_1$, 
hence $\te^0(t)=\te^1(t)+\te$ for some $\te\in\R$ and all $t\le T_1$. 
Since $|\la_+(t)|\to 0$ monotonically as $t\to\I$, 
all such solutions $\vec v^0,\vec v^1$ with the same $\sign\la_+$ coincide after a translation in $t$, 
as well as $\te^0,\te^1$ after a shift, or in other words, $\vec u^0$ and $\vec u^1$ coincide after the time translation and the phase shift. 
The conditions $d_Q(\vec Q^\pm(0))=\de_X$, $\om(\fy_\pm,\vec Q_{(\la)})=0$ and $\pm\om(\fy_\pm,\g^-)\sim\de_X$ 
fix the time shift, the phase shift and the sign, respectively for $\vec Q^\pm$. 
The fate of $\vec Q^\pm$ for $t>0$ follows from the classification after the ejection 
with $\fS(\vec Q^\pm(t))=\pm 1$ for $0\le t<T^*_\pm$. 

For the existence of $\vec Q^\pm$, let $\vec u^\pm_n$ be the solution of \eqref{eq:Zak2}
 with $\vec u^\pm_n(0)=[\pm\tf 1n,m(\pm\tf 1n,0),0]$ for $n\in\N$. 
Then Proposition \ref{prop:mfd} implies that for large $n$, $\vec u^\pm_n$ are trapped as $t\to-\I$, 
$\vec u^+_n$ grows up in $t>0$, and $\vec u^-_n$ scatters as $t\to\I$. 
Then Lemma \ref{lem:trap} 
implies existence of $T^\pm_n>0$ such that $d_Q(\vec u^\pm_n(T^\pm_n))=\de_X$, 
and $d_Q(\vec u^\pm_n(t))\lec\max(\de_X e^{\mu(t-T_n^\pm)},1/n)$ for all $t<T^\pm_n$. 
Let $\vec u^n_\pm(t):=\vec u^\pm_n(t+T^\pm_n)$. 
Then they are solutions of \eqref{eq:Zak2} with $S_Z(\vec u^n_\pm)=S_Z(\vec Q)+O(1/n^2)$ as $n\to\I$ 
and $d_Q(\vec u^n_\pm(t))\lec \max(\de_X e^{\mu t},1/n)$ for $t\le 0$. 
Hence there are subsequences weakly convergent in $\cH$ locally uniformly on $t\le 0$. 
Let $\vec u^\I_\pm$ be the weak limits. Note that in the coordinate $\vec u(t)=\Phi_{\te(t)}(\la_+(t),\la_-(t),\ga(t))$, 
$\te$ and $\la_\pm$ are strongly convergent by Ascoli-Arzela, while $\ga$ is weakly convergent. 
Hence $\vec u^\pm$ are weak solutions of \eqref{eq:Zak2} on $t\le 0$,  
satisfying $d_Q(\vec u^\I_\pm(t))\lec \de_Xe^{\mu t}$ for all $t\le 0$ 
and $\fS(\vec u^\I_\pm(0))=\pm 1$. 
The uniqueness of weak solutions \cite{MN-uniq} implies $\vec u^\pm \in C((-\I,T^\pm);\cH)$ for some maximal $T^+>0$ 
and $T^-=\I$. 
Thus we obtain $\vec Q^\pm$ after some shifts in time and phase. 
\end{proof} 
Hence we have the unstable manifold of $\cQ_1$ of 2 dimensions:
\EQ{
 \M_u(\cQ_1) \pt:=\{\vec u(0) \mid \lim_{t\to-\I}d_Q(\vec u(t))=0\} 
 \pr=\{\vec Q^{\s,\te}(t) \mid \te\in\R,\ t<T^*_\pm,\ \s\in\{\pm,0\}\} \subset \M^\e_{cu}(\cQ_1),}
consisting of all solutions converging to $\cQ_1$ as $t\to-\I$, and the stable manifold: 
\EQ{
 \M_s(\cQ_1) \pt:=\{\vec u(0) \mid \lim_{t\to\I}d_Q(\vec u(t))=0\} 
  =\ba{\M_u(\cQ_1)} \subset \M_{cs}(\cQ_1),}
and $\M_u(\cQ_1)\cap\M_s(\cQ_1)=\cQ_1$. Thus we conclude 
\begin{proof}[Proof of Theorem \ref{thm:thres}]
First consider the solutions $\uu$ with $S_Z(\uu)\le S_S(Q)$. 
If $\uu$ is trapped as $t\to\I$, then Lemma \ref{lem:trap} (with all $\e>0$) implies $d_Q(\uu(t))\to 0$ as $t\to\I$, 
namely $\uu(t)\in\M_s(\cQ_1)$. Similarly, the trapping as $t\to-\I$ is equivalent to $\uu(t)\in\M_u(\cQ_1)$. 
Adding those information to Theorem \ref{thm:9set} implies the classification into (1)-(7) of Theorem \ref{thm:thres} 
for $S_Z(\uu)\le S_S(Q)$. 
The scaling argument as in \eqref{scale extension} extends it to the region $E_Z(\uu)M(\uu)\le E_S(Q)M(Q)$, 
adding another dimension to the (un)stable manifolds. Indeed, its tangent vector at $\vec Q$ is $\vec Q_{(\la)}$, 
which is linearly independent of $\{\vec Q_{(\te)},\g^\pm\}$. 
\end{proof}
Note however that we cannot apply scaling to $\M_u(\cQ_1)$ or $\g^\pm$ to get the corresponding objects 
in $E_Z(\uu)M(\uu)\le E_S(Q)M(Q)$, as they are affected by the wave speed $\al$. 
The scaling invariance with a fixed $\al$ applies only to the ground states (or more generally standing waves).

\end{document}